\documentclass[12pt]{amsart}
\usepackage{amsmath,latexsym,amsfonts,amssymb,amsthm}

\usepackage{geometry}
\usepackage{hyperref}
\usepackage{mathrsfs}
\usepackage{tikz-cd}
\usepackage[alphabetic]{amsrefs}
\usepackage[all,cmtip]{xy}
\usepackage{bbm}

\usepackage{bm}

\newtheorem{prop}{Proposition}[section]
\newtheorem{thm}[prop]{Theorem}
\newtheorem{cor}[prop]{Corollary}

\newtheorem{lem}[prop]{Lemma}

\theoremstyle{definition}

\newtheorem{defn}[prop]{Definition}

\newtheorem{expl}[prop]{Example}
\newtheorem{rem}[prop]{\it Remark}

\newtheorem*{claim*}{Claim}

\newcommand{\bP}{\mathbb{P}}
\newcommand{\bC}{\mathbb{C}}
\newcommand{\bR}{\mathbb{R}}
\newcommand{\bA}{\mathbb{A}}
\newcommand{\bQ}{\mathbb{Q}}
\newcommand{\bZ}{\mathbb{Z}}

\newcommand{\bN}{\mathbb{N}}

\newcommand{\bG}{\mathbb{G}}
\newcommand{\bT}{\mathbb{T}}

\newcommand{\Rr}{\mathbb{R}}
\newcommand{\Qq}{\mathbb{Q}}

\newcommand{\oX}{\overline{X}}

\newcommand{\oD}{\overline{D}}

\newcommand{\tX}{\widetilde{X}}
\newcommand{\tY}{\widetilde{Y}}

\newcommand{\tD}{\widetilde{D}}
\newcommand{\tDelta}{\widetilde{\Delta}}

\newcommand{\cX}{\mathcal{X}}
\newcommand{\cY}{\mathcal{Y}}
\newcommand{\cZ}{\mathcal{Z}}
\newcommand{\cW}{\mathcal{W}}

\newcommand{\cO}{\mathcal{O}}
\newcommand{\cL}{\mathcal{L}}
\newcommand{\cI}{\mathcal{I}}
\newcommand{\cM}{\mathcal{M}}
\newcommand{\cF}{\mathcal{F}}
\newcommand{\cG}{\mathcal{G}}

\newcommand{\cE}{\mathcal{E}}
\newcommand{\cN}{\mathcal{N}}
\newcommand{\cS}{\mathcal{S}}
\newcommand{\cT}{\mathcal{T}}

\newcommand{\cR}{\mathcal{R}}
\newcommand{\cD}{\mathcal{D}}

\newcommand{\cB}{\mathcal{B}}

\newcommand{\tcD}{\widetilde{\mathcal{D}}}

\newcommand{\ocX}{\overline{\mathcal{X}}}
\newcommand{\ocZ}{\overline{\mathcal{Z}}}
\newcommand{\ocD}{\overline{\mathcal{D}}}
\newcommand{\ocL}{\overline{\mathcal{L}}}

\newcommand{\fa}{\mathfrak{a}}
\newcommand{\fb}{\mathfrak{b}}

\newcommand{\Center}{\mathrm{Center}}
\newcommand{\Spec}{\mathrm{Spec}}
\newcommand{\Proj}{\mathrm{Proj}}
\newcommand{\Supp}{\mathrm{Supp}}
\newcommand{\Hom}{\mathrm{Hom}}
\newcommand{\mult}{\mathrm{mult}}
\newcommand{\lct}{\mathrm{lct}}
\newcommand{\Ex}{\mathrm{Ex}}
\newcommand{\Pic}{\mathrm{Pic}}
\newcommand{\Aut}{\mathrm{Aut}}
\newcommand{\vol}{\mathrm{vol}}
\newcommand{\ord}{\mathrm{ord}}
\newcommand{\Val}{\mathrm{Val}}
\newcommand{\Nklt}{\mathrm{Nklt}}

\newcommand{\hvol}{\widehat{\rm vol}}
\newcommand{\hVol}{\widehat{\rm Vol}}

\newcommand{\wt}{\mathrm{wt}}
\newcommand{\Coef}{\mathrm{Coeff}}
\newcommand{\Diff}{\mathrm{Diff}}
\newcommand{\mld}{\mathrm{mld}}

\newcommand{\Cl}{\mathrm{Cl}}
\newcommand{\gr}{\mathrm{gr}}

\newcommand{\Sym}{\mathrm{Sym}}

\numberwithin{equation}{section}

\title{Boundedness in general type MMP}

\author{Jingjun Han}
\address{Shanghai Center for Mathematical Sciences, Fudan University, Jiangwan Campus, Shanghai, 200438, China}
\email{hanjingjun@fudan.edu.cn}

\author{Lu Qi}
\address{School of Mathematical Sciences, East China Normal University, Shanghai 200241, China}
\email{lqi@math.ecnu.edu.cn}

\author{Ziquan Zhuang}
\address{Department of Mathematics, Johns Hopkins University, Baltimore, MD 21218, USA}
\email{zzhuang@jhu.edu}

\date{}

\begin{document}

\begin{abstract}
We show that in any sequence of a general type MMP, the minimal log discrepancy of singularities takes at most finitely many values, and the fibers of all the extremal contractions and flips belong to a bounded family. A key ingredient in the proof is an analysis of the behavior of local volumes in the MMP. 
\end{abstract}

\maketitle

\tableofcontents

\emph{Throughout this paper, we work over $\mathbb{C}$, the field of complex numbers.} 

\section{Introduction}

The Minimal Model Program (MMP) aims to construct, for any projective variety, a birational model that is as simple as possible. One of the main open problems in the MMP is the termination of the program. Conjecturally, every sequence of MMP terminates after finitely many steps; in particular, most natural invariants of the variety are expected to be bounded in any given MMP sequence. The purpose of this article is to prove some boundedness statements of this kind, without assuming termination, for varieties of general type, or more generally, for any MMP that starts with a Kawamata log terminal (klt) pair $(X,\Delta)$ such that either $K_X+\Delta$ or $\Delta$ is big. % As applications, we prove some effective and explicit termination results in dimension at most five. 

\subsection{Boundedness}

Our first result concerns the boundedness of the index and minimal log discrepancy. Recall that the minimal log discrepancy (mld), denoted $\mld(x,X,\Delta)$, of a klt pair $(X,\Delta)$ at a point $x\in X$ is the smallest log discrepancy of prime divisors over $x\in X$. It is one of the most important invariants in the study of the MMP, partly due to its monotonicity under the MMP.

\begin{thm}[see Theorem \ref{thm:index+mld in MMP}] \label{main:index+mld}
Let $(X,\Delta)$ be a projective $\bQ$-factorial klt pair such that either $K_X+\Delta$ or $\Delta$ is big. Then there exist a positive integer $r$ and a finite set $S\subseteq \bR_+$, depending only on the pair $(X,\Delta)$, such that for any sequence of steps of a $(K_X+\Delta)$-MMP $(X,\Delta)\dashrightarrow (Y,\Delta_Y)$,
\begin{enumerate}
    \item the Cartier index of any Weil divisor on $Y$ is at most $r$, and
    \item for any (not necessarily closed) point $y\in Y$, we have $\mld(y,Y,\Delta_Y)\in S$.
\end{enumerate}
\end{thm}
Shokurov \cite{Sho-mld-conj} has shown that if the minimal log discrepancy of singularities that appear in the MMP satisfies both lower semi-continuity (LSC) and ascending chain condition (ACC), then the MMP must terminate after finitely many steps. Theorem \ref{main:index+mld} confirms the ACC conjecture for minimal log discrepancies in general type MMPs, thus verifies this part of Shokurov's approach. For some past works on the ACC conjecture in special cases, see for example \cites{Ale93,Amb06,Kaw14,Jia21,HLL22,NS22,HLS24,Kaw24,HL25}. 

\medskip

Our second result shows that when running the MMP on varieties of general type (or when the boundary is big), the fibers of all possible extremal contractions and flips belong to a bounded family. This is certainly expected if the MMP terminates. Indeed, one can also show that a log canonical MMP terminates if and only if the pairs that appear in the MMP belong to a bounded family, see Theorem \ref{thm:bdd imply termination} in the appendix.

\begin{thm}[see Theorem \ref{thm:fiber bdd}]\label{main:fiber}
Let $(X,\Delta)$ be a projective $\bQ$-factorial klt pair such that either $K_X+\Delta$ or $\Delta$ is big. Then there exists a projective family $\cW\to \cB$ over a finite type base $\cB$, such that in any sequence of $(K_X+\Delta)$-MMP, every fiber of the extremal contractions or the flips is isomorphic to $\cW_b$ for some $b\in \cB$.
\end{thm}

Recall that a flip $X\dashrightarrow X^+$ in the MMP is a diagram $X\rightarrow Z \leftarrow X^+$ where $X\to Z$ is the flipping contraction and $X^+\to Z$ is (also called) the flip\footnote{In some literature it is called the flipped contraction.}. We refer to the fiber of $X^+\to Z$ as the fiber of the flip.

More generally, our method can be used to show that in the setting of Theorem \ref{main:fiber}, any fixed order infinitesimal neighborhood of the fibers belongs to a bounded family. This suggests that there are only finitely many ``analytic types'' of flips in the MMP, and most invariants of the pair can only change in finitely many different ways in the MMP.

\subsection{Local volume}
A key player in our proof is the local volume of klt singularities, an invariant introduced by Chi Li \cite{Li-normalized-volume} with motivation from questions in K-stability. A natural question is how the local volume behaves in the MMP. In general, this invariant does not seem to satisfy any monotonicity in the MMP, see Remark \ref{rem:local vol not monotone}. Nonetheless, one of our observations is that the local volumes are bounded away from zero in the MMP, at least when the boundary is big. By standard argument (as in \cite{BCHM}), any general type MMP can be turned into an MMP with a big boundary.

\begin{thm}[see Theorem \ref{thm:index+mld+vol bound}] \label{main:vol bdd}
Let $(X,\Delta)$ be a projective $\bQ$-factorial klt pair such that $\Delta$ is big. Then there exists some $\varepsilon>0$ such that for any sequence of steps of a $(K_X+\Delta)$-MMP $(X,\Delta)\dashrightarrow (X',\Delta')$ and any closed point $x'\in X'$, we have $\hvol(x',X')\ge \varepsilon$.
\end{thm}

Here $\hvol(x',X')$ denotes the local volume of the klt singularity $x'\in X'$ (see Section \ref{ss:local volume}). To prove this theorem, we further introduce the \emph{log canonical volume} (see Section \ref{ss:lc vol}) of the big $\mathbb R$-divisor $\Delta$, defined as the largest volume of a graded linear series $V_\bullet$ of $\Delta$ such that $\left(X,\frac{1}{m}\Gamma\right)$ is log canonical for any positive integer $m$ and any $\mathbb R$-divisor $\Gamma$ in the linear system $|V_m|$. It turns out that the log canonical volume is always non-decreasing in the MMP (Lemma \ref{lem:Vol MMP}), and the local volume of singularities is always bounded from below by the log canonical volume (Lemma \ref{lem:local vs global vol}), which implies Theorem \ref{main:vol bdd}. On the other hand, the local volume controls many other invariants of the singularity. In particular, by the local K-stability theory of klt singularities (see \cite{Z-SDC-survey} for a recent survey), the index and minimal log discrepancy of a klt singularity are bounded from above once the local volume is bounded away from zero, and Theorem \ref{main:index+mld} follows from this. 

Local volume also plays an important role in the proof of Theorem \ref{main:fiber}. A basic idea is that if $(X',\Delta')\to Z$ is an extremal contraction in a $(K_X+\Delta)$-MMP and $\Delta$ is big, then the local volume of the relative cone (see Section \ref{ss:rel cone} for the construction) of $(X',\Delta')$ over $Z$ is also bounded from below by the log canonical volume of $\Delta$ (up to a suitable multiplicative constant that only depends on the coefficients of $\Delta$), see Theorem \ref{thm:cone volume}. The local K-stability theory then implies that the relative cones are at least specially bounded (namely, they are bounded up to degenerations induced by special test configurations), though in this case a significant amount of work is still needed to deduce the boundedness in Theorem \ref{main:fiber} from the above special boundedness.

\subsection{Outline of the article} In Section \ref{sec: preliminary} we summarize some preliminary results that will be used in the later part of this paper. In Section \ref{sec: bdd inv mmp} we prove Theorems \ref{main:index+mld} and \ref{main:vol bdd}. In Section \ref{sec:fiber bdd} we prove Theorem \ref{main:fiber}. In the appendix, we prove that if the pairs that appear in a log canonical MMP are bounded, then the MMP terminates after finitely many steps.

%Jungkai A. Chen, 
\subsection*{Acknowledgment} 
The authors would like to thank Christopher D. Hacon, Jihao Liu, James M\textsuperscript{c}Kernan, Vyacheslav Shokurov, Chenyang Xu for helpful discussions and comments. JH is supported by NSFC for Excellent Young Scientists (\#12322102), and the National Key R\&D Program of China (\#2023YFA1010600). JH is a member of LMNS, Fudan University. LQ is partially supported by Shanghai Sailing program 24YF2709800. ZZ is partially supported by the NSF Grants DMS-2240926, DMS-2234736, a Sloan research fellowship, a Packard fellowship, and the Simons Collaboration on Moduli of Varieties.

\section{Preliminary}\label{sec: preliminary}

In this section, we collect some preliminary materials that will be used throughout the article. We suggest skimming through most of this section on first reading and returning to the relevant subsections for more details when certain definitions and statements are needed.

\subsection{Notation and conventions}

In this paper, varieties are assumed to be quasi-projective unless otherwise specified. We follow the standard terminology from \cites{KM98,Kol13}.

Given two $\bR$-divisors $\Delta$ and $\Delta'$ on a variety $X$, we write $\Delta'\ge \Delta$ if $\mult_D\Delta'\geq\mult_D\Delta$ for any prime divisor $D$ on $X$. 

A \emph{sub-pair} $(X,\Delta)$ consists of a normal variety $X$ together with an $\bR$-divisor $\Delta$ on $X$ (a priori, we do not require that $K_X+\Delta$ is $\bR$-Cartier). It is called a \emph{pair} if $\Delta\geq 0$. We denote by $\Coef(\Delta)$ the (finite) set of coefficients of $\Delta$ (in particular, we require that $0\in\Coef(\Delta)$). A \emph{log (sub-)pair} is a (sub-)pair $(X,\Delta)$ such that $K_X+\Delta$ is $\mathbb R$-Cartier. A \emph{singularity} $x\in (X,\Delta)$ consists of a pair $(X,\Delta)$ and a closed point $x\in X$. If there is no confusion, we shall assume that $X$ is affine and $x$ is contained in any irreducible component of $\Supp(\Delta)$. A sub-pair $(X,\Delta)$ is called \emph{log smooth} if $X$ is smooth and $\Supp(\Delta)$ is a simple normal crossing divisor. % A \emph{stratum} (resp. \emph{open stratum}) of a log smooth pair $(X,\Delta=\sum_{i\in I} b_i B_i)$, where $B_i$ are the irreducible components of $\Delta$, is defined to be a connected component of $B_J:=\cap_{j\in J} B_j$ (resp. $B^\circ_J:=\cap_{j\in J} B_j \setminus \cup_{j\not\in J} B_j$) for some $J\subseteq I$ (by convention, $B_\emptyset = X$ and $B^\circ_\emptyset = X\setminus \Supp(\Delta)$).

Let $(X,\Delta)$ be a log sub-pair. A proper birational morphism $\pi\colon (Y,\Delta_Y)\rightarrow (X,\Delta)$ from a sub-pair is called \emph{crepant} if $K_Y+\Delta_Y=\pi^*(K_X+\Delta)$, in which case we also call $(Y,\Delta_Y)$ the \emph{crepant pullback} of $(X,\Delta)$ (via $\pi$). 

An $\bR$-divisor $D$ on a variety $X$ is \emph{big} if $D=A+E$ for some ample $\bR$-divisor $A$ and some $E\ge 0$. Let $X\to Z$ be a morphism. We say $D$ is \emph{big over $Z$} if $D\sim_{\bR,Z} D'$ for some big $\bR$-divisor $D'$ on $X$. 

Let $f\colon X\to Z$ be a morphism and $z\in Z$ a point. We denote by $X_z:=f^{-1}(z)$ the (scheme-theoretic) fiber over $z$.

An $\bR$-divisor $D$ on a projective variety $X$ is \emph{pseudo-effective} if its numerical class is a limit of $\bR$-divisors $D_i\ge 0$ ($i=1,2,\dots$). Given a projective morphism $X\to Z$, we say $D$ is \emph{pseudo-effective over $Z$} if its restriction to the fiber $X_\eta$ is pseudo-effective, where $\eta\in Z$ is the image of the generic point of $X$.

A \emph{test configuration} of a pair $(X,\Delta)$ consists of the following data:
\begin{enumerate}
    \item A pair $(\cX,\cD)$ and a flat morphism $\pi\colon \cX \to \mathbb{A}^1$.    
    \item A $\mathbb{G}_m$-action on $(\cX,\cD)$ lifting the canonical multiplication action of $\mathbb{G}_m$ on $\mathbb{A}^1$.
    \item An isomorphism $(\cX_t,\cD_t) \cong (X,\Delta)$ for all $0\neq t\in \bA^1$. 
\end{enumerate}

A \emph{fibration} is a proper surjective morphism $f\colon X\to Z$ of normal varieties such that $f_* \cO_X=\cO_Z$. A \emph{fibration germ}, denoted by $f\colon X\to Z\ni z$, consists of a fibration $f\colon X\to Z$ where $Z$ is affine and a closed point $z\in Z$. If $X\to Z$ is a fibration and $(X,\Delta)$ is a pair, by abuse of notation we also call $(X,\Delta)\to Z$ a fibration.

Let $f: X\dashrightarrow Y$ be a birational map and let $D$ be an $\bR$-Cartier $\bR$-divisor on $Y$. The \emph{birational pullback} $f^*D$ is defined as $p_* q^* D$ where $p:Z\to X$, $q:Z\to Y$ is a resolution of the indeterminacy of $f$ (one can check that the definition is independent of the resolution).

Let $X$ be a variety. For any proper birational morphism $\pi\colon Y\to X$ from a normal variety and any prime divisor $E$ on $Y$, we call $E$ a \emph{(prime) divisor over} $X$ and denote by $c_X(E)$ % $\Center_X(E):=\pi(E)$ 
the \emph{center} of $E$ on $X$, which is the generic point of $\pi(E)$. A \emph{divisor over} $x\in X$ is a divisor over $X$ with center $x$.  % We also denote by $c_X(E)$ the generic point of $\Center_X(E)$ and sometimes also call it the center of $E$ if no confusion arises. We also call $E$ as a \emph{(prime) divisor over} $x\in X$ where $x$ is the generic point of $\pi(E)$. 
We denote by $\Ex(\pi)\subseteq Y$ the exceptional locus of $\pi$. The \emph{discrepancy} of $E$ with respect to a log sub-pair $(X,\Delta)$, is defined by
\[
a(E,X,\Delta):=\mult_E (K_Y-\pi^*(K_X+\Delta)).
\]
The \emph{log discrepancy} of $E$ with respect to $(X,\Delta)$ is defined by $A_{X,\Delta}(E):=1+a(E,X,\Delta)$.

The notation of \emph{klt} and \emph{log canonical} (sub-)pairs are defined as in \cite[Definitions 2.34]{KM98}. The \emph{non-klt locus} of a log pair $(X,\Delta)$ is defined as
\[
\Nklt(X,\Delta) := \{x\in X\,|\,(X,\Delta)\mbox{ is not klt near }x\}.
\]
% Let $(X,\Delta)$ be a log canonical pair. A subvariety $W\subseteq X$ is called a \emph{log canonical center} if either $W=X$, or $W$ is the center of some divisor $E$ over $X$ with $A_{X,\Delta}(E)=0$.

\subsection{Log discrepancy}

In this and the next two subsections, we recall some invariants of pairs and their basic properties.

\begin{defn}
For any log canonical sub-pair $(X,\Delta)$ and a (not necessarily closed) point $x\in X$, we define
\[
\mld(x,X,\Delta):=\inf\{A_{X,\Delta}(E)\mid E\text{ is a prime divisor over } x\in X\} 
\]
as the \emph{minimal log discrepancy} (\emph{mld}) of $(X,\Delta)$ at $x$, and define
\[
\mld(X,\Delta):=\inf\{A_{X,\Delta}(E)\mid E\text{ is a prime divisor over } X\} 
\]
as the \emph{total minimal log discrepancy} (\emph{total mld}) of $(X,\Delta)$.
\end{defn}

\begin{lem} \label{lem:index bdd imply mld discrete}
Let $n$ and $r$ be two positive integers and $S_0\subseteq [0,1]$ a finite set. Then there exists a discrete set $S\subseteq \bR_{\ge 0}$ depending only on $n,r$ and $S_0$ satisfying the following. Assume that $(X,\Delta)$ is a log canonical pair of dimension $n$ with $\Coef(\Delta)\subseteq S_0$ such that $rL$ is Cartier for any $\mathbb Q$-Cartier Weil divisor $L$ on $X$.
% \begin{enumerate}
%     \item $(X,\Delta)$ is a dlt pair of dimension $n$,
%     \item $\Coef(\Delta)\subseteq S_0$,
%     \item $rL$ is Cartier for any $\mathbb Q$-Cartier Weil divisor $L$ on $X$, and
%     \item $W$ is a log canonical center of $(X,\Delta)$ of dimension $\geq 1$ and $(W,\Delta_W)$ is the dlt pair induced by adjunction
%     $K_W+\Delta_W:=(K_X+\Delta)|_{W}.$
% \end{enumerate}
Then for any prime divisor $E$ over $X$, we have $A_{X,\Delta}(E)\in S$.
\end{lem}

\begin{proof}
By  \cite{HLS24}*{Theorem 5.6}, there exist real numbers $a_1,\dots,a_k\in (0,1]$ and a positive integer $N$ depending only on $n$ and $S_0$, and $\bQ$-divisors $\Delta_1,\dots,\Delta_k\ge 0$ such that $\sum_{i=1}^k a_i=1$, $\Delta=\sum_{i=1}^k a_i \Delta_i$, each pair $(X,\Delta_i)$ is log canonical, and $N\Delta_i$ is a Weil divisor for each $i$. For each $i$, the divisor $Nr(K_X+\Delta_i)$ is Cartier by assumption, hence $A_{X,\Delta_i}(E)\in \frac{1}{Nr}\bN$. Thus the set
$$S=\frac{1}{Nr}\left\{\sum_{i=1}^k m_ia_i\middle|\, m_i\in\mathbb N\right\}$$
satisfies the statement of the lemma.
\end{proof}

% \begin{proof}
% By  \cite{HLS24}*{Theorem 5.6}, there exist real numbers $a_1,\dots,a_k\in (0,1]$ and a positive integer $N$ depending only on $n$ and $S_0$, and $\bQ$-divisors $\Delta_1,\dots,\Delta_k\ge 0$ such that $\sum_{i=1}^k a_i=1$, $\Delta=\sum_{i=1}^k a_i \Delta_i$, each pair $(X,\Delta_i)$ is dlt, and $N\Delta_i$ is a Weil divisor for each $i$. In particular, $W$ is a log canonical center of each $(X,\Delta_i)$. Let $K_W+\Delta_{W,i}:=(K_X+\Delta_i)|_W$ for each $i$, then each $(W,\Delta_{W,i})$ is dlt and $\Delta_W=\sum_{i=1}^ka_i\Delta_{W,i}$. For each $i$, since $Nr(K_X+\Delta_i)$ is Cartier by assumption (3), we have $Nr(K_W+\Delta_{W,i})$ is also Cartier. Thus the set
% $$S=\frac{1}{Nr}\left\{\sum_{i=1}^k m_ia_i\middle|\, m_i\in\mathbb N\right\}$$
% satisfies the statement of the lemma.
% \end{proof}

\subsection{Local volume} \label{ss:local volume}
Let $X$ be a variety and $x\in X$ a closed point. A \emph{valuation} on the function field $\mathbb C(X)$ of a variety $X$, also called a \emph{valuation on $X$}, is an $\mathbb R$-valued function $v: \mathbb C(X)^*\rightarrow\mathbb R$ such that $v|_{\bC^*}=0$, $v(ab)=v(a)+v(b)$ for all $a,b\in \mathbb C(X)^*$, and $v(a+b)\ge \min\{v(a),v(b)\}$ if $a+b\neq 0$. We say that $v$ is a valuation over $x\in X$ if $v$ is centered at $x$. The set of valuations on $X$ is denoted by $\Val_{X}$. The set of valuations over $x\in X$ is denoted by $\Val_{X,x}$.

Given a klt sub-pair $(X,\Delta)$, the log discrepancy function
\[
A_{X,\Delta}\colon \Val_X\to \bR \cup\{+\infty\},
\]
is defined as in \cite{JM-val-ideal-seq} and \cite{BdFFU-log-discrepancy}*{Theorem 3.1}. For a divisorial valuation $v=\lambda\cdot \ord_E$ (where $\lambda>0$ and $E$ is a prime divisor over $X$), one has $A_{X,\Delta}(v)=\lambda\cdot A_{X,\Delta}(E)$. Since $(X,\Delta)$ is klt, $A_{X,\Delta}(v)>0$ for all $v\in \Val_{X,x}$. However, it is possible that $A_{X,\Delta}(v) = +\infty$ for some $v\in \Val_X$, see e.g. \cite{JM-val-ideal-seq}*{Remark 5.12}.

The volume of a valuation $v\in\Val_{X,x}$ is defined as
\[
\vol(v)=\vol_{X,x}(v)=\limsup_{m\to+\infty}\frac{\ell(\cO_{X,x}/\fa_m(v))}{m^n/n!},
\] 
where $n=\dim X$ and $\fa_m (v)$ denotes the valuation ideal, i.e.
\[
\fa_m (v):=\{f\in \cO_{X,x}\mid v(f)\ge m\}. 
\]
Thanks to the works of \cites{ELS03, LM-Okounkov-body, Cut13}, the above limsup is actually a limit.

\begin{defn} \label{defn:local volume}
Let $(X,\Delta)$ be a klt pair of dimensional $n$ and $x\in X$ a closed point. For any $v\in \Val_{X,x}$, we define the \emph{normalized volume} of $v$ as
\[
\hvol_{X,\Delta}(v):=
\begin{cases}
A_{X,\Delta}(v)^n\cdot\vol_{X,x}(v), & \mbox{if } A_{X,\Delta}(v)<+\infty, \\
+\infty, & \mbox{if } A_{X,\Delta}(v)=+\infty.
\end{cases}
\]
The \emph{local volume} of $(X,\Delta)$ at $x$ is defined as
\[
  \hvol(x,X,\Delta):=\inf_{v\in\Val_{X,x}} \hvol_{X,\Delta}(v).
\]
\end{defn}

By \cite{Li-normalized-volume}*{Theorem 1.2}, the local volume of a klt singularity is always positive. By \cite{Blu-minimizer-exist}, there always exists some valuation $v$ such that $\hvol(x,X,\Delta)=\hvol_{X,\Delta}(v)$. We also need the following properties of the local volume.

\begin{thm}[{\cite{XZ-minimizer-unique}*{Corollary 1.4}}] \label{thm:index bound via volume}
Let $(X,\Delta)$ be a klt pair of dimensional $n$, let $x\in X$ be a closed point, and let $D$ be a $\bQ$-Cartier Weil divisor on $X$. Then $rD$ is Cartier near $x$ for some positive integer $r\leq\frac{n^n}{\hvol(x,X,\Delta)}$.
\end{thm}

\begin{thm} \label{thm:mld bound via volume}
Fix $n\in\bN$ and $\varepsilon>0$. Then there exists some constant $A>0$ depending only on $n$ and $\varepsilon$, such that for any $n$-dimensional klt pair $(X,\Delta)$ satisfying $\hvol(x,X,\Delta)\ge \varepsilon$ for all closed points $x\in X$, we have
\[
\mld(x,X,\Delta)\le A
\]
for all (not necessarily) closed points $x\in X$.
\end{thm}

\begin{proof}
By \cite{Amb-mld}*{Proposition 2.3}, it suffices to prove that $\mld(x,X,\Delta)\le A$ for all closed point $x\in X$. We are done by \cite[Theorem 1.3]{XZ-local-bdd} or \cite{HLQ24}*{Corollary 2.7}.
\end{proof}

\subsection{Log canonical threshold}

For any log canonical sub-pair $(X,\Delta)$ and any $\mathbb R$-Cartier $\mathbb R$-divisor $D$ on $X$, we denote by
\[
\lct(X,\Delta;D):=\sup\{\,t\,|\, (X,\Delta+tD)\text{ is log canonical}\,\} \in \bR\cup \{+\infty\} 
\]
the \emph{log canonical threshold} of $D$ with respect to $(X,\Delta)$. For any $x\in X$, we also denote by
\[
\lct_x(X,\Delta;D):=\sup\{\,t\,|\,  (X,\Delta+tD)\text{ is log canonical at }x\,\}
\]
the log canonical threshold at $x$. The following result shows how the log canonical thresholds behave when we perturb the boundary. 

\begin{lem} \label{lem:lct inequality}
Let $(X,\Delta)$ be a klt pair and let $D,D'\ge 0$ be $\mathbb R$-Cartier $\mathbb R$-divisors on $X$. Then
\[
\lct(X,\Delta;D+D')\ge \left(\frac{1}{\lct(X,\Delta;D)}+\frac{1}{\lct(X,\Delta;D')}\right)^{-1}.
\]
If $(X,\Delta+D')$ is log canonical, then we also have
\[
\lct(X,\Delta+D';D)\ge \left(1-\frac{1}{\lct(X,\Delta;D')}\right)\cdot \lct(X,\Delta;D).
\]
\end{lem}

\begin{proof}
 Let $a:=\lct(X,\Delta;D)$ and $a':=\lct(X,\Delta;D')$. Then $(X,\Delta+aD)$ and $(X,\Delta+a'D')$ are log canonical, so the convex combination 
\[
\left(X,\Delta+\frac{a'}{a+a'}\cdot aD+\frac{a}{a+a'}\cdot a'D' \right) = \left(X,\Delta+\frac{aa'}{a+a'}(D+D')\right)
\]
is log canonical,  which gives the first inequality $\lct(X,\Delta;D+D')\ge \frac{aa'}{a+a'}=\left(\frac{1}{a}+\frac{1}{a'}\right)^{-1}$. If $a'\ge 1$, then the convex combination 
\[
\left(X,\Delta+\frac{a'-1}{a'}\cdot aD+\frac{1}{a'}\cdot a'D' \right) = \left(X,\Delta+D'+\frac{a'-1}{a'}\cdot aD\right)
\]
is log canonical, which implies the second inequality $\lct(X,\Delta+D';D)\ge\left(1-\frac{1}{a'}\right)a$.
\end{proof}

The next statement is a version of Izumi's inequality and should be well-known to experts. It will be useful in proving effective local volume estimates in Section \ref{sec:vol explicit bdd}. 

\begin{lem} \label{lem:Izumi}
Let $\varepsilon\in [0,1]$, let $(X,\Delta)$ be a log smooth sub-pair such that the coefficients of $\Delta$ are at most $1-\varepsilon$, and let $D\ge 0$ be an $\mathbb R$-divisor on $X$. Then for any closed point $x\in\Supp (D)$, we have
\[
\lct_x(X,\Delta;D)\ge \frac{\varepsilon}{\mult_x D}.
\]
\end{lem}

\begin{proof}
By removing the components of $\Delta$ with negative coefficients, we obtain an $\bR$-divisor $\Delta^{\ge 0}$ with $\lct_x(X,\Delta;D)\ge\lct_x(X,\Delta^{\ge 0};D)$. Possibly replacing $\Delta$ with $\Delta^{\ge 0}$, it suffices to prove the lemma when $\Delta\ge 0$. Possibly replacing $D$ with a multiple, we may also assume that $\mult_x D=1$. We need to show that $x\in (X,\Delta+\varepsilon D)$ is log canonical.

Let $\pi\colon\tX\to X$ be the blowup of $x$ and let $E$ be the exceptional divisor. Let $\tDelta$ (resp. $\tD$) be the strict transform of $\Delta$ (resp. $D$) on $\tilde X$. Since $\mult_x D=1$, we have $\tD|_E\sim_\bR H$ where $H$ is the hyperplane class on $E$ under the natural isomorphism $E\cong \bP^{\dim X-1}$. Note that $(E,\tDelta|_E)$ is a toric pair and the coefficients of $\tDelta|_E$ are at most $1-\varepsilon$, thus $\lct(E,\tDelta|_E;G)\ge \varepsilon$ for any torus invariant $\bR$-divisor $0\le G\sim_\bR H$ on $E$. By \cite{BJ-delta}*{Theorem F}, we deduce that $\lct(E,\tDelta|_E;G)\ge \varepsilon$ for any $\bR$-divisor $0\le G\sim_\bR H$ on $E$. In particular, the pair $(E,(\tDelta+\varepsilon \tD)|_E)$ is log canonical. By inversion of adjunction ({\it cf.} \cite{KM98}*{Theorem 5.50}), the pair $(\tX,\tDelta+\varepsilon \tD+E)$ is log canonical near $E$. Since the coefficients of $\Delta$ are at most $1-\varepsilon$, we have  $A_{X,\Delta}(E)\ge \varepsilon$, hence $A_{X,\Delta+\varepsilon D}(E)=A_{X,\Delta}(E)-\varepsilon\cdot  \mult_x D\ge 0$. Then $K_{\tX}+\tDelta+\varepsilon \tD+E\ge \pi^*(K_X+\Delta+\varepsilon D)$, so $ (X,\Delta+\varepsilon D)$ is log canonical at $x$.
\end{proof}

\subsection{Linear series} \label{ss:linear series}

For any $\bR$-divisor $D$ on a normal variety $X$, we set
\[
H^0(X,D):=\{0\neq s\in \mathbb C(X)\,|\,\mathrm{div}(s)+D\ge 0\}\cup \{0\}.
\]
Any finite dimensional subspace $V$ of $H^0(X,D)$ is called a \emph{linear series} of $D$. It defines a \emph{linear system}
\[
|V|:=\{D'=\mathrm{div}(s)+D\,|\,0\neq s\in V\}\cong \bP(V).
\]
A graded linear series $V_\bullet$ of $D$ is a sequence of finite dimensional subspaces $V_m\subseteq H^0(X,mD)$ (one for each $m\in\bN$) such that $V_m\cdot V_l\subseteq V_{m+l}$ for all $m,l\in \bN$. Its volume is defined as
\begin{equation} \label{eq:vol(V.)=limsup}
\vol(V_\bullet):=\limsup_{m\to+\infty} \frac{\dim(V_m)}{m^n/n!}
\end{equation}
where $n=\dim X$. We denote by $M(V_\bullet)\subseteq \bN^+$ the semigroup of those positive integers $m$ with $V_m\neq 0$. If $D$ is $\bR$-Cartier, and $\Delta\ge 0$ is an $\bR$-divisor such that $(X,\Delta)$ is klt, we also define the \emph{alpha invariant} of $V_\bullet$ as
\[
\alpha(X,\Delta;V_\bullet):=\inf\left\{\,\lct(X,\Delta;\Gamma)\,\middle|\ m\in M(V_\bullet), \,\Gamma\in \frac{1}{m}|V_m| \,\right\}
\]
and simply write $\alpha(V_\bullet)$ if the klt pair $(X,\Delta)$ is clear from the context. When $V_\bullet$ is the complete linear series of $D$, the corresponding alpha invariant is denoted by $\alpha(X,\Delta;D)$.

We say that a graded linear series $V_\bullet$ is \emph{eventually birational}, if the rational map $X\dashrightarrow \bP(V^\vee_m)$ induced by the linear system $|V_m|$ is birational onto its image for all sufficiently large $m\in M(V_\bullet)$. This condition is preserved under birational contractions; by the following lemma, it also ensures that the limsup in \eqref{eq:vol(V.)=limsup} is a limit. For this reason, in the rest of this paper we will exclusively consider graded linear series that are eventually birational.

\begin{lem} \label{lem:vol=lim}
Let $(X,\Delta)$ be a klt pair and let $V_\bullet$ be an eventually birational graded linear series such that $\alpha(V_\bullet)>0$. Then $\vol(V_\bullet)>0$ and we have
\begin{equation} \label{eq:vol=lim}
    \vol(V_\bullet)=\lim_{M(V_\bullet)\ni m\to+\infty} \frac{\dim(V_m)}{m^n/n!}.
\end{equation}
\end{lem}

\begin{proof}
We refer to \cite{LM-Okounkov-body} for the construction of Okounkov bodies and their general properties. It is clear that $\vol(V_\bullet)>0$ as $|V_m|$ defines a birational map for large enough $m$. To prove \eqref{eq:vol=lim}, let $\pi\colon X'\to X$ be the blowup of some smooth point $x\in X$ and consider the Okounkov body $\Delta(V_\bullet)$ constructed using a full flag of subvarieties of the form $X'\supseteq E\supseteq \cdots \supseteq \{\mathrm{pt\}}$, where $E$ is the exceptional divisor of the blowup. Since $\alpha(V_\bullet)>0$, there exists some constant $C>0$ such that $\mult_E(\pi^*\Gamma)=\mult_x (\Gamma)\le Cm$ for all $m\in \bN$ and $\Gamma\in |V_m|$. As $E\cong \bP^{n-1}$ is projective, by \cite{LM-Okounkov-body}*{Lemma 1.11} we know that $\Delta(V_\bullet)$ satisfies condition (A) from \cite{LM-Okounkov-body}*{Definition 2.4}. Since $V_\bullet$ is eventually birational, it also satisfies condition (B) from \cite{LM-Okounkov-body}*{Definition 2.5}. By \cite{LM-Okounkov-body}*{Theorem 2.13}, we see that 
\[
\lim_{M(V_\bullet) \ni m\to+\infty} \frac{\dim(V_m)}{m^n/n!}
\]
exists and equals $\vol(V_{\bullet})$.
\end{proof}

We also need the following Bertini type result.

\begin{lem} \label{lem:Bertini ample Q-div}
Let $(X,\Delta)$ be a klt pair and let $H$ be an ample $\bQ$-divisor on $X$. Then there exist some $\delta>0$ and some $m_0\in \bN$ such that for any positive integer $m\ge m_0$, there exists some $D\in |mH|$ with
\[
\lct(X,\Delta;D)\ge \delta.
\]
\end{lem}

\begin{proof}
Let $m_0$ be a sufficiently large and divisible integer such that $|m_0 H|$ is base point free and $|mH|\neq \emptyset$ for all $m\ge m_0$. For each $i=0,1,\dots,m_0-1$, choose some $D_i \in |(m_0+i)H|$ and let
\[
\delta := \min_{0\le i\le m_0-1}\{\lct(X,\Delta;D_i)\} >0.
\]
For any integer $m\ge m_0$ we may write $m=(\ell +1)m_0+i$ for some $\ell \in \bN$ and some $0\le i\le m_0-1$. By our choice of $m_0$ and $\delta$, we know that $|\ell m_0 H|$ is base point free and $(X,\Delta+\delta D_i)$ is log canonical, thus by Bertini's theorem, there exists some $D'\in |\ell m_0 H|$ such that $(X,\Delta+\delta D_i+D')$ is still log canonical. Note that $\delta\le 1$. Then $D=D_i+D'\in |mH|$ satisfies $\lct(X,\Delta;D)\ge \delta$.
\end{proof}

\subsection{Minimal Model Program (MMP)}

We refer to e.g. \cite[Paragraph 3.31]{KM98} for a description of the MMP. In this subsection, we recall some results on the MMP that will be used in this paper. First we recall the two basic operations in the MMP.

\begin{defn} 
Let $(X,\Delta)$ be a $\mathbb Q$-factorial log canonical pair and let $f \colon X \longrightarrow Z$ be a fibration. Then $f$ is called a $(K_X+\Delta)$-\emph{flipping contraction} if
\begin{enumerate}
    \item $f$ is small birational,
    \item the relative Picard number $\rho(X/Z) = 1$, and
    \item $-(K_X + \Delta)$ is $f$-ample.
\end{enumerate}
The $(K_X+\Delta)$-\emph{flip} $f^+ \colon X^+ \longrightarrow Z$ of the flipping contraction $f \colon X \longrightarrow Z$ is a small birational fibration $f^+ \colon X^+ \longrightarrow Z$ such that $K_{X^+} + \Delta^+$ is $f^+$-ample, where $\Delta^+$ is the strict transform of $\Delta$ on $X^+$. When the pair $(X,\Delta)$ is clear from the context, we shall refer to $f$ (resp. $f^+$) as simply the flipping contraction (resp. the flip). By abuse of terminology, we also call the composition $(f^+)^{-1}\circ f\colon X\dashrightarrow X^+$ a flip (however, the fibers of a flip only refer to the fibers of $f^+$). 
\end{defn}

\begin{defn}
Let $(X, \Delta)$ be a $\mathbb{Q}$-factorial log canonical pair and let $f \colon X \longrightarrow Z$ be a projective morphism of normal varieties. Then $f$ is called a \emph{divisorial contraction} if
\begin{enumerate}
    \item $f$ is birational with exceptional locus a divisor,
    \item the relative Picard number $\rho(X/Z) = 1$, and
    \item $-(K_X + \Delta)$ is $f$-ample.
\end{enumerate}
\end{defn}

For some of our proofs, it will be convenient to work with a broader class of birational maps as follows.

\begin{defn} \label{defn:MMP type contraction}
% both $K_X+\Delta$ and $K_{X'}+\Delta'$ are $\bR$-Cartier,
A birational map $\varphi\colon (X,\Delta)\dashrightarrow (X',\Delta')$ of log pairs is called an \emph{MMP type contraction} if $\varphi^{-1}$ has no exceptional divisor,  and there exist proper birational morphisms $f\colon Y\to X$, $g\colon Y\to X'$ such that
\[
f^*(K_X+\Delta)-g^*(K_{X'}+\Delta')\ge 0
\]
and $g=\varphi\circ f$. In particular, we have $\varphi_*\Delta\ge \Delta'$. 
\end{defn}

Observe that if $\varphi\colon (X,\Delta)\dashrightarrow (X',\Delta')$ is an MMP type contraction and $(X,\Delta)$ is klt (resp. log canonical), then so is $(X',\Delta')$. As the name suggests, if $\varphi\colon (X,\Delta)\dashrightarrow (X',\Delta')$ is obtained by running a $(K_X+\Delta)$-MMP, then $\varphi\colon (X,\Delta)\dashrightarrow (X',\Delta')$ is an MMP type contraction. If $\Delta'=\varphi_*\Delta$, then the birational contraction  
$\varphi\colon (X,\Delta)\dashrightarrow (X',\Delta')$ is an MMP type contraction if and only if $\varphi$ is $(K_X+\Delta)$-non-positive in the sense of \cite{BCHM}*{Definition 3.6.1}. We will also use the following simple observation.

\begin{lem} \label{lem:MMP type change boundary}
Let $\varphi\colon (X,\Delta+D)\dashrightarrow (X',\Delta'+D')$ be an MMP type contraction with $\Delta,\Delta',D\ge 0$ and $D'=\varphi_* D$. Let $\Gamma\ge 0$ be an $\bR$-divisor on $X$ such that $K_X+\Gamma\sim_\bR \mu (K_X+D)$ for some $\mu\ge 0$, and let $\Gamma'=\varphi_* \Gamma$. Then $\varphi$ also induces an MMP type contraction $\varphi: (X,\mu\Delta+\Gamma)\dashrightarrow (X',\mu\Delta'+\Gamma')$.
\end{lem}

\begin{proof}
Let $f: Y\rightarrow X, g: Y\rightarrow X'$ be two proper birational morphisms such that
$$E:=f^*(K_X+\Delta+D)-g^*(K_{X'}+\Delta'+D')\geq 0$$
and $g=\varphi\circ f$. So
\begin{equation}\label{equ: mu e}
    \mu E=f^*(\mu K_X+\mu\Delta+\mu D)-g^*(\mu K_{X'}+\mu\Delta'+\mu D')\geq 0.
\end{equation}
Since $K_X+\Gamma\sim_{\mathbb R}\mu(K_X+D)$ and $\varphi^{-1}$ does not contract any divisor, 
$$K_{X'}+\Gamma'-\mu(K_{X'}+D')=g_*f^*(K_X+\Gamma-\mu(K_X+D))\sim_{\mathbb R}0.$$
Thus by the negativity lemma, we have
\begin{equation}\label{equ: crepant gamma d}
    f^*(K_X+\Gamma-\mu(K_X+D))-g^*(K_{X'}+\Gamma'-\mu(K_{X'}+D'))=0,
\end{equation}
as the left hand side is $g$-exceptional and $\sim_\bR 0$.
Combining \eqref{equ: mu e} and \eqref{equ: crepant gamma d}, we get
$$ \mu E=f^*(K_X+\mu\Delta+\Gamma)-g^*(K_{X'}+\mu\Delta'+\Gamma')\geq 0.$$
Hence $\varphi: (X,\mu\Delta+\Gamma)\dashrightarrow (X',\mu\Delta'+\Gamma')$ is an MMP type contraction.
\end{proof}

We will frequently use the following observation, which is well known to the experts ({\it cf.} \cite[Lemma 3.7.5 and Proof of Theorem 1.2]{BCHM}), to reduce questions about general type MMPs to MMPs with a big boundary.

\begin{lem}\label{lem: gt klt sim to big boundary}
Let $(X,\Delta)$ be a klt pair and $X\rightarrow T$ a morphism. 
\begin{enumerate}
    \item If $K_X+\Delta$ is big over $T$, then there exists a klt pair $(X,\Delta')$, such that $K_X+\Delta'\sim_{\Rr,T} \mu(K_X+\Delta)$ for some $\mu>0$ and $\Delta'$ is big. % In particular, any $(K_X+\Delta)$-MMP is also a $(K_X+\Delta')$-MMP.
    \item  If $\Delta$ is big over $T$, then there exists a klt pair $(X,\Delta_0)$ such that $K_X+\Delta\sim_{\Rr,T}K_X+\Delta_0$ and $\Delta_0\geq A\geq 0$ for some $\mathbb R$-divisor $A$ that is ample over $T$. \qed
\end{enumerate} 
\end{lem}

\subsection{\texorpdfstring{$\bT$}{}-singularities}

In this subsection, we collect some results concerning singularities with a torus action. Such singularities naturally appear in the local K-stability theory of klt singularities. We only need the results in this subsection when we prove the boundedness of the fibers of extremal contractions and flips in Section \ref{sec:fiber bdd}.

Let $X=\Spec(R)$ be an affine normal variety and let $\rho\colon \bG_m\to \Aut(X)$ be a one-parameter subgroup. Then it induces a weight decomposition $R=\oplus_{m\in\bZ} R_m$ and a $\bG_m$-equivariant valuation $\wt_\rho$ on $X$ such that $\wt_\rho (f) = m$ when $0\neq f\in R_m$. % If moreover $\xi$ is a primitive element in the dual of the weight monoid $S^\vee = \Hom(S,\bN)\subseteq N_\bQ$, then $\wt_\xi = \ord_E$ where $E$ is the unique exceptional divisor of the blowup 
% \[
% X_\xi:=\Proj \oplus_{m\in \bN} \fa_{\xi,m} \to X
% \]
% where $\fa_{\xi,m}:=\sum_{\alpha:\langle\alpha,\xi\rangle\ge m} R_\alpha$ are the valuations ideals of $\wt_\xi$.

In general, let $\bT\subseteq \Aut(X)$ be an algebraic torus (i.e. $\bT\cong \bG_m^r$ for some $r>0$). Let $N:=\Hom(\bG_m,\bT)$ be the co-weight lattice, and $M:=\Hom(\bT,\bG_m)$ the weight lattice. Then we have a weight decomposition 
\[
R=\oplus_{\alpha\in M} R_\alpha.
\]
For any $f\in R$, we denote by $f_\alpha$ the component of weight $\alpha\in M$ in this weight decomposition. For any $0\neq \xi\in N_\bR$, we can define a valuation $\wt_\xi$ on $X$ by
\[
\wt_\xi (f):=\min\{\langle \xi, \alpha \rangle\mid \alpha\in M, f_\alpha\neq 0\}.
\]
A one-parameter subgroup $\rho\colon \bG_m\to \bT$ can be viewed as a primitive element of $N$. The corresponding valuation $\wt_\rho$ is exactly the one defined in the previous paragraph.

We say that the $\bT$-action on $X$ is \emph{good} if there is a unique closed point $x\in X$ that is in the orbit closure of any $\bT$-orbit. This is equivalent to the condition that $R_0=\mathbb C$. The \emph{weight monoid} $S$ consists of all $\alpha\in M$ such that $R_\alpha\neq 0$, and the \emph{weight cone} $\sigma\subseteq M_\bR$ is defined as the (convex, closed, polyhedral) cone generated by $S$. Note that the valuation $\wt_\xi$ has a center in $X$ if and only if $\xi$ is contained in the dual $\sigma^\vee$ of the weight cone. 

\begin{lem} \label{lem:log discrep linear T-variety}
Let $(X,\Delta)$ be a log pair and $\bT\subseteq \Aut(X,\Delta)$ a torus. Assume that $X$ is affine. Then there exists some $\alpha\in M_\bR$ such that $A_{X,\Delta}(\wt_\xi) = \langle \xi,\alpha\rangle$ for all $\xi\in\sigma^\vee$, where $\sigma\subseteq M_\bR$ is the weight cone.
\end{lem}

\begin{proof}
({\it cf.} \cite{LX-stability-higher-rank}*{Theorem 2.15(3)}.) By \cite[Theorems 3.1 and 3.4]{AH-T-variety}, there exist a proper birational $\bT$-equivariant morphism $\varphi\colon \tX\to X$ and a morphism $\pi\colon \tX\to Y=\tX/\!\!/\bT$ whose generic fiber $F$ is an affine toric variety. In particular, $X$ is $\bT$-equivariantly birational to $\bT\times Y$ (where $\bT$ acts trivially on $Y$). Let $K_{\tX}+\widetilde{\Delta}=\varphi^*(K_X+\Delta)$ be the crepant pullback. Then $\widetilde{\Delta}|_F$ is a torus invariant $\mathbb R$-divisor. Since $A_{X,\Delta}(\wt_\xi)=A_{\tX,\widetilde{\Delta}}(\wt_\xi)$ and the center of $\wt_\xi$ dominates $Y$, it suffices to compute the log discrepancy along the generic fiber $F$ of $\pi$. Thus we reduce to the toric case, where the desired statement is well known (see, for example, \cite{Amb-mld}*{Section 4}).
\end{proof}

Given a morphism $\cX\to \cB$, we call a torus $\bT$-action on $\cX$ that commutes with the trivial $\bT$-action on $\cB$ a \emph{fiberwise $\bT$-action}. We next observe that the weight cone and the log discrepancy function are locally constant for any $\bR$-Gorenstein family of klt singularities with a fiberwise torus action.

\begin{defn} \label{defn:R-Gor family}
We call $f\colon (\cX,\cD)\to\cB$ an \emph{$\bR$-Gorenstein family of pairs} if 
\begin{enumerate}
    \item $f\colon \cX\to \cB$ is an affine, flat morphism,
    \item for any closed point $b\in \cB$, the fiber $\cX_b$ is connected and normal, 
    \item $\cD\geq 0$ is an $\bR$-divisor on $\cX$ whose support does not contain any fiber $\cX_b$ of $f$,
    \item $\cX$ and $\cB$ are normal, and $K_{\cX/\cB}+\cD$ is $\bR$-Cartier.
\end{enumerate}
If $f\colon (\cX,\cD)\to \cB$ is an $\bR$-Gorenstein family of pairs, and $B\subseteq \cX$ is a section of $f$, We call $\cB\subseteq (\cX,\cD)\xrightarrow{f} \cB$ an \emph{$\bR$-Gorenstein family of singularities}.
% if
% \begin{enumerate}
%     \item $f\colon \cX\to \cB$ is an affine, flat morphism, and $\cB\subseteq \cX$ is a section of $f$,
%     \item for any closed point $b\in \cB$, the fiber $\cX_b$ is connected and normal, 
%     \item $\cD\geq 0$ is an $\bR$-divisor on $\cX$ whose support does not contain any fiber $\cX_b$ of $f$,
%     \item $\cX$ and $\cB$ are normal, and $K_{\cX/\cB}+\cD$ is $\bR$-Cartier.
% \end{enumerate}
If in addition $b\in (\cX_b,\cD_b)$ is a klt singularity for any closed point $b\in \cB$, we say it is an \emph{$\bR$-Gorenstein family of klt singularities}.
\end{defn}

\begin{lem} \label{lem:weight monoid constant}
Let $\cX\to \cB$ be an affine, surjective morphism of varieties with a fiberwise torus $\bT$-action. Then the weight monoid of the $\bT$-action on the fiber $\cX_b$ is locally constant as $b\in \cB$ varies. 
\end{lem}

\begin{proof}
We may assume that both $\cB$ and $\cX$ are affine, and $\cB$ is connected. Let $\cX=\Spec(R)$. It suffices to show that for any $\alpha\in M$ with $R_\alpha\neq 0$, we have $R_\alpha |_{\cX_b}\neq 0$ for all $b\in \cB$. Since we have a fiberwise $\bT$-action, the morphism $\cX\to \cB$ factors through $X_0:=\cX/\!\!/\bT=\Spec(R_0)$, thus we may replace $\cB$ by $X_0$ to show that the restriction $R_\alpha |_{\cX_b}$ is nonzero. Since $R$ is generated by finitely many homogeneous elements, it is not hard to see that each $R_\alpha$ is finitely generated over $R_0$. Alternatively, this also follows directly from \cite[Theorems 3.1 and 3.4]{AH-T-variety}: there exist a projective morphism $Y\to \Spec(R_0)$ and a semiample line bundle $L_\alpha$ (one for each $\alpha$ in the weight monoid of $R$) on $Y$ such that $R_\alpha = H^0(Y,L_\alpha)$. If $R_\alpha |_{X_b}=0$, then Nakayama's lemma implies that $R_\alpha=0$ in a neighborhood of $b\in X_0$, thus $R_\alpha$ is a torsion $R_0$-module. Since $R$ is an integral domain, we deduce that $R_\alpha=0$. This proves the lemma.
\end{proof}

\begin{lem} \label{lem:log discrep constant T-family}
Let $\cB\subseteq (\cX,\cD)\to \cB$ be an $\bR$-Gorenstein family of singularities with a fiberwise torus $\bT$-action and let $\sigma$ be the weight cone of the this $\bT$-action. Then for any $\xi\in \sigma^\vee$, the log discrepancy function $b\mapsto A_{\cX_b,\cD_b}(\wt_\xi)$ is locally constant on $\cB$.
\end{lem}

Note that the morphism $\cX\to \cB$ is surjective since it has a section. Thus by Lemma \ref{lem:weight monoid constant}, the weight monoid $S$ and the weight cone $\sigma$ of the $\bT$-action on the fiber $b\in (\cX_b,\cD_b)$ is locally constant in the family.

\begin{proof}
We may assume that $\cB$ is affine and connected. Since the log discrepancy function is linear on the weight cone by Lemma \ref{lem:log discrep linear T-variety}, it suffices to prove the statement when $\xi$ is a primitive element in the dual $S^\vee$ of the weight monoid. In this case $\xi$ corresponds to a one-parameter subgroup of $\bT$. After replacing $\bT$ by this one-parameter subgroup, we reduce to the case when $\bT=\bG_m$.

Let $\cX=\Spec(\cR)$. Then we have a weight decomposition $\cR=\oplus_{m\in\bN} \cR_m$. Let $\fa_m:=\sum_{k\ge m} \cR_k \subseteq \cR$ and $\cY:=\Proj \left(\bigoplus_{m\in \bN} \fa_m\right)$. Then $\fa_m|_{\cX_b}$ is the valuation ideal of $\wt_\xi$ (as a valuation on $\cX_b$), and we have a proper birational morphism $\pi\colon \cY\to \cX$ with a unique exceptional divisor $\cE\cong \Proj(\cR)$. Note that $\cE$ is $\bQ$-Cartier, we have $\cE_b \cong \Proj(\cR_b)$ (where $\cR_b:=\cR\otimes k(b)$) is a prime divisor on $\cY_b$ for all $b\in \cB$ (since $\cX_b\cong \Spec(\cR_b)$ is integral), and $\wt_\xi = \ord_{\cE_b}$ as valuations on $\cX_b$. Let $\cD_{\cY}$ be the strict transform of $\cD$ on $\cY$. Then 
\[
K_{\cY}+D_{\cY}+\cE \sim_\bR \pi^*(K_{\cX}+\cD)+a\cE
\]
where $a=A_{\cX,\cD}(\cE)$. As the family $(\cX,\cD)\to \cB$ is $\bR$-Gorenstein, by restricting to the fibers we get 
\[
K_{\cY_b}+D_{\cY,b}+\cE_b \sim_\bR \pi^*(K_{\cX_b}+\cD_b)+a\cE_b
\]
for all $b\in \cB$; in other words, $A_{\cX_b,\cD_b}(\wt_\xi)=A_{\cX_b,\cD_b}(\cE_b)=a$ is independent of $b\in \cB$.
\end{proof}

We also need the special boundedness of klt singularities when the local volume is bounded away from zero. To state the precise result, we need some definitions.

\begin{defn}
A \emph{test configuration} of a singularity $x\in (X,\Delta)$ consists of a test configuration $(\cX,\cD)\to \bA^1$ of the pair $(X,\Delta)$ and a $\bG_m$-equivariant section $\sigma\colon \bA^1\to \cX$ such that $\sigma(t)=x$ for $t\neq 0$ under the isomorphism $(\cX_t,\cD_t)\cong (X,\Delta)$. We still denote the test configuration by $(\cX,\cD)\to \bA^1$ as the section $\sigma$ is uniquely determined by this data. When $(X,\Delta)$ is klt, the test configuration $(\cX,\cD)\to \bA^1$ is called \emph{special} if $(\cX,\cX_0+\cD)$ is plt. In this case $(\cX_0,\cD_0)$ is also klt (by adjunction) and we call $\sigma(0)\in (\cX_0,\cD_0)$ a \emph{special degeneration} of $x\in (X,\Delta)$.
\end{defn}

\begin{expl} \label{exp:tc induced by divisor}
Let $x\in (X=\Spec(R),\Delta)$ be a singularity and let $E$ be a prime divisor over $x\in X$. Let
\[
\fa_m:=\{f\in R\,|\,\ord_E(f)\ge m\}.
\]
Assume that the graded algebra $\gr_E R:=\bigoplus_{m\in\bN} \fa_m/\fa_{m+1}$ is finitely generated. Then so is $\bigoplus_{m\in \bN} \fa_m$. By the Rees construction, we get a flat family $\pi\colon \cX=\Spec(\cR)\to \bA^1$ where 
\[
\cR=\bigoplus_{m\in \bZ} t^{-m} \fa_m
\]
is the (extended) Rees algebra and $t$ is the affine coordinate on $\bA^1$. The $\bZ$-grading on $\cR$ induces a $\bG_m$-action on the family and this gives a test configuration of the singularity $x\in X$. Note that $\cX_0\cong \Spec(\gr_E R)$, and $\gr_E R$ is an integral domain, hence $\cX_0$ is integral. 

Let $\cD$ be the closure of $\Delta\times (\bA^1\setminus \{0\})$ in $\cX$ under the isomorphism $\cX\setminus \cX_0\cong X\times (\bA^1\setminus \{0\})$. Then $(\cX,\cD)\to \bA^1$ is a test configuration of $x\in (X,\Delta)$. Let us prove that if $K_X+\Delta$ is $\bR$-Cartier, then so is $K_{\cX}+\cD$. Indeed, by shrinking $X$ around $x$ (which preserves $\gr_E R$ and thus does not change the central fiber $\cX_0$) we may assume that $K_X+\Delta\sim_\bR 0$, thus $K_{\cX\setminus\cX_0}+\cD|_{\cX\setminus\cX_0}\sim_\bR 0$. As $\cX_0$ is irreducible, we deduce that $K_{\cX}+\cD\sim_\bR c\cX_0$ for some $c\in \bR$. But $\cX_0\sim_\bR 0$, hence $K_{\cX}+\cD\sim_\bR 0$. In particular, $K_{\cX}+\cD$ is $\bR$-Cartier. It follows (by inversion of adjunction) that if $x\in (X,\Delta)$ is klt, then $(\cX,\cD)\to \bA^1$ is a special test configuration if and only if $(\cX_0,\cD_0)$ is klt. 
\end{expl}

The following special boundedness statement is one of the main results in the local K-stability theory of klt singularities.

\begin{thm} \label{thm:special bdd klt singularity}
Let $n\in \bN$, $\varepsilon>0$ and let $I\subseteq [0,1]$ be a finite set. Then there exists an $\bR$-Gorenstein family $\cB\subseteq (\cX,\cD)\to \cB$ of klt singularities with the action of an algebraic torus $\bT$ over $\cB$, such that:

For any klt singularity $x\in (X,\Delta)$ of dimension $n$, local volume at least $\varepsilon$ and coefficients in $I$, and any torus $\bT_0\subseteq \Aut(X,\Delta)$, there exists a $\bT_0$-equivariant special degeneration of $x\in (X,\Delta)$ to $b\in (\cX_b,\cD_b)$ for some $b\in \cB$ that identifies $\bT_0$ with a subtorus of $\bT_b$.
\end{thm}

\begin{proof}
This essentially follows from the stable degeneration theory of klt singularities together with the main results of \cite{XZ-local-bdd}. We refer to \cite[Section 2.2-2.3]{XZ-local-bdd} for the relevant definitions that appear in this proof. Enlarging the finite set $I$ if necessarily, we may assume that for any $a,b\in I$ with $a+b\le 1$ we have $a+b\in I$. By the Stable Degeneration Theorem \cites{Blu-minimizer-exist,LX-stability-higher-rank,Xu-quasi-monomial,XZ-minimizer-unique,XZ-SDC} (see also \cite{Z-SDC-survey} for a survey and \cite{Z-mld^K-2} for the extension to real coefficients), for every klt singularity $x\in (X=\Spec(R),\Delta)$, up to rescaling there is a unique and hence $\bT_0$-invariant quasi-monomial valuation $v$ minimizing the normalized volume function $\hvol_{X,\Delta}$. It induces a $\bT_0$-equivariant degeneration of $x\in (X,\Delta)$ to a K-semistable log Fano cone singularity $x_0\in (X_0=\Spec(\gr_v R),\Delta_0;
\xi_v)$ with 
\[
\hvol(x,X,\Delta)=\hvol(x_0,X_0,\Delta_0).
\]
Note that this implies that the torus $\langle \xi_v \rangle$ generated by the Reeb vector field $\xi_v$ commutes with $\bT_0$. Here $\gr_v R$ is defined by $\bigoplus_{\lambda\in\bR} \fa_\lambda/\fa_{>\lambda}$ where $\fa_\lambda = \{f\in R\,|\,v(f)\ge \lambda\}$ and $\fa_{>\lambda} = \{f\in R\,|\,v(f)>\lambda\}$.

We claim that $x_0\in (X_0,\Delta_0)$ can be realized as a $\bT_0$-equivariant special degeneration of $x\in (X,\Delta)$. Indeed, by \cite[Lemma 2.10]{LX-stability-higher-rank}, there exists some divisorial valuation $v_0$ over $x\in X$ (obtained from a small perturbation of the weights defining the quasi-monomial valuations $v$) such that $\gr_v R \cong \gr_{v_0} R$. Moreover, since $v$ is $\bT_0$-invariant, the construction in the proof of {\it loc. cit.} can be made $\bT_0$-equivariant (for example, we can choose $\bT_0$-invariant generators of $\gr_v R$ as well as their lifts to $R$), thus we may assume that the perturbed valuation $v_0$ is also $\bT_0$-invariant. Write $v_0=\lambda\cdot \ord_E$ where $E$ is a prime divisor over $x\in X$. Since $\gr_E R\cong \gr_{v_0} R\cong \gr_v R$, the divisor $E$ induces a $\bT_0$-equivariant special degeneration of $x\in (X,\Delta)$ to $x_0\in (X_0,\Delta_0)$ by the construction in Example \ref{exp:tc induced by divisor}. This proves the claim.

By construction, $\hvol(x_0,X_0,\Delta_0)\ge \varepsilon$ and the coefficients of $\Delta_0$ belong to $I$. Since $x_0\in (X_0,\Delta_0;
\xi_v)$ is K-semistable, we also have $\Theta(X_0,\Delta_0;\xi_v)=1$ where $\Theta(X_0,\Delta_0;\xi_v)$ denotes the volume ratio (see \cite[Definition 2.7]{XZ-local-bdd}). By perturbing $\xi_v$ in the Reeb cone, we can choose some Reeb vector field $\xi_0$ such that the torus $\langle \xi_0 \rangle$ it generates contains both $\langle \xi_v \rangle$ and $\bT_0$ and we have $\Theta(X_0,\Delta_0;\xi_0)\ge \frac{1}{2}$ (this uses the continuity of the normalized volume function on the Reeb cone, see \cite[Lemma 2.11]{Z-mld^K-2}). By \cite[Theorem 1.4]{XZ-local-bdd}, such polarized log Fano cone singularities $x_0\in (X_0,\Delta_0;\xi_0)$ are bounded. By definition (see \cite[Definition 2.16]{XZ-local-bdd}), this means that there exists an $\bR$-Gorenstein family $\cB\subseteq (\cX,\cD)\to \cB$ of klt singularities with the action of an algebraic torus $\bT$ over $\cB$, such that $x_0\in (X_0,\Delta_0;\xi_0)$ is isomorphic to some fiber $b\in (\cX_b,\cD_b)$ and $\xi_0$ is contained in the Reeb cone of $\bT_b$. Thus $\bT_b$ contains the torus generated by $\xi_0$ and in particular it contains $\bT_0$. This proves the theorem.
\end{proof}

\subsection{Log Fano and log Calabi-Yau fibrations} 

A key step in bounding the fibers of extremal contractions and flips (Section \ref{sec:fiber bdd}) is the special boundedness of certain log Fano and log Calabi-Yau fibrations. In this subsection, we discuss some properties of such fibrations and define their special boundedness.

\begin{defn} 
Let $(X,\Delta)$ be a pair and $X\to Z$ a projective morphism. % We say that
\begin{enumerate}
    \item We say that $(X,\Delta)$ is \emph{log Fano} over $Z$ if $(X,\Delta)$ is klt and $-(K_X+\Delta)$ is ample over $Z$.
    \item We say that $(X,\Delta)$ is \emph{log Calabi-Yau} over $Z$ if $(X,\Delta)$ is log canonical and $K_X+\Delta\sim_{\bR,Z} 0$.
    \item We say that $(X,\Delta)$ is \emph{of Fano type} over $Z$ if there exists some $\bR$-divisor $\Delta'\ge \Delta$ such that $(X,\Delta')$ is log Fano over $Z$. Equivalently, $-(K_X+\Delta)$ is big over $Z$ and there exists some $\bR$-divisor $0\le D\sim_{\bR,Z} -(K_X+\Delta)$ such that $(X,\Delta+D)$ is klt.
\end{enumerate}
A fibration $(X,\Delta)\to Z$ (or fibration germ $(X,\Delta)\to Z\ni z$) is called a \emph{log Fano} (resp. \emph{log Calabi-Yau}, resp. \emph{Fano type}) \emph{fibration} (or \emph{fibration germ}) if $(X,\Delta)$ is log Fano (resp. log Calabi-Yau, resp. of Fano type) over $Z$.
\end{defn}

\begin{lem} \label{lem:Fano type}
Let $f\colon (X',\Delta')\dashrightarrow (X,\Delta)$ be a birational map between pairs that are projective over $Z$. Assume that $K_X+\Delta$ is $\bR$-Cartier, $f^*(K_X+\Delta)\ge K_{X'}+\Delta'$, and $(X,\Delta)$ is of Fano type over $Z$. Then $(X',\Delta')$ is also of Fano type over $Z$.
\end{lem}

\begin{proof}
Since $(X,\Delta)$ is of Fano type over $Z$, we know that $-(K_X+\Delta)$ is big over $Z$ and there exists some $\bR$-divisor $0\le D\sim_{\bR,Z} -(K_X+\Delta)$ such that $(X,\Delta+D)$ is klt. As $f^*(K_X+\Delta)\ge K_{X'}+\Delta'$, by taking crepant pullback we see that both conditions are also satisfied by $(X',\Delta')$, hence $(X',\Delta')$ is of Fano type over $Z$.
\end{proof}

\begin{defn}
A \emph{test configuration} of a fibration $f\colon (X,\Delta)\to Z$ consists of a test configuration $(\cX,\cD)\to \bA^1$ of the pair $(X,\Delta)$ and a $\bG_m$-equivariant fibration $\widetilde{f}\colon (\cX,\cD)\to \cZ$ over $\bA^1$, where $\cZ$ is a test configuration of $Z$, such that $\widetilde{f}|_{\cX_t}=f$ for all $0\neq t \in \bA^1$. For simplicity, we denote the test configuration by $(\cX,\cD)\to \cZ$. Test configuration of a fibration germ $(X,\Delta)\to Z\ni z$ is defined as a test configuration $(\cX,\cD)\to \cZ$ of the fibration $(X,\Delta)\to Z$ such that $\cZ$ is also a test configuration of the singularity $z\in Z$.

A test configuration $(\cX,\cD)\to \cZ$ of a log Fano fibration $(X,\Delta)\to Z$ is called \emph{special} if $(\cX,\cX_0+\cD)$ is plt, $\cZ_0$ is normal and $-(K_{\cX}+\cD)$ is ample over $\cZ$. Note that under these conditions $\cX_0$ is normal by \cite{Kol13}*{Theorem 4.16}. In this case $(\cX_0,\cD_0)\to \cZ_0$ is also a log Fano fibration and we call it a \emph{special degeneration} of $(X,\Delta)\to Z$. Special test configurations and special degenerations of a log Fano fibration germ are defined analogously.
\end{defn}

\begin{defn}
We say that $(\cX,\cD)\to \cZ$ is an \emph{$\bR$-Gorenstein family of log Fano fibrations} over a finite type base $\cB$ if:
\begin{enumerate}
    \item Both $\cX$ and $\cZ$ are flat over $\cB$, and $\cX\to \cZ$ is a fibration,
    \item for any $b\in \cB$, the fibers $\cX_b$ and $\cZ_b$ are connected and normal,
    \item $\cD\geq 0$ is an $\bR$-divisor on $\cX$ whose support does not contain any fiber $\cX_b$ of $\cX\to \cB$,
    \item $\cX$ and $\cB$ are normal, $-(K_{\cX/\cB}+\cD)$ is $\bR$-Cartier and ample over $\cZ$, and
    \item the pair $(\cX_b,\cD_b)$ is klt for any closed point $b\in \cB$.
\end{enumerate}
In particular, $(\cX_b,\cD_b)\to \cZ_b$ is a log Fano fibration for every $b\in \cB$.
\end{defn}

\begin{defn} \label{defn:Fano type fibration bdd}
We say that a set $\cS$ of Fano type fibrations is \emph{bounded} if there exist an $\bR$-Gorenstein family of log Fano fibration $(\cX,\widetilde{\cD})\to \cZ$ over a finite type base $\cB$ and an $\bR$-divisor $0\leq\cD\le \tcD$, such that any Fano type fibration $(X,\Delta)\to Z$ in $\cS$ is isomorphic to $(\cX_b,\cD_b)\to \cZ_b$ for some $b\in \cB$.
\end{defn}

For log Fano fibrations, it is tempting to define their boundedness by requiring $\cD=\tcD$ in the above definition. The next lemma shows that these two definitions are equivalent.

\begin{lem} \label{lem:bdd defn equiv}
Let $\cS$ be a set of log Fano fibrations. Assume that $\cS$ is bounded as a set of Fano type fibrations. Then there exists a $\bR$-Gorenstein family $(\cX,\cD)\to \cZ$ of log Fano fibrations over a finite type base $\cB$ such that any log Fano fibration $(X,\Delta)\to Z$ in $\cS$ is isomorphic to $(\cX_b,\cD_b)\to \cZ_b$ for some $b\in \cB$.
\end{lem}

\begin{proof}
Let $(\cX,\tcD)\to \cZ$ and $0\leq\cD\le \tcD$ be as in Definition \ref{defn:Fano type fibration bdd}. Then there exists some constant $\varepsilon>0$ such that the all fibers of $(\cX,\tcD)\to \cB$ are $\varepsilon$-klt (i.e. their total minimal log discrepancies are larger than $\varepsilon$). It follows that for any $(X,\Delta)\to Z$ in $\cS$ we also have $(X,\Delta)$ is $\varepsilon$-klt. By \cite[Proposition 2.4]{HX-CY-bdd} and Noetherian induction, we deduce that possibly after stratifying $\cB$, we may assume that $K_{\cX/\cB}+\cD$ is $\bR$-Cartier. Thus $(\cX,\cD)\to \cZ$ is an $\bR$-Gorenstein family of log Fano fibrations and we are done.
\end{proof}

\begin{defn}
Let $(X,\Delta)\to Z$ be a Fano type fibration (or fibration germ). A test configuration $(\cX,\cD)\to \cZ$ of $(X,\Delta)\to Z$ is said to be \emph{of Fano type} if there exists an $\bR$-divisor $\widetilde{\cD}\ge \cD\geq 0$ on $\cX$ such that % $(\cX,\cX_0+\widetilde{\cD})$ is plt, $\cZ_0$ is normal, and $-(K_{\cX}+\widetilde{\cD})$ is ample over $\cZ$. In other words, 
$(\cX,\tcD)\to \cZ$ is an $\bR$-Gorenstein family of log Fano fibrations over $\bA^1$. In particular, $(\cX_0,\cD_0)\to \cZ_0$ is a Fano type fibration. In this case we also say that $(X,\Delta)\to Z$ degenerates to $(\cX_0,\cD_0)\to \cZ_0$ by a Fano type test configuration.
\end{defn}

\begin{lem} \label{lem:Fano type tc}
Let $(\cX,\cD)\to \cZ$ be a Fano type test configuration (of some fibration), let $(\cX',\cD')\to \cZ$ be another test configuration (of a possibly different fibration), and let $f\colon \cX'\dashrightarrow \cX$ be a $\bG_m$-equivariant birational map over $\cZ$. Assume that $\cX'_0$ is irreducible, $f$ is an isomorphism at the generic point of $\cX_0$, and $f^*(K_{\cX}+{\cD})\ge K_{\cX'}+\cD'$. Then $(\cX',\cD')\to \cZ$ is also a Fano type test configuration.
\end{lem}

\begin{proof}
The condition that $(\cX,\cD)\to \cZ$ is a Fano type test configuration is equivalent to saying that $-(K_{\cX}+\cD)$ is big over $\cZ$ and there exists some $\bR$-divisor $0\le \cG\sim_{\bR,\cZ} -(K_{\cX}+\cD)$ on $\cX$ such that $(\cX,\cX_t+\cD+\cG)$ is plt for every $t\in \bA^1$. These conditions are preserved by the given birational map $f$.
\end{proof}

\begin{defn} \label{defn:special bdd}
We say that a set $\cS$ of log Fano fibrations germs is \emph{specially bounded} if there exists a bounded set $\cT$ of log Fano fibrations such that any log Fano fibration germ $(X,\Delta)\to Z\ni z$ in $\cS$ specially degenerates to some $(X_0,\Delta_0)\to Z_0\ni z_0$ in $\cT$.

We say that a set $\cS$ of Fano type fibrations germs is \emph{bounded up to Fano type degenerations} if there exists a bounded set $\cT$ of Fano type fibrations such that any Fano type fibration germ $(X,\Delta)\to Z\ni z$ in $\cS$ degenerates to some $(X_0,\Delta_0)\to Z_0\ni z_0$ in $\cT$ by a Fano type test configuration. 

Finally, we say that a set $\cS$ of fibration germs is \emph{log specially bounded} if there exists a bounded set $\cT$ of Fano type fibrations, such that for any fibration germ $(X,\Delta)\to Z\ni z$ in $\cS$, we can find an $\bR$-divisor $D\ge 0$ on $X$ with $\Supp(\Delta)\subseteq \Supp(D)$, a Fano type fibration $(X_0,\Delta_0)\to Z_0\ni z_0$ in $\cT$, and some $\bR$-divisor $D_0\ge 0$ on $X_0$ with $\Supp(D_0)\subseteq \Supp(\Delta_0)$, such that $(X,D)\to Z\ni z$ is a Fano type fibration, and it degenerates to $(X_0,D_0)\to Z_0\ni z_0$ by a Fano type test configuration. Intuitively, the conditions mean that $\Supp(\Delta)$ degenerates into $\Supp(\Delta_0)$ by some Fano type test configuration.
\end{defn}

\section{Boundedness of invariants in MMP}\label{sec: bdd inv mmp}

The major goal of this section is to prove the following theorem, which is a relative version of Theorem \ref{main:index+mld}.

\begin{thm} \label{thm:index+mld in MMP}
Let $(X,\Delta)$ be a $\bQ$-factorial klt pair and $X\rightarrow T$ a projective morphism. Assume that either $K_X+\Delta$ or $\Delta$ is big over $T$. Then there exist a positive integer $r$ and a finite set $S\subseteq \bR_{+}$, depending only on $(X,\Delta)$ and $\pi$, such that for any sequence of steps of a $(K_X+\Delta)$-MMP $(X,\Delta)\dashrightarrow (Y,\Delta_Y)$ over $T$,
\begin{enumerate}
    \item the Cartier index of any Weil divisor on $Y$ is at most $r$, and
    \item for any (not necessarily closed) point $y\in Y$, we have $\mld(y,Y,\Delta_Y)\in S$.
\end{enumerate}
\end{thm}

\subsection{Local volumes in MMP type contractions}

We shall deduce Theorem \ref{thm:index+mld in MMP} from the following more general result on the boundedness of invariants in MMP type contractions (Definition \ref{defn:MMP type contraction}).

\begin{thm} \label{thm:index+mld+vol bound}
Let $(X,\Delta+D)$ be a klt pair where $\Delta, D\ge 0$ and $D$ is big. Then there exist a positive integer $r$ and two real number $\varepsilon,A>0$, depending only on $X$, $\Delta$ and $D$, such that for any MMP type contraction $\varphi\colon (X,\Delta+D)\dashrightarrow (X',\Delta'+D')$ with $\Delta'\ge 0$ and $D'=\varphi_* D$, we have:
\begin{enumerate}
    \item The Cartier index of any $\bQ$-Cartier Weil divisor on $X'$ is at most $r$,
    \item for any (not necessarily closed) point $x'\in X'$, we have $\mld(x',X',\Delta'+D')\le A$, and
    \item if $K_{X'}+\Delta'$ is $\bR$-Cartier, then $\hvol(x',X',\Delta')\ge \varepsilon$ for all closed points $x'\in X'$.
\end{enumerate}
\end{thm}

Before we prove this more general result, let us explain how it implies Theorem \ref{thm:index+mld in MMP}.

\begin{proof}[Proof of Theorem \ref{thm:index+mld in MMP} assuming Theorem \ref{thm:index+mld+vol bound}]
By Lemma \ref{lem: gt klt sim to big boundary}, there exists some $\bR$-divisor $\Delta_0\ge 0$ on $X$ such that $(X,\Delta_0)$ is klt, $\Delta_0$ is big, and any $(K_X+\Delta)$-MMP over $T$ is also a $(K_X+\Delta_0)$-MMP. Part (1) then follows directly from Theorem \ref{thm:index+mld+vol bound}(1). By Lemma \ref{lem:index bdd imply mld discrete}, this implies that the log discrepancies $A_{Y,\Delta_Y}(E)$ of prime divisors $E$ over $Y$ belong to a discrete set $S_0$ that depends only on the coefficients of $\Delta$ and the index bound $r$. Note that $Y$ is $\bQ$-factorial. Thus by Theorem \ref{thm:index+mld+vol bound}(3), we also know that there exists some positive real number $\varepsilon$ (depending only on $X$ and $\Delta_0$) such that $\hvol(y,Y)\ge \varepsilon$ for any closed point $y\in Y$. Combined with Theorem \ref{thm:mld bound via volume}, we see that there exists some constant $A>0$ (depending only on $\dim X$ and $\varepsilon$) such that $\mld(y,Y)\le A$ for all points $y\in Y$, hence $\mld(y,Y,\Delta_Y)\le A$ as well. It follows that $\mld(y,Y,\Delta_Y)$ belongs to the finite set $S:=S_0\cap [0,A]$.
\end{proof}

In the rest of this subsection, we prove Theorem \ref{thm:index+mld+vol bound}. Since local volume controls most other invariants such as index and mld, we will primarily focus on bounding the local volume away from zero. A key ingredient is the following generalization of \cite{XZ-local-bdd}*{Lemma 2.13} ({\it cf.} \cite{BJ-delta}*{Theorem D}).

\begin{prop} \label{prop:alpha-vol inequality linear series}
Let $(X,\Delta)$ be a klt pair of dimension $n$, and let $V_\bullet$ be an eventually birational graded linear series (see Section \ref{ss:linear series}) of some $\bR$-Cartier $\bR$-divisor $L$ on $X$. Then for any closed point $x\in X$ we have
\[
\hvol(x,X,\Delta)\ge \alpha(X,\Delta;V_\bullet)^n \cdot \vol(V_\bullet).
\]
\end{prop}

The proof is in turn based on the following auxiliary lemma that we will apply again in Section \ref{sec:fiber bdd}.

\begin{lem} \label{lem:vol estimate}
Let $(X,\Delta)$ be a klt pair of dimension $n$, let $x\in X$ be a closed point, and let $c_0>c_1$ be positive constants. Then for any sufficiently large integer $m$ and any linear subspace $V\subseteq \cO_{X,x}$ of dimension at least $c_0\frac{m^n}{n!}$, there exists some nonzero element $s\in V$ such that 
\[
m\cdot \lct_x (X,\Delta;\{s=0\})\le \left(\frac{\hvol(x,X,\Delta)}{c_1}\right)^{1/n}.
\]
\end{lem}

\begin{proof}
This is proved by the same argument as in \cite{BJ-delta}*{Theorem D} or \cite{XZ-local-bdd}*{Lemma 2.13}. Let $c:=\left(\frac{\hvol(x,X,\Delta)}{c_1}\right)^{-1/n}$ and let $v$ be a valuation over $x\in X$ such that $\hvol_{X,\Delta}(v)=\hvol(x,X,\Delta)$. Since the normalized volume is rescaling invariant (i.e. $\hvol_{X,\Delta}(\lambda v) = \hvol_{X,\Delta}(v)$ for all $v\in \Val_{X,x}$ and $\lambda>0$), we may rescale the minimizing valuation $v$ and assume that $A_{X,\Delta}(v)=1$. It follows that $\vol(v)=\hvol(x,X,\Delta)$. Let $\fa_\lambda:=\fa_\lambda(v)$ ($\lambda\in \bR$) be the valuation ideals of $v$. Since $c_0>c_1>0$, for sufficiently large $m$ we have
\[
\dim (\cO_{X,x}/\fa_{cm}) = \hvol(x,X,\Delta)\cdot \frac{(cm)^n}{n!} + o(m^n) = \frac{c_1 m^n}{n!} + o(m^n) < \frac{c_0 m^n}{n!}.
\]
Thus for any linear subspace $V\subseteq \cO_{X,x}$ of dimension at least $c_0\frac{m^n}{n!}$, the natural map $V\to \cO_{X,x}/\fa_{cm}$ cannot be injective. It follows that there exists some nonzero element $s\in V\cap \fa_{cm}(v)$; in particular $v(s)\ge cm$ and
\[
m\cdot \lct_x (X,\Delta;\{s=0\})\le m\cdot \frac{A_{X,\Delta}(v)}{v(s)}\le \frac{1}{c} = \left(\frac{\hvol(x,X,\Delta)}{c_1}\right)^{1/n}.
\]
This completes the proof.
\end{proof}

\begin{proof}[Proof of Proposition \ref{prop:alpha-vol inequality linear series}]
First assume that $L$ is $\bQ$-Cartier. For each sufficiently divisible positive integer $m$, after choosing a local trivialization of $mL$ around $x\in X$, we may view each $V_m$ as a subspace of $\cO_{X,x}$. Let $0<\varepsilon\ll 1$ and set $c_0=\vol(V_\bullet)-\varepsilon$ and $c_1=\vol(V_\bullet)-2\varepsilon$. Then $\dim V_m \ge c_0\frac{m^n}{n!}$ when $m\gg 0$. Thus Lemma \ref{lem:vol estimate} gives $\hvol(x,X,\Delta)\ge c_1\cdot \alpha(X,\Delta;V_\bullet)^n$. Letting $\varepsilon\to 0$ and hence $c_1\to \vol(V_\bullet)$, we get the desired result.

In general, choose some Cartier divisor $D\ge 0$ on $X$ such that $\Supp(L)\subseteq \Supp(D)$. For any $\varepsilon>0$, there exists some $\bQ$-Cartier $\bQ$-divisor $L'$ such that $L\le L'\le L+\varepsilon D$. Let $V'_\bullet$ be the image of $V_\bullet$ under the natural embedding $H^0(mL)\hookrightarrow H^0(mL')$. Since $L'$ is $\bQ$-Cartier, the previous case implies that
\[
\hvol(x,X,\Delta)\ge \alpha(X,\Delta;V'_\bullet)^n \cdot \vol(V'_\bullet) = \alpha(X,\Delta;V'_\bullet)^n \cdot \vol(V_\bullet).
\]
On the other hand, as $|V'_m|=|V_m|+m(L'-L)$ and $0\le L'-L\le \varepsilon D$, by Lemma \ref{lem:lct inequality}, we deduce that 
\[
\left(\frac{1}{\alpha(X,\Delta;V_\bullet)}+\frac{\varepsilon}{\lct(X,\Delta;D)} \right)^{-1}\le \alpha(X,\Delta;V'_\bullet) \le \alpha(X,\Delta;V_\bullet),
\]
hence $\alpha(X,\Delta;V'_\bullet)$ can be arbitrarily close to $\alpha(X,\Delta;V_\bullet)$ when $\varepsilon\to 0$. The proposition thus follows.
\end{proof}

In the context of Theorem \ref{thm:index+mld+vol bound}, the linear systems we shall use come from the following construction.

\begin{lem} \label{lem:perturb div into lin sys}
Let $(X,\Delta+D)$ be a  klt pair such that $\Delta,D\ge 0$ and $D$ is big. Then there exists an eventually birational graded linear series of $D$ such that $(X,\Delta+\Gamma)$ is log canonical for all $m\in M(V_\bullet)$ and all $\Gamma\in \frac{1}{m}|V_m|$.
\end{lem}

\begin{proof}
Since $D$ is big, we may write $D=A+B$ where $A$ is an ample $\bQ$-divisor and $B\ge 0$. Let $\oX$ be a normal projective compactification of $X$ such that $A$ extends to an ample $\bQ$-divisor $\overline{A}$ (for example, consider the morphism from $X$ into some projective space induced by the linear system $|qA|$ for some sufficiently positive and divisible integer $q$, and take $\oX$ to be the normalization of the closure of the image). Then $D$ also extends to a big $\bR$-divisor $\oD=\overline{A}+\overline{B}$ on $\oX$, where $\overline{B}$ is the closure of $B$. Let $V_m \subseteq H^0(X,mD)$ be the image of the restriction map $H^0(\oX,m\oD)\to H^0(X,mD)$. Since $(X,\Delta+D)$ is klt, by Lemma \ref{lem:Izumi} applied to some log resolution of the pair $(X,\Delta+D)$, we know that there exists some sufficiently large integer $\ell>0$ such that $(X,\Delta+\left(1-\frac{1}{\ell}\right)D + \frac{1}{\ell} G)$ is log canonical for all $m\in \bN^+$ and all $G\in \frac{1}{m}|V_m|$. Define a new graded linear series $V'_\bullet$ of $D$ by 
\[
|V'_m| = m\left(1-\frac{1}{\ell}\right)D+|V_{m/\ell}|
\]
when $\ell$ divides $m$, and $V'_m=0$ otherwise. Since $\oD$ is big, both $V_\bullet$ and $V'_\bullet$ are eventually birational. By construction, $(X,\Delta+\Gamma)$ is log canonical for all $m\in \ell\cdot \mathbb N^+$ and $\Gamma\in \frac{1}{m}|V'_m|$. Thus $V'_\bullet$ satisfies the desired conditions. 
\end{proof}

We are now ready to prove the theorem.

\begin{proof}[Proof of Theorem \ref{thm:index+mld+vol bound}]
We first prove (3). Since $K_{X'}+\Delta'$ is $\bR$-Cartier, so is $D'$. By Lemma \ref{lem:perturb div into lin sys}, we can find an eventually birational graded linear series $V_\bullet$ of $D$ such that $(X,\Delta+\Gamma)$ is log canonical for all $m\in M(V_\bullet)$ and all $\Gamma\in \frac{1}{m}|V_m|$. Since $\Gamma\sim_\bR D$ and $\varphi$ is an MMP type contraction between $(X,\Delta+D)$ and $(X',\Delta'+D')$, by Lemma \ref{lem:MMP type change boundary} it is also an MMP type contraction between $(X,\Delta+\Gamma)$ and $(X',\Delta'+\Gamma')$ where $\Gamma':=\varphi_* \Gamma$. It follows that $(X',\Delta'+\Gamma')$ is log canonical and thus $\alpha(X',\Delta';\varphi_*V_\bullet)\ge 1$. Note that $\varphi_*V_\bullet$ is a graded linear series of $D'$. By Proposition \ref{prop:alpha-vol inequality linear series}, we obtain 
$$\hvol(x',X',\Delta')\ge \vol(\varphi_* V_\bullet)=\vol(V_\bullet)$$ 
for all closed points $x'\in X'$. Hence (3) holds with $\varepsilon:=\vol(V_\bullet)>0$.

Next we prove (1). Let $n=\dim X$. Take a small $\bQ$-factorial modification $\pi\colon X''\to X'$ (which exists by \cite{BCHM}*{Corollary 1.4.3}) and let $\Delta'',D''$ be the strict transforms of $\Delta',D'$ respectively. Then $K_{X''}+\Delta''$ is $\bR$-Cartier and the induced map $(X,\Delta+D)\dashrightarrow (X'',\Delta''+D'')$ is also an MMP type contraction. By the base point free theorem \cite{KM98}*{Theorem 3.24}, the Cartier index of any $\bQ$-Cartier Weil divisor $L'$ on $X'$ is the same as the Cartier index of $\pi^*L'$ on $X''$. Now for any closed point $x''\in X''$ we have $\hvol(x'',X'',\Delta'')\ge \varepsilon$ by (3) and hence the Cartier index of $\pi^*L'$ is at most $\left\lfloor \frac{n^n}{\varepsilon}\right\rfloor$ at $x''$ by Theorem \ref{thm:index bound via volume}. Thus (1) holds with $r=\left\lfloor \frac{n^n}{\varepsilon}\right\rfloor!$.

Finally, by Theorem \ref{thm:mld bound via volume}, there exists some constant $A>0$ (depending only on $n,\varepsilon$) such that $\mld(x'',X'',\Delta'') \le A$ for all points $x''\in X''$. As $\pi$ is small, we have $\pi^*(K_{X'}+\Delta'+D')=K_{X''}+\Delta''+D''$, hence $$\mld(x',X',\Delta'+D')\le \mld(x'',X'',\Delta''+D'')\le \mld(x'',X'',\Delta'')$$ for any $x''\in \pi^{-1}(x)$. Thus it also holds that $\mld(x',X',\Delta'+D')\le A$, which gives (2).
\end{proof}

\begin{rem} \label{rem:local vol not monotone}
It is not clear if the local volume satisfies any monotonicity in the MMP; at least the na\"ive versions fail. For example, there exist flops $X_1\dashrightarrow X_2$ and valuations $v$ whose center is a closed point on both $X_1$ and $X_2$ but $\hvol_{X_1}(v)\neq \hvol_{X_2}(v)$. To get an explicit example, let $Z$ be the cone over the smooth quadric surface $Q\cong \bP^1\times \bP^1\subseteq \bP^3$ and consider the Atiyah flop $X_1\dashrightarrow X_2$ between its two small resolutions. Then $Z$ is a toric variety and most toric valuations $v$ satisfy $\hvol_{X_1}(v)\neq \hvol_{X_2}(v)$. By adding appropriate boundaries to $X_i$ we can turn the flop $X_1\dashrightarrow X_2$ into either a flip or an anti-flip. Thus there exist flips $(X,\Delta)\dashrightarrow (X^+,\Delta^+)$  where the normalized volume of valuations increases (resp. decreases).

Alternatively, one may ask whether the smallest local volume $\inf_{x\in X} \hvol(x,X,\Delta)$ of the klt pair $(X,\Delta)$ is monotone in the MMP. The answer is again negative. Indeed, by \cite{LX-cubic-3fold}*{Theorem A.4}, we have $\inf_{x\in X}\hvol(x,X,\Delta)\le n^n$ with equality if and only if $X$ is smooth and $\Delta=0$. There are flips $X\dashrightarrow X^+$ such that $X$ is singular (resp. smooth) whereas $X^+$ is smooth (resp. singular), in which case the smallest local volume increases (resp. decreases) after the flips. For some explicit examples, consider the total space $\tX$ of the line bundle $L=\cO_{\bP^m\times \bP^n}(-1,-2)$ on $\bP^m\times \bP^n$. Let $X_1$ (resp. $X_2$) be obtained from $\tX$ by contracting the zero section $E\cong \bP^m\times \bP^n$ using the first (resp. second) projection of $\bP^m\times \bP^n$. Then $X_1$ is smooth and $X_2$ is singular. On the other hand, a direct calculation shows that the induced birational map $X_1\dashrightarrow X_2$ (resp. $X_2\dashrightarrow X_1$) is a flip if and only if $2m+1<n$ (resp. $2m+1>n$).
\end{rem}

\subsection{Log canonical volume} \label{ss:lc vol}

Although the local volume does not seem to satisfy monotonicity in the MMP, we can nonetheless introduce a global variant that does. The definition is motivated by the argument in the previous subsection. We begin with an auxiliary definition.

\begin{defn} \label{defn:admissible linear series}
Let $(X,\Delta)$ be a pair and let $L$ be a big $\bR$-divisor such that $K_X+\Delta+L$ is $\bR$-Cartier. A graded linear series $V_\bullet$ of $L$ is called \emph{admissible} (with respect to $(X,\Delta)$) if it is eventually birational and $(X,\Delta+\Gamma)$ is log canonical for all $m\in M(V_\bullet)$ and all $\Gamma\in \frac{1}{m}|V_m|$. Note that the latter condition is equivalent to $\alpha(X,\Delta;V_{\bullet})\geq 1$ when $K_X+\Delta$ is $\bR$-Cartier. 
\end{defn}

\begin{lem}\label{lem: admissible-exist-klt}
Let $(X,\Delta)$ and $L$ be as in Definition \ref{defn:admissible linear series}. Then an admissible graded linear series of $L$ exists if and only if there exists some $0\le D\sim_\bR L$ such that $(X,\Delta+D)$ is klt.
\end{lem}

\begin{proof}
If there exists some $0\le D\sim_\bR L$ such that $(X,\Delta+D)$ is klt, then we can choose such $D$ with $D\sim_\bQ L$. Thus one direction is implied by Lemma \ref{lem:perturb div into lin sys}. For the other direction, let $V_\bullet$ be an admissible graded linear series of $L$. By assumption, there exists some $m\in\bN_+$ such that the rational map $f\colon X\dashrightarrow \bP^N$ induced by the linear system $|V_m|$ is birational onto its image. Let $\pi\colon X'\to X$ be a log resolution of $(X,\Delta)$ that resolve the indeterminacy of $f$ and let $g\colon X'\to \bP^N$ the induced morphism. For any prime divisor $F$ on $X'$, let
\[
a_F:=-\sup\left\{a(F,X,\Delta+D)\middle|\ D\in\frac{1}{m} |V_m|\right\},
\]
and let $\Delta':=\sum a_F F$. Note that this is a finite sum since $a_F=0$ unless $F$ is $\pi$-exceptional or the strict transform of some component of $\Delta+D$ for any $D\in |V_m|$. We also note that $\pi_*\Delta'-\Delta$ is the fixed part of $|V_m|$.

Let $L':=\pi^*(K_X+\Delta+L)-(K_{X'}+\Delta')$. Then the linear series $V_m$ of $mL$ induces a linear series $V'_m$ of $mL'$ defined by
\[
|V'_m| = \{\pi^*(mK_X+m\Delta+D)-mK_{X'}-m\Delta' \,|\,D\in |V_m| \},
\]
and the morphism $g$ is induced by $|V'_m|$. In particular, the linear system $|V'_m|$ is base point free. We claim that the sub-pair $(X',\Delta')$ is klt. In fact, since $V_\bullet$ is admissible, we know that $(X,\Delta+D)$ is log canonical for any positive integer $\ell$ and all $D\in \frac{1}{\ell m} |V_{\ell m}|$. This implies that for any hypersurface $H\subseteq \bP^N$ of degree $\ell$ whose support does not contain $g(X')$, the sub-pair $(X',\Delta'+\frac{1}{\ell m}g^*H)$ is log canonical, as $\frac{1}{\ell m}g^*H\sim L'$ and thus the sub-pair is the crepant pullback of $(X,\Delta+D)$ where $D=\pi_*\Delta'-\Delta + \frac{1}{\ell m}\pi_*g^*H\in \frac{1}{\ell m} |V_{\ell m}|$. If $(X',\Delta')$ is not klt, then there exists some prime divisor $E$ over $X'$ such that $A_{X',\Delta'}(E)\le 0$. As $g$ is a birational morphism, we can choose some hypersurface $H$ as above of sufficiently large degree such that $g^*H$ contains the center of $E$. But then $A_{X',\Delta'+\frac{1}{\ell m}g^*H}(E)<0$ and hence $(X',\Delta'+\frac{1}{\ell m}g^*H)$ cannot be log canonical, a contradiction. This proves the lemma.
\end{proof}

We now introduce the global invariant that behaves well in the MMP.

\begin{defn} \label{defn:Vol}
Let $(X,\Delta)$ be a pair and let $L$ be a big $\bR$-divisor such that $K_X+\Delta+L$ is $\bR$-Cartier. We define the \textit{log canonical volume} of $L$ with respect to $(X,\Delta)$, denoted by $\hVol_{X,\Delta}(L)$, as follows:
\begin{enumerate}
    \item If $(X,\Delta+D)$ is klt for some $0\le D\sim_\bR L$, we set
    \[
    \hVol_{X,\Delta}(L):=\sup\{\vol(V_\bullet)\,|\,V_\bullet \mbox{ is an admissible graded linear series of } L\}.
    \]
    In this case $\hVol_{X,\Delta}(L)>0$ by Lemma \ref{lem:perturb div into lin sys}.
    \item Otherwise, by Lemma \ref{lem: admissible-exist-klt}, there are no admissible graded linear series of $L$ and we set $\hVol_{X,\Delta}(L):=0$. 
\end{enumerate}
% A graded linear series $V_\bullet$ of $L$ is called admissible (with respect to $(X,\Delta)$) if it is eventually birational and $(X,\Delta+\Gamma)$ is log canonical for all $m\in M(V_\bullet)$ and all $\Gamma\in \frac{1}{m}|V_m|$. Note that the latter condition is equivalent to $\alpha(X,\Delta;V_{\bullet})\geq 1$ when $K_X+\Delta$ is $\bR$-Cartier. 
% \[
% \hVol_{X,\Delta}(L):=\sup\{\vol(V_\bullet)\,|\,V_\bullet \mbox{ is an admissible graded linear series of } L\}.
% \]
% By Lemma \ref{lem:perturb div into lin sys}, $\hVol_{X,\Delta}(L)>0$.
\end{defn}

While we do not need the next statement in this paper, we note that the log canonical volume defined above only depends on the numerical equivalence class of $L$. Throughout the rest of this subsection, let $n=\dim X$.

\begin{lem}
Let $(X,\Delta)$ and $L$ be as in Definition \ref{defn:Vol}, and let $L'$ be another $\bR$-divisor on $X$ such that $L'-L$ extends to an $\bR$-divisor in $\Pic^0(\oX)\otimes \bR$ for some projective compactification $\oX$ of $X$. Then
\[
\hVol_{X,\Delta}(L)=\hVol_{X,\Delta}(L').
\]
\end{lem}

\begin{proof}
It suffices to show that $\hVol_{X,\Delta}(L')\ge \hVol_{X,\Delta}(L)$, as by symmetry we also get the reverse inequality. To prove this, we may assume that there exists some big $\bR$-divisor $0\le D\sim_\bR L$ such that $(X,\Delta+D)$ is klt, and it is enough to show that $\hVol_{X,\Delta}(L')\ge \left(1-\frac{1}{\ell}\right)^n \vol(V_\bullet)$ for any sufficiently large integer $\ell$ and any admissible graded linear series $V_\bullet$ of $L$. Since $L'-L \in \Pic^0(\oX)\otimes \bR$, we can find some $\bR$-divisor $0\le D'\sim_\bQ \ell(L'-L)+L$ such that $(X,\Delta+D')$ is also klt. To see this, write $D=A+B$ where $A$ is ample and $B\ge 0$ such that the pair $(X,\Delta+B)$ is also klt, then $A+\ell(L'-L)$ is still ample and we can take $D'=A'+B$ for some general $0\le A'\sim_\bQ A+\ell(L'-L)$. By construction, $(1-\frac{1}{\ell})L+\frac{1}{\ell}D'\sim_\bQ L'$. For any sufficiently divisible integer $m$, let $V'_m$ be the linear series of $mL'$ defined by $|V'_m|=|V_{(1-\frac{1}{\ell})m}|+\frac{m}{\ell}D'$. Then $V'_\bullet$ is admissible by the convexity of log canonicity. It follows that $\hVol_{X,\Delta}(L')\ge \vol(V'_\bullet) = \left(1-\frac{1}{\ell}\right)^n \vol(V_\bullet)$ as desired.
\end{proof}

The proof of Theorem \ref{thm:index+mld+vol bound} can be reformulated in terms of some properties of the log canonical volume $\hVol_{X,\Delta}(L)$. These include its monotonicity in the MMP, and its comparison with the local volume.

\begin{lem} \label{lem:Vol MMP}
Let $\varphi\colon (X,\Delta+D)\dashrightarrow (X',\Delta'+D')$ be an MMP type contraction between klt pairs. Assume that $\Delta,\Delta'\ge 0$, $D$ is big, and $D'=\varphi_* D$. Then $\hVol_{X',\Delta'}(D')\ge \hVol_{X,\Delta}(D)$.
\end{lem}

\begin{proof}
If $V_\bullet$ is an admissible graded linear series of $D$, then $\varphi_* V_\bullet$ is an admissible graded linear series of $D'$ as in the proof of Theorem \ref{thm:index+mld+vol bound}.
\end{proof}

\begin{lem} \label{lem:local vs global vol}
Let $(X,\Delta)$ be a klt pair and let $D\ge 0$ be a big $\bR$-Cartier $\bR$-divisor on $X$. % such that $(X,\Delta+D)$ is also klt. 
Then for any closed point $x\in X$ we have $\hvol(x,X,\Delta)\ge \hVol_{X,\Delta}(D)$.
\end{lem}

\begin{proof}
This follows immediately from Proposition \ref{prop:alpha-vol inequality linear series}. 
\end{proof}

We make a few more elementary observations on the log canonical volume $\hVol_{X,\Delta}(L)$ that will be useful in the next several sections. % Most of them are either reformulations from the proof of Theorem \ref{thm:index+mld+vol bound}, or follow easily from a similar argument.

\begin{lem} \label{lem:index bound via global Vol}
Let $(X,\Delta)$ be a klt pair such that $\Delta$ is big. Then the Cartier index of any $\bQ$-Cartier Weil divisor on $X$ is at most $\left\lfloor\frac{n^n}{\hVol_X(\Delta)}\right\rfloor !$.
\end{lem}

\begin{proof}
Let $X'\to X$ be a small $\bQ$-factorial modification and let $\Delta'$ be the strict transform of $\Delta$. By Lemma \ref{lem:local vs global vol}, we have $\hvol(x',X')\ge \hVol_{X'}(\Delta')$ for any closed point $x'\in X'$, hence by Theorem \ref{thm:index bound via volume} we know that the Cartier index of any Weil divisor on $X'$ is at most $\left\lfloor\frac{n^n}{\hVol_{X'}(\Delta')}\right\rfloor !$. Note that $\hVol_{X'}(\Delta')=\hVol_X(\Delta)$ by Lemma \ref{lem:Vol MMP}. The result then follows from the base point free theorem as in the proof of Theorem \ref{thm:index+mld+vol bound}.
\end{proof}

\begin{lem} \label{lem:Vol change boundary}
Let $(X,\Delta)$ be a klt pair such that $\Delta$ is big, let $\lambda\in [0,1)$, and let $0\le \Delta_0\le \lambda \Delta$ be an $\bR$-divisor. Then $$\hVol_X(\Delta)\ge \hVol_{X,\Delta_0}(\Delta-\Delta_0)\ge (1-\lambda)^n \hVol_X(\Delta).$$
\end{lem}

\begin{proof}
If $V_\bullet$ is an admissible graded linear series of $\Delta-\Delta_0$ with respect to $(X,\Delta_0)$, then its image $V'_\bullet$ under the natural inclusion $H^0(X,m(\Delta-\Delta_0))\hookrightarrow H^0(X,m\Delta)$ is an admissible graded linear series of $\Delta$ with respect to $X$. This gives the first inequality. To prove the second inequality we may assume that $\lambda\in \bQ$. For any admissible graded linear series $V_\bullet$ of $\Delta$, let $V'_\bullet$ be the graded linear series of $\Delta-\Delta_0$ defined by $|V'_m|=|V_{(1-\lambda)m}|+m(\lambda \Delta - \Delta_0)$ for sufficiently divisible $m\in\bN$. By convexity of log canonicity, we know that $(X,\lambda\Delta+(1-\lambda)D)$ is log canonical for all $m\in M(V_\bullet)$ and $D\in \frac{1}{m}|V_m|$. This implies that $V'_\bullet$ is admissible with respect to $(X,\Delta_0)$. Since $V_\bullet$ is arbitrary and $\vol(V'_\bullet)=(1-\lambda)^n \vol(V_\bullet)$, the lemma follows.
\end{proof}

\begin{lem} \label{lem:Vol add ample boundary}
Let $(X,\Delta+D)$ be a klt pair where $\Delta\ge 0$ and $D$ is big. Let $H$ be an ample $\bR$-divisor on $X$. Then
\[
\hVol_{X,\Delta}(D+H)\ge \hVol_{X,\Delta}(D).
\]
\end{lem}

\begin{proof}
Let $\ell$ be a sufficiently large and divisible integer. It is enough to show that $\hVol_{X,\Delta}(D+H)\ge (1-\frac{1}{\ell})^n \vol(V_\bullet)$ for any admissible graded linear series $V_\bullet$ of $D$. Since $(X,\Delta+D)$ is klt, $H$ is ample, and $\ell$ is sufficiently large and divisible, by Bertini's theorem we can find some $H_0\in |\ell H|$ such that $(X,\Delta+D+H_0)$ is log canonical. Define a graded linear series $V'_\bullet$ of $D+H$ by $|V'_m| = |V_{\left(1-\frac{1}{\ell}\right)m}|+\frac{m}{\ell}(D+H_0)$ for all $m\in\ell\cdot\bN$. Then $V'_\bullet$ is admissible by convexity of log canonicity as before and this proves the result.
\end{proof}

We also prove a semi-continuity result for the log canonical volume, which allows us to reduce many questions to the case of rational coefficients. Let $\mathrm{Div}(X)$ be the free group generated by the prime divisors on $X$. 
\begin{lem} \label{lem:Vol lsc}
Let $(X,\Delta)$ be a pair, let $\Lambda\subseteq \mathrm{Div}(X)_\bR$ be a finite dimensional subspace, and let $\cL(\Lambda)\subseteq\Lambda$ be the subset of $\bR$-divisors $D\geq 0$ such that $(X,\Delta+D)$ is klt and $D$ is big. Then the function $D\mapsto \hVol_{X,\Delta}(D)$ is lower semi-continuous in $\cL(\Lambda)$.
\end{lem}

\begin{proof}
Let $\ell\in\bN$ be a sufficiently large integer. It suffices to show that $\hVol_{X,\Delta}(D')\ge\left(1-\frac{1}{\ell}\right)^n \hVol_{X,\Delta}(D)$ whenever $D'\in \cL(\Lambda)$ is sufficiently close to $D$. To see this, we may assume that $D'\ge\left(1-\frac{1}{\ell}\right)D$ and $$\left(X,\Delta+\ell\left(D'-\left(1-\frac{1}{\ell}\right)D\right)\right)=(X,\Delta+D+\ell(D'-D))$$ is klt. If $V_\bullet$ is an admissible graded linear series of $D$, and define
\[
\left|V'_m\right| = \left|V_{\left(1-\frac{1}{\ell}\right)m}\right|+mD'-m\left(1-\frac{1}{\ell}\right)D
\]
when $\ell$ divides $m$ (otherwise set $V'
_m=0$), then by convexity of log canonicity as above we see that $V'_\bullet$ is an admissible graded linear series of $D'$ and by construction $\vol(V'_\bullet)\ge (1-\frac{1}{\ell})^n \vol(V_\bullet)$. This proves the result.
\end{proof}

\subsection{Explicit bounds} \label{sec:vol explicit bdd}

In the remaining part of this section, we prove some effective versions of Theorem \ref{thm:index+mld+vol bound}, giving explicit bounds on the invariants at least when the pair is log smooth and projective (in general, one can reduce to this case by passing to a log resolution of some projective compactification). % We also prove an explicit index bound for $3$-fold MMP. These will be useful in Section \ref{sec:explicit tof} where we bound the number of steps in certain MMP.

We start with an explicit bound of the log canonical volume.

\begin{lem} \label{lem:Vol explicit bdd D big}
Let $(X,\Delta)$ be a projective log smooth klt pair of dimension $n$ such that $\Delta$ is big, and let $H$ be a very ample divisor on $X$. Then
\[
\hVol_X(\Delta)\ge \left\lceil \frac{(\Delta\cdot H^{n-1})}{1-\max \Coef(\Delta)}\right\rceil^{-n} \cdot \vol(\Delta)
\]
\end{lem}

\begin{proof}
Since $H$ is very ample, we have $\mult_x D\le (\Delta\cdot H^{n-1})$ for any $x\in X$ and any $0\le D\sim_\bR \Delta$. By Lemma \ref{lem:Izumi}, this gives 
\[
\lct(X,\Delta;D)\ge \frac{1-\max \Coef(\Delta)}{(\Delta\cdot H^{n-1})}.
\]
Thus in the proof of Lemma \ref{lem:perturb div into lin sys} we can take $\ell$ to be $\left\lceil \frac{(\Delta\cdot H^{n-1})}{1-\max \Coef(\Delta)}\right\rceil$. It then follows from the argument in \textit{loc. cit. } that there exists an admissible linear series $V_\bullet$ of $\Delta$ with $\vol(V_\bullet)\ge \ell^{-n}\vol(\Delta)$, hence $\hVol_X(\Delta)\ge \ell^{-n}\vol(\Delta)$ as desired.
\end{proof}

\begin{cor} \label{cor:index explicit bdd D big}
Let $(X,\Delta)$ be a projective log smooth klt pair of dimension $n$, and let $N$ be a positive integer. Assume that $N\Delta$ has integer coefficients, $\vol(\Delta)\ge \frac{1}{N}$, and $(\Delta\cdot H^{n-1})\le N$ for some very ample divisor $H$ on $X$. Then in any $(K_X+\Delta)$-MMP, the Cartier index of any $\bQ$-Cartier Weil divisor is at most $(n^n N^{2n+1})!$.
\end{cor}

\begin{proof}
% Since $N(K_X+\Delta)$ is Cartier, 
By assumption, we have $1-\max \Coef(\Delta)\ge \frac{1}{N}$. Hence by Lemma \ref{lem:Vol explicit bdd D big} and our assumptions, we obtain $\hVol_X(\Delta)\ge N^{-2n-1}$. The result then follows from Lemmas \ref{lem:Vol MMP} and \ref{lem:index bound via global Vol}.
\end{proof}

By a similar argument, we can also bound the index explicitly in a general type MMP.

\begin{lem} \label{lem:vol explicit bdd K+D big}
Let $(X,\Delta)$ be a projective klt pair of dimension $n$ such that $K_X+\Delta$ is big. Then for any sequence of steps of a $(K_X+\Delta)$-MMP $\varphi\colon (X,\Delta)\dashrightarrow (Y,\Delta_Y)$ and any closed point $y\in Y$, we have
\[
\hvol(y,Y,\Delta_Y)\ge \alpha(X,\Delta;K_X+\Delta)^n\cdot \vol(K_X+\Delta).
\]
\end{lem}

\begin{proof}
Let $\alpha:=\alpha(X,\Delta;K_X+\Delta)$ and let $V_\bullet$ be the complete linear series of $K_X+\Delta$, i.e. $V_m=H^0(X,m(K_X+\Delta))$. Then for any positive integer $m$ and any $D\in \frac{1}{m}|V_m|$, the pair $(X,\Delta+\alpha D)$ is log canonical, and the $(K_X+\Delta)$-MMP is also a $(K_X+\Delta+\alpha D)$-MMP. It follows that $(Y,\Delta_Y+\alpha D_Y)$ is also log canonical, hence $\alpha(Y,\Delta; \varphi_* V_\bullet)\ge \alpha$. Note that $\vol(\varphi_* V_\bullet)=\vol(K_X+\Delta)$. The result then follows from Proposition \ref{prop:alpha-vol inequality linear series}.
\end{proof}

\begin{cor} \label{cor:index explicit bdd K+D big}
Let $(X,\Delta)$ be a projective log smooth klt pair of dimension $n$, and let $N$ be a positive integer. Assume that $N\Delta$ has integer coefficients, $\vol(K_X+\Delta)\ge \frac{1}{N}$, and $((K_X+\Delta)\cdot H^{n-1})\le N$ for some very ample divisor $H$ on $X$. Then in any $(K_X+\Delta)$-MMP, the Cartier index of any $\bQ$-Cartier Weil divisor is at most $(n^n N^{2n+1})!$.
\end{cor}

\begin{proof}
By Lemma \ref{lem:Izumi}, we have 
\[
\alpha(X,\Delta;K_X+\Delta)\ge \frac{1-\max \Coef(\Delta)}{((K_X+\Delta)\cdot H^{n-1})}\ge \frac{1}{N^2}
\]
as in the proof of Lemma \ref{lem:Vol explicit bdd D big}. We then conclude by Lemma \ref{lem:vol explicit bdd K+D big} and Theorem \ref{thm:index bound via volume}.
\end{proof}

\section{Boundedness of fibers} \label{sec:fiber bdd}

In this section, we improve the argument from the previous section to prove the boundedness of the fibers of the extremal contractions and flips in the MMP.

\begin{thm} \label{thm:fiber bdd}
Let $(X,\Delta)$ be a $\bQ$-factorial klt pair and $X\rightarrow T$ a projective morphism. Assume that $K_X+\Delta$ or $\Delta$ is big over $T$. Then there exists a projective family $\cW\to\cB$ over a finite type base $\cB$, such that in any sequence of a $(K_X+\Delta)$-MMP over $T$, every fiber of the extremal contractions or the flips is isomorphic to $
\cW_b$ for some $b\in \cB$.
\end{thm}

Recall that a projective morphism $f\colon (X,\Delta)\to Z$ is called an extremal contraction if it is the contraction of a $(K_X+\Delta)$-negative extremal ray. There are three types of extremal contractions: divisorial contractions, flipping contractions, and Mori fiber spaces. 

\subsection{Strategy} \label{ss:straregy for fiber bdd}

Our strategy towards Theorem \ref{thm:fiber bdd} is to first prove that the germs of the extremal contractions and the flips are bounded up to special degenerations, and then analyze the possible deformations to show the boundedness of their fibers. For the special boundedness step, we shall prove the following statement.

\begin{thm} \label{thm:fibration special bdd, no boundary}
Let $n\in \bN$ and $\varepsilon>0$. Let $\cS$ be the set of Fano type fibration germs $X\to Z\ni z$ with $\dim X=n$ such that $\hVol_X(\Delta)\ge \varepsilon$ for some $\bR$-divisor $0\le \Delta\sim_{\bR,Z} -K_X$. Then $\cS$ is bounded up to Fano type degenerations (Definition \ref{defn:special bdd}).
\end{thm}

Observe that the Fano type assumption implies that $\Delta\sim_{\bR,Z} -K_X$ is big over $Z$. As $Z$ is affine (by the definition of a fibration germ), the $\bR$-divisor $\Delta$ is in fact big, thus its log canonical volume $\hVol_X(\Delta)$ is well-defined.

Using similar arguments, we also prove the following more general result on log Calabi-Yau fibration germs, which is in turn a variant of \cite{HLQ-vol-ACC}*{Conjecture 8.9} and \cite{HX-CY-bdd}*{Theorem 1.3}.

\begin{thm} \label{thm:fibration special bdd}
Let $n\in \bN$ and $\varepsilon,\delta>0$. Let $\cS$ be the set of log Calabi-Yau fibration germs $(X,\Delta)\to Z\ni z$ such that 
\begin{enumerate}
    \item $\dim X=n$, and all the nonzero coefficients of $\Delta$ are at least $\delta$,
    \item $X$ is of Fano type over $Z$, and
    \item $\hVol_X(\Delta)\ge \varepsilon$ (Definition \ref{defn:Vol}).
\end{enumerate}
Then $\cS$ is log specially bounded (Definition \ref{defn:special bdd}).
\end{thm}

Section \ref{ss:rel cone}--\ref{ss:special bdd} are devoted to the proof of Theorems \ref{thm:fibration special bdd, no boundary} and \ref{thm:fibration special bdd}. Before we delve into this somewhat lengthy proof, we first explain the deformation step in the proof of Theorem \ref{thm:fiber bdd} (assuming Theorem \ref{thm:fibration special bdd, no boundary}). This is done through the following auxiliary result. 

\begin{lem} \label{lem:fiber of special bdd set}
Let $\cS$ be a set of Fano type fibration germs that is bounded up to Fano type degenerations. Then the set 
\[
\cT=\{f^{-1}(z)\,|\,f\colon X\to Z\ni z \mbox{ belongs to } \cS\}
\]
is bounded.
\end{lem}

\begin{proof}
By assumption, there exists a bounded family $\cS_0$ of Fano type fibrations such that any fibration germ $f\colon X\to Z\ni z$ in $\cS$ degenerates to some $f_0\colon X_0\to Z_0\ni z_0$ in $\cS_0$ by a Fano type test configuration $\widetilde{f}\colon \cX\to \cZ$. Since $\cS_0$ is bounded, there exist some positive integers $r$ and $N$ depending only on $\cS_0$, some very ample line bundle $L_0$ on $X_0$, and a surjective homomorphism 
\[
\cE_0:=\cO_{Z_0}^{\oplus N}\to f_{0*}\cO_{X_0}(L_0),
\]
such that the induced map 
\[
\mu\colon \Sym(\cE_0)\to \cR_0:=\oplus_{m\in\bN} f_{0*}\cO_{X_0}(mL_0)
\]
of graded $\cO_{Z_0}$-algebras is also surjective, and its kernel $\cG$ is generated by homogeneous elements of degree at most $r$. % the $\cO_{Z_0}$-algebra $\cR_0:=\oplus_{m\in\bN} f_{0*}\cO_{X_0}(mL_0)$ is generated by at most $N+1$ homogeneous elements of degree one, and the kernel $\cG$ of the multiplication map
% \[
% \mu\colon \Sym_{Z_0} (f_{0*}\cO_{X_0}(L_0)) \to \cR_0
% \]
% is generated by homogeneous elements of degree at most $r$. 
In particular, the linear system $|L_0|$ induces a projectively normal embedding $X_0\hookrightarrow Z_0\times \bP^{N-1}$ and $X_0$ is scheme-theoretically cut out by equations of degree at most $r$ in $Z_0\times \bP^N$.

% Suppose that $\widetilde{f}\colon \cX\to \cZ$ is a test configuration of Fano type that degenerates $f\colon X\to Z\ni z$ in $\cS$ to the above $f_0\colon X_0\to Z_0\ni z_0$. 

Since $\cX$ is of Fano type over $\cZ$, we have $R^1 \widetilde{f}_* \cO_{\cX}=R^2 \widetilde{f}_* \cO_{\cX}=0$ by Kawamata-Viehweg vanishing, thus from the long exact sequence associated to the exponential exact sequence $0\to \bZ\to \cO_\cX\xrightarrow{\mathrm{exp}}\cO^*_\cX\to 0$ we see that over some sufficiently small analytic neighborhood $U\subseteq \cZ$ of $z_0$ we have 
\[
\Pic(\widetilde{f}^{-1}(U))\cong H^2(\widetilde{f}^{-1}(U),\bZ)\cong H^2(\widetilde{f}^{-1}(z_0),\bZ)
\]
where the last isomorphism comes from a deformation retraction of $\widetilde{f}^{-1}(U)$ to $\widetilde{f}^{-1}(z_0)$. Similarly, as $\cX_0$ is of Fano type over $\cZ_0$, there exists some sufficiently small analytic neighborhood $V\subseteq Z_0$ of $z_0$ such that 
\[
\Pic(f_0^{-1}(V))\cong H^2(f_0^{-1}(z_0),\bZ)=H^2(\widetilde{f}^{-1}(z_0),\bZ).
\]
It follows that $L_0$ extends to a line bundle $\cL$ over some small analytic neighborhood $U$ of $z_0\in \cZ$. Possibly shrinking $U$, we may assume that $\cL$ is ample over $U$. 

By Kawamata-Viehweg vanishing, we have $R^1 \widetilde{f}_* \cL^m=0$ for all $m\in\bN$. Thus the long exact sequence associated to $0\to \cL^m\stackrel{t}{\to}\cL^m\to L_0^m\to 0$ (where $t$ is the linear coordinate on $\bA^1$) gives
\begin{equation} \label{eq:exact seq f_*}
    0\to \widetilde{f}_* \cL^m|_U \stackrel{t}{\to} \widetilde{f}_* \cL^m|_U \to f_{0*}\cO_{X_0}(mL_0)|_U\to 0
\end{equation}
for all $m\in\bN$. Since $\cL$ is ample over $U$, the $\cO_U$-algebra
$\cR:=\oplus_{m\in\bN} \widetilde{f}_* \cL^m|_U$ is finitely generated. Recall that $\cR_0$ is generated by at most $N$ homogeneous elements of degree $1$ by our choice of $L_0$. Combined with \eqref{eq:exact seq f_*} and Nakayama's lemma we deduce that $\cR$ is also generated (as an $\cO_U$-algebra) by at most $N$ homogeneous elements of degree $1$. More precisely, on the open set $U$ by \eqref{eq:exact seq f_*} we may lift the homomorphism $\cE_0|_U\to f_{0*}\cO_{X_0}(L_0)|_U$ to a homomorphism 
\[
\cE:=\cO_{U}^{\oplus N} \to \widetilde{f}_* \cL|_U,
\]
and for any positive integer $m$, after possibly shrinking $U$, the induced homomorphism 
\[
\widetilde{\mu}\colon \Sym(\cE)\to \cR
\]
is surjective in degree $\le m$ by Nakayama's lemma. Since $\cR$ is generated by finitely many homogeneous elements, we see that $\widetilde{\mu}$ is surjective after shrinking $U$. In particular,
% \[
% \widetilde{\mu}\colon \Sym_U(\widetilde{f}_* \cL) \to \cR
% \]
% is surjective and 
we have a closed embedding $\widetilde{f}^{-1}(U)\hookrightarrow U\times \bP^{N-1}$. Let $\cF=\ker(\widetilde{\mu})$ and recall $\cG=\ker(\mu)$. By construction, $\mu$ is the reduction of $\widetilde{\mu}$ modulo $t$. Note that $\cR$ is torsion free and hence flat over $\bA^1$. It follows that the exact sequence $0\to \cF\to \Sym(\cE)\to \cR\to 0$ remains exact modulo $t$ and hence $\cG=\cF/t\cF$. Therefore, as $\cG$ is generated by homogeneous elements of degree at most $r$, the same holds for $\cF$ by another application of Nakayama's lemma as before. Thus $\widetilde{f}^{-1}(U)$ is cut out by equations of degree at most $r$ in $U\times \bP^{N-1}$ and hence every fiber $\widetilde{f}^{-1}(u)$, where $u\in U$, is cut out by equations of degree at most $r$ in $\bP^{N-1}$. 

Since $\widetilde{f}\colon \cX\to \cZ$ is a test configuration of $f\colon X\to Z\ni z$, for any open neighborhood $U$ of $z_0\in \cZ$ there exists some $u\in U$ such that $\widetilde{f}^{-1}(u)$ is isomorphic to $f^{-1}(z)$. Therefore, every member of $\cT$ can be cut out by equations of degree at most $r$ in $\bP^{N-1}$ and hence $\cT$ is bounded.
\end{proof}

\begin{proof}[Proof of Theorem \ref{thm:fiber bdd} assuming Theorem \ref{thm:fibration special bdd, no boundary}]
% If $K_X+\Delta$ is big over $T$, then there exists some $0\le D\sim_{\bR,T} K_X+\Delta$ and some $0<\delta\ll 1$ such that $(X,\Delta+\delta D)$ is still klt. Since any $(K_X+\Delta)$-MMP is also a $(K_X+\Delta+\delta D)$-MMP, and $\Delta+\delta D$ is big over $T$, we reduce to the case when $\Delta$ is big over $T$. By adding to $\Delta$ the pullback of some ample divisor on $T$, we may further reduce to the case that $\Delta$ is big.

By Lemma \ref{lem: gt klt sim to big boundary}, there exists some big $\bR$-divisor $\Delta_0\ge 0$ on $X$ such that $(X,\Delta_0)$ is klt and every $(K_X+\Delta)$-MMP over $T$ is also a $(K_X+\Delta_0)$-MMP. Replacing $\Delta$ by $\Delta_0$, we see that it suffices to prove the theorem when $\Delta$ is big.

Suppose that $(X,\Delta)\dashrightarrow (Y,\Delta_Y)$ is a sequence of steps of a $(K_X+\Delta)$-MMP and $f\colon (Y,\Delta_Y)\to Z$ is an extremal contraction over $T$. By Lemma \ref{lem:Vol MMP}, we have $\hVol_Y(\Delta_Y)\ge \hVol_X(\Delta)>0$. Since $-(K_Y+\Delta_Y)$ is ample over $Z$, for any closed point $z\in Z$ we can choose some $0\le H\sim_{\bR,Z} -(K_Y+\Delta_Y)$ such that $(Y,\Delta_Y+H)\to Z\ni z$ is a log Calabi-Yau fibration germ (implicitly, here we replace $Z$ by an affine neighborhood of $z$ in order to talk about the fibration germ). By Lemma \ref{lem:Vol add ample boundary}, we also have $\hVol_Y(\Delta_Y+H)\ge \hVol_X(\Delta)$. By Theorem \ref{thm:fibration special bdd, no boundary}, the Fano type fibration germ $Y\to Z\ni z$ belongs to a set depending only on $
\dim X$ and $\hVol_X(\Delta)$ that is bounded up to Fano type degenerations. By Lemma \ref{lem:fiber of special bdd set}, we conclude that the fiber $f^{-1}(z)$ of the extremal contraction belongs to a bounded family depending only on $(X,\Delta)$.

Next suppose that $(Y,\Delta_Y)\to Z$ is a small contraction and $g\colon (Y',\Delta')\to Z$ is the flip. Let $H'$ be the strict transform of $H$ on $Y'$. Then by Lemma \ref{lem:Vol MMP} and the above discussion we have $\hVol_{Y'}(\Delta'+H')\ge \hVol_Y(\Delta_Y+H)\ge \hVol_X(\Delta)$. Thus the same argument above applies to the fibration $Y'\to Z\ni z$ and we conclude that the fiber $g^{-1}(z)$ of the flip also belongs to a bounded family depending only on $(X,\Delta)$. 
\end{proof}

\subsection{Relative cone construction} \label{ss:rel cone}

We now turn to the proofs of Theorems \ref{thm:fibration special bdd, no boundary} and \ref{thm:fibration special bdd}. We first consider a special boundedness result for log Fano fibrations with a fixed finite set of rational coefficients. 

\begin{thm} \label{thm:fibration special bdd, finite rational coef}
Let $n\in \bN$, $\varepsilon>0$, and let $I\subseteq [0,1]\cap \bQ$ be a finite set. Let $\cS$ be the set of log Fano fibration germs $(X,\Delta)\to Z\ni z$ with $\dim X=n$, $\Coef(\Delta)\subseteq I$ such that $\hVol_{X,\Delta}(H)\ge \varepsilon$ for some $\bR$-divisor $H\sim_{\bR,Z} -(K_X+\Delta)$. Then $\cS$ is specially bounded (Definition \ref{defn:special bdd}).
\end{thm}

When $X=Z$, this essentially follows from (by Lemma \ref{lem:local vs global vol}) the special boundedness of klt singularities proved in \cite{XZ-local-bdd}. Combined with the argument from the previous subsection, Theorem \ref{thm:fibration special bdd, finite rational coef} already implies the rational coefficient case of Theorem \ref{thm:fiber bdd} for fibers of the extremal contractions.

The main idea of the proof is to consider the relative cone of the log Fano fibration and reduce to \cite{XZ-local-bdd} after analyzing the local volume of the cone singularity. Before we do so, we make a digression to discuss some generality of the relative cone construction. Many results in this subsection are parallel to those of \cite{Kol13}*{Section 3.1}, which deals with cones over projective pairs.

We start with some general setup and introduce some notation that will be fixed throughout this subsection. Let $f\colon X\to Z$ be a fibration and let $L$ be an $f$-ample $\bQ$-Cartier Weil divisor on $X$. Let 
\[
\cR_m := f_*\cO_X(mL)\quad \mbox{and}\quad \cR:=\oplus_{m\in\bN} \cR_m.
\]
We are interested in the relative (orbifold) cone $Y:=\Spec_Z(\cR)$. % Before we discuss some properties of the relative cone construction, let us introduce some notation that will be fixed throughout this subsection. First, 
Next, let 
\[
\tY:=\Spec_X \left(\oplus_{m\in\bN} \cO_X(mL) \right)
\]
and let $\pi\colon \tY\to X$ be the induced morphism. It has a section $E\subseteq \tY$ defined by the ideal $\oplus_{m\in \mathbb N^+} \cO_X(mL)$. The morphism $\tY\setminus E\to X$ is a Seifert $\bG_m$-bundle (in the sense of \cite{Kol-Seifert-bundle}) and $E$ can be identified with the zero section. In particular, the divisor $E$ is $\bQ$-Cartier and $E\cong X$. Locally on $Z$ we have $\cR = H^0(\cO_{\tY})$ (in other words $(f\circ \pi)_* \cO_{\tY} = \cR$). Hence we have an induced morphism $\varphi\colon \tY\to Y$ over $Z$ which is an isomorphism away from the zero section $E$ (by the ampleness of $L$ over $Z$), and $\varphi|_E = f$. Let $\Gamma:=\varphi(E)\cong Z$. It gives a section $\sigma\colon Z\to Y$ of the projection $g\colon Y\to Z$. The situation is summarized by the following diagram.

\begin{equation} \label{eq:cone diagram}
\begin{tikzcd}
    E \arrow[hook]{r} \arrow{d} & \tY \arrow[r,"\pi"] \arrow[d,"\varphi"] & X \arrow[d,"f"] \\
    \Gamma \arrow[hook, r] & Y \arrow[r,"g"] & Z \arrow[l, bend left, "\sigma"]
\end{tikzcd}
\end{equation}

Note that if $f\colon X\to Z$ is birational, then $\varphi$ is small and induces an isomorphism $\varphi_*\colon \Cl(\tY)\cong \Cl(Y)$. Otherwise $E$ is the only exceptional divisor of $\varphi$ and we get an isomorphism $\varphi_*\colon \Cl(\tY)/[E]\cong \Cl(Y)$. We also have a natural pullback map $\pi^*\colon \Cl(X)\to \Cl(\tY)$ since $\pi$ is equidimensional. 

\begin{lem} \label{lem:class group}
The induced map $\pi^*\colon \Cl(X)\to \Cl(\tY)$ is an isomorphism. Moreover, we have $-E\sim \pi^*L$ and $K_{\tY}+E\sim \pi^*K_X$.
\end{lem}

\begin{proof}
Over the smooth locus $U$ of $X$ the divisor $L$ is Cartier, and $\pi^{-1}(U)$ is an $\bA^1$-bundle over $U$, hence $\pi^*$ induces an isomorphism $\Cl(\pi^{-1}(U))\cong \Cl(U)$. Since $\pi$ is equidimensional, the complement of $\pi^{-1}U$ in $\tY$ has codimension at least $2$, thus $\Cl(X)\cong \Cl(U)$ and $\Cl(\tY)\cong \Cl(\pi^{-1}(U))$. It follows that $\pi^*\colon \Cl(X)\to \Cl(\tY)$ is also an isomorphism.

Any rational section of $\cO_X(L)\subseteq \pi_*\cO_{\tY}$ defines a rational function on $\tY$ and induces the linear equivalence relation $\pi^*L+E\sim 0$. If $z$ is the affine coordinate of the fiber of $\pi^{-1}(U)\to U$, then $\frac{dz}{z}$ is a nowhere zero section of $K_{\pi^{-1}(U)/U}+E$. This gives $K_{\pi^{-1}(U)/U}+E\sim 0$ and hence $K_{\tY}+E\sim \pi^*K_X$ as $\pi^{-1}(U)$ (resp. $U$) is a big open subset of $\tY$ (resp. $X$).
\end{proof}

For the next several lemmas, We say that the fibration $f\colon X\to Z$ is \emph{of fiber type} if $\dim X>\dim Z$.

% \begin{equation} \label{eq:cone diagram}
% \begin{tikzcd}
%     E \arrow[hook]{r} \arrow{d} & \tY \arrow[r,"\pi"] \arrow[d,"\varphi"] & X \arrow[d,"f"] \\
%     \Gamma \arrow[hook, r] & Y \arrow[r,"g"] & Z \arrow[l, bend left, "\sigma"]
% \end{tikzcd}
% \end{equation}

\begin{lem} \label{lem:Cartier criterion}
Let $G$ be a Weil divisor on $X$. Then $\varphi_*\pi^*G$ is Cartier if and only if either $f$ is birational and $G\sim_Z 0$, or $f$ is of fiber type and $G\sim_Z mL$ for some $m\in \bZ$. Moreover, $\pi^*G\sim_Y 0$ if and only if $G\sim_Z 0$.
\end{lem}

\begin{proof}
First consider the case when $f$ is birational. If $G=f^*G_0$ for some Cartier divisor $G_0$ on $Z$, then $\pi^*G\sim_Y 0$ and $\varphi_*\pi^*G = g^*G_0$ is also Cartier. Conversely, if $G':=\varphi_*\pi^*G$ is Cartier then $\pi^*G = \varphi^* G'$ as $\varphi$ is small. Restricting to the zero section $E\cong X$ gives $G = f^*\sigma^*G'$, hence $G\sim_Z 0$.

Next assume that $f$ is of fiber type. The proof is similar. If $G-mL=f^*G_0$ for some Cartier divisor $G_0$, then by Lemma \ref{lem:class group} we have $\varphi_*\pi^*G = \varphi_*(\pi^*G+mE)\sim \varphi_*\pi^*(G-mL) = g^*G_0$, hence $\varphi_*\pi^*G$ is also Cartier. Conversely, if $G':=\varphi_*\pi^*G$ is Cartier then $\pi^*G +mE = \varphi^* G'$ for some integer $m$. If in addition $\pi^* G\sim_Y 0$ then $m=0$. Restricting to the zero section $E\cong X$ and noting that $-E|_E\cong L$ by Lemma \ref{lem:class group}, we obtain $G-mL \sim f^*\sigma^*G'$, hence $G\sim_Z mL$.
\end{proof}

Next let $\Delta\geq 0$ be an $\bR$-divisor on $X$ and consider the $\bR$-divisor $D:=\varphi_*\pi^*\Delta\geq 0$ on $Y$. Geometrically, it is the relative cone over $\Delta$.

\begin{lem} \label{lem:cone klt birational}
Assume that $f$ is birational and $0< \lambda\le 1$ (resp. $0\le \lambda\le 1$). Then the pair $(Y,D+(1-\lambda) \Gamma)$ is klt (resp. log canonical) if and only if $K_X+\Delta+\lambda L\sim_{\bR,Z} 0$ and $(X,\Delta)$ is klt (resp. log canonical).
\end{lem}

\begin{proof}
We only treat the klt case since the log canonical version is similar. By Lemma \ref{lem:class group}, we have
\[
K_Y+D+(1-\lambda) \Gamma \sim_\bR \varphi_*(K_{\tY}+\pi^*\Delta+(1-\lambda) E) \sim_\bR \varphi_*\pi^*(K_X+\Delta+\lambda L).
\]
Thus by Lemma \ref{lem:Cartier criterion}, we see that $K_Y+D+(1-\lambda) \Gamma$ is $\bR$-Cartier if and only if $K_X+\Delta+\lambda L\sim_{\bR,Z} 0$, in which case $(\tY,\pi^*\Delta+(1-\lambda) E)$ is the crepant pullback of $(Y,D+(1-\lambda) \Gamma)$.

By \cite{Kol-Seifert-bundle}*{Paragraphs 40-42}, if $(Y,D+(1-\lambda) \Gamma)$ is klt, then so is $(X,\Delta)$ since the latter is the orbifold quotient of $(Y\setminus \Gamma,D|_{Y\setminus \Gamma})$ by the $\bG_m$-action. Conversely, we show that  $(Y,D+(1-\lambda) \Gamma)$ is klt if $(X,\Delta)$ is klt. To see this, it suffices to show that the crepant pullback $(\tY,\pi^*\Delta+(1-\lambda) E)$ is klt. Since 
$(K_{\tY}+\pi^*\Delta+E)|_E\cong K_X+\Delta$ (where we identify $E$ with $X$) and $(X,\Delta)$ is klt, by inversion of adjunction we see that $(\tY,\pi^*\Delta+E)$ is plt in a neighborhood of $E$. Since the natural $\bG_m$-action on $\tY$ preserves the pair $(\tY,\pi^*\Delta+E)$, every non-klt center intersects the $\bG_m$-fixed locus $E$; this implies that $(\tY,\pi^*\Delta+E)$ is plt everywhere. As $0< \lambda\le 1$, we see that $(\tY,\pi^*\Delta+(1-\lambda) E)$ is klt.
\end{proof}

By essentially the same argument, we also have the following.

\begin{lem} \label{lem:cone klt fiber type}
Assume that $f$ is of fiber type. Then the pair $(Y,D)$ is klt (resp. log canonical) if and only if $K_X+\Delta+\lambda L\sim_{\bQ,Z} 0$ for some $\lambda> 0$ (resp. $\lambda\ge 0$) and $(X,\Delta)$ is klt (resp. log canonical). In such case we have $A_{Y,D}(E)=\lambda$.
\end{lem}

\begin{proof}
Since $K_Y+D=\varphi_*\pi^*(K_X+\Delta)$, the first part follows from Lemma \ref{lem:Cartier criterion} as in the proof of Lemma \ref{lem:cone klt fiber type}. It remains to show that $K_X+\Delta+\lambda L\sim_{\bR,Z} 0$ where $\lambda:=A_{Y,D}(E)$. But this again follows from Lemma \ref{lem:Cartier criterion} as $\pi^*(K_X+\Delta+\lambda L)\sim_\bR K_{\tY}+\pi^*\Delta+(1-\lambda)E = \varphi^*(K_Y+D) \sim_{\bR,Y} 0$.
\end{proof}

\subsection{Volume estimate} 

We next apply the above cone construction to log Fano fibrations and prove the key local volume estimate towards the proof of Theorem \ref{thm:fibration special bdd, finite rational coef}. 

Let $f\colon (X,\Delta)\to Z\ni z$ be a log Fano fibration germ. We assume that the coefficients of $\Delta$ are rational. Let $r$ be a positive integer such that $r\Delta$ has integer coefficients, let 
\[
\cR=\oplus_{m\in\bN} H^0(X,-mr(K_X+\Delta)),
\]
and let $Y=\Spec (\cR)$ be the relative cone with respect to the Weil divisor $L:=-r(K_X+\Delta)$. We keep the notation from the diagram \eqref{eq:cone diagram}. Let 
\begin{equation} \label{eq:boundary in relative cone}
D=
\begin{cases}
    \varphi_*\pi^*\Delta+\left(1-\frac{1}{r}\right)\Gamma, & \quad\mbox{if } f \mbox{ is birational}, \\
    \varphi_*\pi^*\Delta, & \quad\mbox{if } f \mbox{ is of fiber type}.
\end{cases}
\end{equation}
Note that $\Gamma$ is a divisor on $Y$ if and only if $f$ is birational.
By Lemmas \ref{lem:cone klt birational} and \ref{lem:cone klt fiber type}, the pair $(Y,D)$ is klt. Let $y=\sigma(z)\in Y$. The main result of this subsection is the following.

\begin{thm} \label{thm:cone volume}
Under the above setup, we have
\[
\hvol(y,Y,D)\ge \frac{\hVol_{X,\Delta}(H)}{r}
\]
for any $\bR$-divisor $H\sim_{\bR,Z} -(K_X+\Delta)$.
\end{thm}

\begin{proof}
By Lemma \ref{lem:Vol lsc}, and since $\Delta$ has rational coefficient, it is enough to prove the statement when $H$ is a $\bQ$-divisor. Since the question is local on $Z$, we may also shrink $Z$ around $z$ and assume that $H\sim_\bQ -(K_X+\Delta)$. By definition, it suffices to show that 
\[
\hvol(y,Y,D)\ge \frac{(1-\varepsilon)^n \vol(V_\bullet)}{r}
\]
for any $0<\varepsilon\ll 1$ and any admissible graded linear series $V_\bullet$ of $H$. Our plan is to find a sequence of subspaces $W_m\subseteq \cO_{Y,y}$ of increasing dimensions such that 
\[
\min\{\lct_y(s)\,|\,0\neq s\in W_m\}
\]
are reasonably bounded from below and apply Lemma \ref{lem:vol estimate}. Here we use the shorthand notation $\lct_y(s)$ for $\lct_y(Y,D;\{s=0\})$ where $0\neq s\in \cO_{Y,y}$. In the simplest case, if $-r(K_X+\Delta)$ is Cartier and $H=-(K_X+\Delta)$, then each $V_{mr}$ ($m\in\bN$) is already a subspace of $\cR_m\subseteq \cO_{Y,y}$, and we can take $W_m$ to be the span of $V_r,V_{2r},\cdots,V_{mr}$.

In general, let $k$ be a sufficiently divisible positive integer such that $kH$ and $k(K_X+\Delta)$ are both Cartier, $kH\sim -k(K_X+\Delta)$, and $V_{km}\neq 0$ for all $m\in\bN$. Let $V'_\bullet$ be the graded linear series of $kH$ given by $V'_m = V_{km}$. Let $\alpha = \frac{1-\varepsilon}{k}$. Since $H\sim_\bQ -(K_X+\Delta)$ is ample over $Z$, so is
\[
-mr(K_X+\Delta)- \lfloor \alpha mr\rfloor \cdot kH \sim -(mr-\lfloor \alpha mr\rfloor \cdot k)(K_X+\Delta)
\]
for any $m>0$; moreover, for any sufficiently positive integer $m$ (say, when $m\ge m_0$) we have $mr-\lfloor \alpha mr\rfloor \cdot k \ge \varepsilon mr\gg 0$ and thus by Lemma \ref{lem:Bertini ample Q-div} we can find some 
\[
u_m\in H^0(X,-mr(K_X+\Delta)- \lfloor \alpha mr\rfloor \cdot kH)
\]
such that $(X,\Delta+\frac{1}{\varepsilon mr}\{u_m=0\})$ is log canonical.

Multiplication by $u_m$ takes $V'_{\lfloor \alpha mr\rfloor}$ into $\cR_{m} = H^0(X,-mr(K_X+\Delta)) \subseteq \cO_{Y,y}$ and we take $W_m\subseteq \cO_{Y,y}$ to be the span of $u_j\cdot V'_{\lfloor \alpha jr\rfloor}$ for $j=m_0,\dots,m$. We claim that
\begin{equation} \label{eq:lct}
    m\cdot \lct_y(s)\ge \frac{1}{r}
\end{equation}
for all $0\neq s\in W_m$. By the following Lemma \ref{lem:toric degeneration}, since $W_m$ is invariant under the $\bG_m$-action on $Y$ and has a weight decomposition $W_m = \oplus_{j=m_0}^m u_j\cdot V'_{\lfloor \alpha jr\rfloor}$, it suffices to prove the claim when $s =  u_j\cdot w$ for some $m_0\le j\le m$ and some $w\in V'_{\lfloor \alpha jr\rfloor}$.

Since $V_\bullet$ is admissible, we know that the pair $(X,\Delta+\frac{1}{\alpha jkr}\{w=0\}) = (X,\Delta+\frac{1}{(1-\varepsilon)jr}\{w=0\})$ is log canonical. By construction $(X,\Delta+\frac{1}{\varepsilon jr}\{u_j=0\})$ is also log canonical. By taking convex combination, we deduce that
\[
\left(X,\Delta+(1-\varepsilon)\cdot \frac{1}{(1-\varepsilon)jr}\{w=0\}+\varepsilon\cdot \frac{1}{\varepsilon jr}\{u_j=0\}\right)
\]
is log canonical. In other words, the pair 
\[
(X,\Delta+\Delta_0):=\left(X,\Delta+\frac{1}{jr}\{u_j w=0\}\right)
\]
is log canonical. On the other hand, we have $K_X+\Delta+\Delta_0\sim_{\bQ,Z} 0$ as $u_j w \in H^0(X,-jr(K_X+\Delta))$. If $D_0:=\frac{1}{jr}\{s=0\}$ is the corresponding divisor on $Y$, then we also have
\[
D_0 =
\begin{cases}
    \varphi_* \pi^*\Delta_0 + \frac{1}{r}\Gamma & \quad \mbox{if } f \mbox{ is birational}, \\
    \varphi_*\pi^*\Delta_0 & \quad\mbox{if } f \mbox{ is of fiber type}.
\end{cases}
\]
By Lemma \ref{lem:cone klt birational} and \ref{lem:cone klt fiber type}, these imply that $(Y,D+D_0)$ is log canonical at $y$, which gives $\lct_y(s)\ge \frac{1}{jr}\ge \frac{1}{mr}$ and hence the claim \eqref{eq:lct}.

By Lemma \ref{lem:vol=lim} we have
\[
\dim W_m = \sum_{j=m_0}^m \dim V'_{\lfloor \alpha jr\rfloor} = \sum_{j=m_0}^m \dim V_{\lfloor \alpha jr\rfloor\cdot k} = (\alpha k r)^n \vol(V_\bullet)\frac{m^{n+1}}{(n+1)!} + o(m^{n+1}).
\]
Combined with \eqref{eq:lct} and Lemma \ref{lem:vol estimate} we see that
\[
\hvol(y,Y,D)\ge \frac{c_1}{r^{n+1}} 
\]
for any $c_1<(\alpha k r)^n \vol(V_\bullet)$. It follows that
\[
\hvol(y,Y,D)\ge \frac{(\alpha k)^n \cdot \vol(V_\bullet)}{r} = \frac{(1-\varepsilon)^n \vol(V_\bullet)}{r}.
\]
This concludes the proof.
\end{proof}

The following lemma is used in the above proof. We keep the same notation $\lct_y(s)$ for $\lct_y(Y,D;\{s=0\})$

\begin{lem} \label{lem:toric degeneration}
Let $y\in (Y,D)$ be a klt singularity with a torus $\bT$-action and let $V\subseteq \cO_{Y,y}$ be a finite dimensional $\bT$-invariant subspace. Assume that for some constant $c\ge 0$ we have $\lct_y(s)\ge c$ for all nonzero $\bT$-eigenvectors $s\in V$. Then $\lct_y(s)\ge c$ for all nonzero $s\in V$.
\end{lem}

\begin{proof}
For any $0\neq s\in V$, there exists some one-parameter subgroup $\rho\colon\bG_m\to \bT$ such that the flat limit $D_0$ of $\rho(t)\cdot D$, where $D=\{s=0\}$ and when $t\to 0$, is a $\bT$-invariant divisor. Indeed, if $s=\sum_{i=1}^m s_i$ is the weight decomposition where $s_i$ has weight $\alpha_i\in \Hom(\bT,\bG_m)$, and choose 
$\rho$ such that $\langle\rho,\alpha_1\rangle< \cdots<\langle\rho,\alpha_m\rangle$, then the flat limit is $\{s_1=0\}$. By assumption, we have $\lct_y(D_0)\ge c$, hence by the lower semi-continuity of log canonical threshold \cite{Amb-lct-lsc}*{Corollary 1.10}, we get $\lct_y (D)\ge c$ as well.
\end{proof}

By \cite{XZ-local-bdd} and the stable degeneration theorem (see Theorem \ref{thm:special bdd klt singularity}), we know that klt singularities whose local volumes are bounded away from zero are bounded up to special degeneration. Combined with Theorem \ref{thm:cone volume}, we are now ready to prove Theorem \ref{thm:fibration special bdd, finite rational coef}.

\begin{proof}[Proof of Theorem \ref{thm:fibration special bdd, finite rational coef}]
Choose some positive integer $r$ such that $I\subseteq \frac{1}{r}\bN$. Let $f\colon (X,\Delta)\to Z\ni z$ be a log Fano fibration germ in $\cS$. Let $y\in (Y,D)$ be the relative cone over $(X,\Delta)$ with respect to the ample Weil divisor $L=-r(K_X+\Delta)$, where the boundary $D$ is defined as in \eqref{eq:boundary in relative cone}. By Theorem \ref{thm:cone volume} and our assumption, we have $\hvol(y,Y,D)\ge \frac{\varepsilon}{r}$. Note that $y\in (Y,D)$ carries a $\bG_m$-action by the cone construction. 

By Theorem \ref{thm:special bdd klt singularity}, there exists an $\bR$-Gorenstein family $\cB\subseteq (\cY,\cD)\to \cB$ of klt singularities, depending only on $n,I$ and $\varepsilon$, with the action of an algebraic torus $\bT$ over $\cB$, such that $y\in (Y,D)$ admits a $\bG_m$-equivariant special degeneration to $b\in (\cY_b,\cD_b)$ for some $b\in \cB$, and by specializing the $\bG_m$-action to $b\in (\cY_b,\cD_b)$, we get a one-parameter subgroup of $\bT_b$. 

By base change, we shall assume that on each of the connected components $\cB_1,\dots,\cB_k$ of $\cB$, we have $\bT=\bT_i\times \cB$ for some torus $\bT_i$. By Lemma \ref{lem:weight monoid constant}, the weight monoid $S_i$ is constant over each $\cB_i$. We claim that for each $i\in \{1,\dots,k\}$, there are only finitely many one-parameter subgroups of $\bT_i$ that are possibly induced by the above construction. To see this, suppose that $b\in \cB_i$ (so we may identify $\bT_b$ with $\bT_i$) and let $E$ be the divisor over $Y$ as in the diagram \eqref{eq:cone diagram}; that is, the zero section of the corresponding Seifert $\bG_m$-bundle over $Z$. Then we have $\wt_\rho = \ord_E$ as valuations on $Y$, where we denote by $\rho\colon \bG_m\to \Aut(Y,D)$ the one-parameter subgroup coming from the cone construction. We can identify the induced one-parameter subgroup of $\Aut(\cY_b,\cD_b)$ with some primitive element of $S_i^\vee$, which we still denote by $\rho$. By \eqref{eq:boundary in relative cone} and Lemma \ref{lem:cone klt fiber type}, we know that $A_{Y,D}(E)=\frac{1}{r}$, thus by Lemma \ref{lem:log discrep constant T-family} we deduce that $A_{\cY,\cD}(\wt_\rho)=A_{\cY_b,\cD_b}(\wt_\rho) = \frac{1}{r}$. Since the log discrepancy function $\xi\mapsto A_{\cY,\cD}(\wt_\xi)$ is positive (as $(\cY,\cD)$ is klt) and linear on $S_i^\vee$ (by Lemma \ref{lem:log discrep linear T-variety}), we see that there are at most finitely many $\rho\in S_i^\vee$ such that $A_{\cY,\cD}(\wt_\rho) = \frac{1}{r}$. This proves the claim.

Suppose that $y\in (Y,D)$ specially and $\bG_m$-equivariantly degenerates to $b\in(\cY_b,\cD_b)$. By the following Lemma \ref{lem:cone special tc induce fibration special tc}, it induces a special degeneration of $(X,\Delta)\to Z\ni z$ to the log Fano fibration germ $(\cY_b\setminus \cY_b^{\bG_m},\cD_b\setminus \cD_b^{\bG_m})/\bG_m\to \cY_b/\!\!/\bG_m \ni z_b$, where $z_b$ is the image of $b\in \cY_b$ in the GIT quotient. Thus it remains to show that latter form an $\bR$-Gorenstein family of log Fano fibrations over a finite type base. This is clear since there are only finitely many one-parameter subgroups of $\bT_i$ to consider and $\cB$ is of finite type.
\end{proof}

\begin{lem} \label{lem:cone special tc induce fibration special tc}
Let $(X,\Delta)\to Z\ni z$ be a log Fano fibration germ and let $y\in (Y,D)$ be the relative cone with respect to $L=-r(K_X+\Delta)$ as above for some $r\in \bN$ such that $r\Delta$ has integer coefficients. Let $(\cY=\Spec (\cR),\Delta_{\cY})\to \bA^1$ be a $\bG_m$-equivariant special test configuration of the singularity $y\in (Y,D)$. Then it induces a special test configuration 
\[
(\cX,\cD)\to \cZ
\]
of $(X,\Delta)\to Z\ni z$ such that $(\cY,\Delta_{\cY})$ is the relative cone of $(\cX,\cD)$ over $\cZ$ with respect to $\cL:=-r(K_{\cX}+\cD)$.
\end{lem}

\begin{proof}
Since $(\cY,\Delta_{\cY})\to \bA^1$ is a $\bG_m$-equivariant test configuration, we have a weight decomposition $\cR=\oplus_{m\in\bN} \cR_m$ with respect to the fiberwise $\bG_m$-action. Note that there are no negative weights by Lemma \ref{lem:weight monoid constant}. Let $\cX:=\Proj(\cR)$ and $\cZ:=\Spec (\cR_0)$. Then the natural morphism $\cX\to \cZ$ over $\bA^1$ is a test configuration of the fibration $X\to Z$. Since $\cY\to \bA^1$ is a test configuration of the singularity $y\in Y$, the closure $C$ of $\{y\}\times (\bA^1\setminus \{0\})$ in $\cY$ is a section of the morphism $\cY\to \bA^1$. The latter factors through $\cZ$, hence the closure of $\{z\}\times (\bA^1\setminus \{0\})$ in $\cZ$ coincides with the image of $C$ and is a section of $\cZ\to \bA^1$. This implies that $\cZ\to \bA^1$ is also a test configuration of the singularity $z\in Z$ and therefore $\cX\to \cZ$ is a test configuration of the fibration germ $X\to Z\ni z$.

Let $\cD$ be the closure of $\Delta\times (\bA^1\setminus \{0\})$ in $\cX$. We next show that the Weil divisor $\cL=-r(K_{\cX}+\cD)$ is ample $\bQ$-Cartier and $(\cY,\Delta_{\cY})$ is the relative orbifold cone of $(\cX,\cD)$ over $\cZ$ respect to $\cL$. Let $\cW=\cY^{\bG_m}$ be the fixed locus of the fiberwise $\bG_m$-action. Then $\cY\setminus \cW\to \cX$ is a Seifert $\bG_m$-bundle. We claim that it has reduced (hence smooth) fibers over all codimension one points of $\cX$. In fact, away from $\cX_0$ the Seifert $\bG_m$-bundle is given by the Weil divisor $-r(K_{\cX}+\cD)$, hence the codimension one fibers are reduced by \cite{Kol-Seifert-bundle}*{Theorem 7}. The fiber over the generic point of $\cX_0$ is also reduced since $(\cY,\Delta_{\cY})\to \bA^1$ is a special test configuration and in particular $\cY_0$ is reduced. This proves the claim. By \cite{Kol-Seifert-bundle}*{Theorem 7}, we deduce that the Seifert $\bG_m$-bundle $\cY\setminus \cW\to \cX$ is given by some Weil divisor $\cL$ on $\cX$. We necessarily have $\cL=\cO_{\cX}(1)$ and thus $\cL$ is $\bQ$-Cartier and ample. By construction, $\cL|_{\cX\setminus \cX_0} \sim -r(K_{\cX}+\cD)|_{\cX\setminus \cX_0}$. Since $\cY_0$ is integral, so is $\cX_0$. As $\cX_0\sim 0$, this implies that $\cL\sim -r(K_{\cX}+\cD)$ and therefore $-r(K_{\cX}+\cD)$ is $\bQ$-Cartier and ample. Replacing $\cL$ by $-r(K_{\cX}+\cD)$, we conclude that $(\cY,\Delta_{\cY})$ is the relative orbifold cone of $(\cX,\cD)$ over $\cZ$ with respect to $\cL=-r(K_{\cX}+\cD)$. % From the construction, it is clear that $\Delta_{\cY}$ is the cone over $\cD$. 
Since $(\cY,\cY_0+\Delta_{\cY})$ is plt by assumption, the base $(\cX,\cX_0+\cD)$ is also plt by \cite{Kol-Seifert-bundle}*{Paragraphs 40-42} as in the proof of Lemma \ref{lem:cone klt birational}. This implies that $(\cX,\cD)\to \cZ$ is a special test configuration and we finish the proof.
\end{proof}

\subsection{Finiteness of models}

% Next, we prove a boundary-free version of Theorem \ref{thm:fibration special bdd}. This is enough for applications towards bounding the fibers of MMP steps.

% \begin{thm}  \label{thm:fibration special bdd, no boundary}
% Let $n\in \bN$ and $\varepsilon>0$. Let $\cS$ be the set of Fano type fibration germs $X\to Z\ni z$ with $\dim X=n$ such that $\hVol_X(\Delta)\ge \varepsilon$ for some $\bR$-divisor $0\le \Delta\sim_{\bR,Z} -K_X$. Then $\cS$ is bounded up to Fano type degenerations (Definition \ref{defn:special bdd}).
% \end{thm}

To go from the special boundedness of log Fano fibration germs in Theorem \ref{thm:fibration special bdd, finite rational coef} to the log special boundedness of Fano type fibration germs in Theorems \ref{thm:fibration special bdd, no boundary} and \ref{thm:fibration special bdd}, we employ the following proof strategy. For simplicity, suppose we are in the setting of Theorem \ref{thm:fibration special bdd, no boundary}. % Let $(X,\Delta)\to Z\ni z$ be a log Calabi-Yau fibration germ with $\hVol_X(\Delta)\ge \varepsilon$. 
Let $X\to Z\ni z$ be a Fano type fibration germ. Then the ample model $X\dashrightarrow X'$ of $-K_X$ exists over $Z$. Assume that $\hVol_X(\Delta)\ge \varepsilon$ for some $\bR$-divisor $0\le \Delta\sim_{\bR,Z} -K_X$ as in Theorem \ref{thm:fibration special bdd, no boundary}. From Lemma \ref{lem:Vol MMP}, it is not hard to see that $\hVol_{X'}(\Delta')\ge \varepsilon$ where $\Delta'\sim_{\bR,Z} -K_{X'}$ is the strict transform of $\Delta$. Since $X'\to Z\ni z$ is a log Fano fibration germ, it belongs to a specially bounded set by Theorem \ref{thm:fibration special bdd, finite rational coef}. Next we observe that $X$ can be recovered from $X'$ up to finite ambiguity. In fact, all the exceptional divisors of the birational map $X\dashrightarrow X'$ have non-positive discrepancies over $X'$. Since $X'$ is log Fano over $Z$, there are only finitely many birational models $X$ with this property. With this in mind, we prove a family version of such finiteness of models (Theorem \ref{thm:bir model bdd}). As a result, if $X'\to Z$ belongs to a bounded (resp. specially bounded) set, then so is $X\to Z$ (Theorem \ref{thm:bir model log special bdd}).

Turning to more details, we first prove in this subsection the following statement that roughly says a set of Fano type fibrations is bounded if their ample models are bounded. The stronger version, where we replace boundedness by special boundedness, will be established in the next subsection.

\begin{thm} \label{thm:bir model bdd}
Let $\cS$ be a bounded set of log Fano fibrations (Definition \ref{defn:Fano type fibration bdd}). Then there exists a bounded set $\cT$ of Fano type fibrations such that for any $(X,\Delta)\to Z$ in $\cS$ and any birational map $f\colon X'\dashrightarrow X$ over $Z$ such that $f^*(K_X+\Delta)=K_{X'}+\Delta'$ for some $\Delta'\ge 0$, we have $(X',\Delta')\to Z$ belongs to $\cT$.
\end{thm}

This can be viewed as a local version of \cite{HX-CY-bdd}*{Section 2.1} and we essentially mimic the strategy of {\it loc. cit.}. The first step is to prove a local version of \cite{dFH-deform-canonical-pair}*{Lemma 6.6}. Namely, given a family of Fano type fibrations, the relative Picard group is a constant local system possibly after a base change (Corollary \ref{cor:rel Pic Fano fibration}). For this we need a few auxiliary results. For any fibration $f\colon X\to Y$, we let $N^1(X/Y):=\bQ\otimes (\Pic(X)/f^*\Pic(Y))$, and let $H_j(X/Y, \bQ) \subseteq H_j(X, \bQ)$ be the subspace generated by the images of 
$H_j(X_y, \bQ) \to H_j(X, \bQ)$ where $y\in Y$ ({\it cf.} \cite{KM-flip-classification}*{Definition (12.1.2)}). We also let $H_i(X/Y,\bQ)_{\mathrm{alg}}\subseteq H_i(X/Y,\bQ)$ be the subspace generated by the algebraic cycles in the fiber of $f$.

\begin{lem} \label{lem:N^1 dual to H_2}
Let $f\colon X\to Y$ be a fibration. Assume that $R^1 f_*\cO_X=0$ and $X$ has rational singularities. Then the intersection pairing $N^1(X/Y)\times H_2(X/Y,\bQ)\to \bQ$ induces an inclusion $\iota\colon N^1(X/Y)\hookrightarrow H_2(X/Y,\bQ)_{\mathrm{alg}}^\vee$. In particular, $\dim N^1(X/Y)<+\infty$. If in addition $X$ is smooth, then $\iota$ is an isomorphism. 
\end{lem}

\begin{proof}
The exponential exact sequence gives rise to a homomorphism $\Pic(X)\to H^2(X,\bZ)$ and hence a natural pairing $\Pic(X)_\bQ \times H_2(X/Y,\bQ)\to H^2(X,\bQ)\times H_2(X,\bQ)\to \bQ$. It descends to $N^1(X/Y)$ as $\langle f^*\Pic(Y),H_2(X/Y,\bQ)\rangle = 0$ by definition. By \cite{KM-flip-classification}*{Proposition (12.1.4)}, we see that the induced map $\iota$ is an inclusion. Suppose that $X$ is smooth and let $\oX$ be a smooth projective compactification of $X$. If $\iota$ is not surjective, then there exists a nonzero algebraic class $\alpha\in H_2(X/Y,\bQ)_{\mathrm{alg}}$ such that $(L\cdot \alpha)=0$ for all line bundle $L$ on $\oX$. As $\oX$ is smooth, this implies that the image of $\alpha$ in $H_2(\oX,\bQ)$ is zero. But by (the dual statement of) \cite{Voi-Hodge-book-II}*{Proposition 4.23}, the composition $H_2(X/Y,\bQ)\hookrightarrow H_2(X,\bQ)\to H_2(\oX,\bQ)$ is injective, a contradiction.
\end{proof}

\begin{lem} \label{lem:Bertini for higher direct image}
Let $f\colon X\to Y$ be a fibration over a reduced base $T$. Let $\cF$ be a coherent sheaf on $X$ such that $R^i f_*\cF = 0$ for some $i\in\bN$. Then there exists a dense open subset $U$ of $T$ such that $R^i f_*\cF_t = 0$ for all $t\in U$.
\end{lem}

\begin{proof}
By shrinking $T$ we may assume that $T$ is smooth. Since $R^j f_* \cF$ is coherent for all $j$, by generic flatness we know that there exists a dense open subset $U\subseteq T$ such that $R^j f_* \cF$ is flat over $U$ for all $j$. Similarly, after possibly shrinking $U$ we may assume that $\cF$ is flat over $U$. In particular, for any smooth hypersurface $Z\subseteq U$ with local defining equation $u$, the multiplication maps $\cF\xrightarrow{u}\cF$ and $R^j f_*\cF\xrightarrow{u} R^j f_*\cF$ are injective for all $j$. From the long exact sequence, we deduce that $R^j f_*\cF \otimes \cO_Z \cong R^j f_*(\cF\otimes \cO_Z)$. By induction on the dimension of $U$, this implies that $R^j f_* \cF_t \cong R^j f_*\cF \otimes k(t)$ for all $j$. Thus as $R^i f_*\cF = 0$, we see that $R^i f_*\cF_t = 0$ for all $t\in U$.
\end{proof}

\begin{lem} \label{lem:rel Pic local system}
Let $\varphi\colon \cX\to \cZ$ be a fibration over a variety $\cB$. Assume that $\cX$ has rational singularities and $R^1\varphi_*\cO_{\cX}=R^2\varphi_*\cO_{\cX}=0$. Then there exists a generically finite dominant morphism $\cB'\to \cB$ such that after base change $\cX'=\cX\times_{\cB} \cB'$, $\cZ'=\cZ\times_{\cB} \cB'$, the restriction map 
\[
N^1(\cX'/\cZ') \to N^1(\cX'_b/\cZ'_b)
\]
is an isomorphism for every $b\in \cB'$.
\end{lem}

\begin{proof}
By Lemma \ref{lem:Bertini for higher direct image} and shrinking $\cB$ if necessary, we shall assume that $\cB$ is smooth, $\cX_b$ has rational singularities, and $R^1\varphi_*\cO_{\cX_b}=R^2\varphi_*\cO_{\cX_b}=0$ for all $b\in \cB$ throughout the proof. First consider the special case when $\cX$ is smooth over $\cB$. By Lemma \ref{lem:N^1 dual to H_2} and \cite{KM-flip-classification}*{Theorem (12.1.3)}, we have 
\[
N^1(\cX/\cZ) \cong H_2(\cX/\cZ,\bQ)_{\mathrm{alg}}^\vee \cong \mathrm{coker}(H^2(\cZ,\bQ)\to H^2(\cX,\bQ))
\]
and similarly
\[
N^1(\cX_b/\cZ_b) \cong \mathrm{coker}(H^2(\cZ_b,\bQ)\to H^2(\cX_b,\bQ)).
\]
Denote the maps $\cX\to \cB$, $\cZ\to \cB$ by $p$ and $q$ respectively. By \cite{Ver-Whitney-stratification}*{Corollaire (5.1)}, there exists a dense, Zariski open subset $U\subseteq \cB$ such that $p^{-1}(U)\to U$ and $q^{-1}(U)\to U$ are locally topologically trivial fibrations (in the analytic topology). After replacing $\cB$ by $U$, we may therefore assume that $\varphi$ induces a map $\varphi^*\colon R^2 q_*\bQ \to R^2 p_*\bQ$ of local systems.  

By the previous discussions, each $N^1(\cX_b/\cZ_b)$ is a fiber of the local system $\mathrm{coker}(\varphi^*)$. As they also form the local system $\cG \cN^1(\cX/B)$ defined in \cite{KM-flip-classification}*{Definition (12.2.4)}, we deduce from \cite{KM-flip-classification}*{Proposition (12.2.5)} that it has finite monodromy. We can then trivialize it by some dominant generically finite cover $\cB'\to \cB$ to produce the desired isomorphism
\[
N^1(\cX'/\cZ')|_{\cX'_b} \cong N^1(\cX'_b/\cZ'_b).
\]
Thus the lemma holds in this case.

In general, let $\pi\colon \cY\to \cX$ be a log resolution and let $\psi\colon \cY\to \cZ$ the induced map. By shrinking $\cB$, we may assume that $\cY$ is smooth over $\cB$. As $\cX$ has rational singularities, we have $R^i\pi_*\cO_{\cY}=0$ for all $i>0$ and hence $R^i \psi_* \cO_{\cY}=R^i\varphi_* \cO_{\cX}$ for all $i\in\bN$. Thus from the smooth case treated above, we know that $N^1(\cY_b/\cZ_b)$ and $N^1(\cY_b/\cX_b)$ each form a local system with finite monodromy possibly after shrinking $\cB$. Since $N^1(\cX_b/\cZ_b)$ is the kernel of the natural restriction map $N^1(\cY_b/\cZ_b)\to N^1(\cY_b/\cX_b)$, we deduce that $N^1(\cX_b/\cZ_b)$ also form a local system with finite monodromy and the lemma follows as before.
\end{proof}

The following provides the first step in the proof of Theorem \ref{thm:bir model bdd}.

\begin{cor} \label{cor:rel Pic Fano fibration}
Let $\varphi\colon \cX\to \cZ$ be a Fano type fibration over a variety $\cB$. Then there exists a generically finite dominant morphism $\cB'\to \cB$ such that after base change $\cX'=\cX\times_\cB \cB'$, $\cZ'=\cZ\times_\cB \cB'$, the natural restriction map 
\[
N^1(\cX'/\cZ') \to N^1(\cX'_b/\cZ'_b)
\]
is an isomorphism for every $b\in \cB'$.
\end{cor}

\begin{proof}
By \cite{KM98}*{Theorem 5.22} and Kawamata-Viehweg vanishing, we know that $\cX$ has rational singularities and $R^1\varphi_*\cO_{\cX}=R^2\varphi_*\cO_{\cX}=0$. The statement then follows directly from Lemma \ref{lem:rel Pic local system}.
\end{proof}

By essentially the same argument, we also get the following result.

\begin{lem} \label{lem:Q-factorial in family}
Let $\varphi\colon \cX\to \cZ$ be a Fano type fibration over a variety $\cB$. Then there exist a generically finite dominant morphism $\cB'\to \cB$ and a small birational modification $\cY\to \cX'=\cX\times_\cB \cB'$ such that $\cY_b$ is $\bQ$-factorial for every $b\in \cB'$.
\end{lem}

\begin{proof}
Let $\pi\colon \widetilde{\cX}\to \cX$ be a log resolution. % and let $\cE_1,\cdots,\cE_m$ be the exceptional divisors. By shrinking $\cB$, we may assume that $\cB$ is smooth, and $\widetilde{\cX}$ and all the $\cE_i$ are smooth over $\cB$. Possibly after a generically finite dominant base change $\cB'\to \cB$, we may also assume that all the fibers $\cE_{i,b}$ are irreducible (this may increase the number $m$ of exceptional divisors). 
Note that for all $i>0$, we have $R^i\pi_* \cO_{\widetilde{\cX}}=0$ as $\cX$ has rational singularities by \cite{KM98}*{Theorem 5.22}, and $R^i\varphi_*\cO_{\cX}=0$ by Kawamata-Viehweg vanishing, hence $R^i(\varphi\circ\pi)_*\cO_{\widetilde{\cX}}=0$ for all $i>0$ as well. Thus by Lemmas \ref{lem:N^1 dual to H_2} and \ref{lem:rel Pic local system}, possibly after a generically finite base change, we may assume that $\cB$ is smooth, $\widetilde{\cX}$ is smooth over $\cB$, and there exist finitely many line bundles $\cL_1,\dots,\cL_r$ on $\widetilde{\cX}$ whose restriction to the fibers $\widetilde{\cX}_b$ generate $N^1(\widetilde{\cX}_b/\cZ_b)$ for all $b\in \cB$. It follows that the class group $\Cl(\cX_b)$ is generated by $\varphi^*\Pic(\cZ_b)$ and $\pi_*\cL_{i,b}$ ($i=1,\dots,r$). As $\cX$ is of Fano type over $\cZ$, it has a small $\bQ$-factorial modification $\cY\to \cX$ by \cite[Corollary 1.4.3]{BCHM}. By shrinking $\cB$ again, we may assume that the induced map $\cY_b\to \cX_b$ is small and $\Cl(\cY_b)\cong \Cl(\cX_b)$ is generated by $\Pic(\cZ_b)$ and $\cL'_{i,b}$ where $\cL'_i$ is the pushforward of $\cL_i$ to $\cY$. Since $\cY$ is $\bQ$-factorial, every $\cL'_i$ is $\bQ$-Cartier. This implies that $\cY_b$ is $\bQ$-factorial.
\end{proof}

% The next step (Lemma \ref{lem:pseff and movable cone constant}) in the proof of Theorem \ref{thm:bir model bdd} is to show that the relative pseudo-effective cone $\overline{\mathrm{Eff}}(\cX_b/\cZ_b)$, the movable cone $\overline{\mathrm{Mov}}(\cX_b/\cZ_b)$, and their Mori chamber decompositions are also locally constant.
The next step (Lemma \ref{lem:pseff cone constant}) in the proof of Theorem \ref{thm:bir model bdd} is to show that the relative pseudo-effective cone in $N^1(\cX_b/\cZ_b)_\bR$ is also locally constant. Before we prove this, we need an auxiliary result.

\begin{lem} \label{lem:compactify pair}
Let $(X,\Delta)$ be a klt pair over $T$. Assume that $\Delta$ has rational coefficients, $\Delta\ge A\ge 0$ for some ample $\mathbb Q$-divisor $A$. Then there exist a log resolution $\pi\colon X'\to X$, a projective compactification $X'\subseteq \oX$ over $T$, and an SNC $\mathbb Q$-divisor $\overline{\Delta}$ on $\oX$ such that $\oX\setminus X'\subseteq \Supp(\overline{\Delta})$, $\overline{\Delta}\geq\bar A\geq 0$ for some ample $\mathbb Q$-divisor $\bar A$, $(\oX,\overline{\Delta})$ is klt, and
\[
K_{X'}+\Delta'\sim_{\bQ} \pi^* (K_X+\Delta)+E
\]
for some $\pi$-exceptional $\mathbb Q$-divisor $E\ge 0$, where $\Delta'=\overline{\Delta}|_{X'}$.
\end{lem}

\begin{proof}
By assumption, we may write $\Delta=A+B$ where $B\ge 0$. First consider the case when $(X,\Delta)$ is log smooth. Then there exists a projective compactification $X\subseteq X_1$ over $T$ such that $A$ extends to an ample divisor $A_1$ on $X_1$. Choose some log resolution $\varphi\colon \oX\to X_1$ that is an isomorphism over $X$ such that the exceptional locus supports a $\varphi$-ample divisor $-F$ and $D=\oX\setminus X$ is simple normal crossing. Then $\varphi^*A_1-\varepsilon F$ is ample when $0<\varepsilon\ll 1$ ($\varepsilon\in \bQ$). Let $\overline{B}$ be the closure of $B$ in $\oX$. Passing to a higher model we may assume that $\overline{B}$ has SNC support. By Bertini's theorem, we may choose some $\bQ$-divisor $0\leq \overline{A}\sim_\bQ \varphi^*A_1-\varepsilon F$ such that $(\oX,\overline{\Delta}:=\overline{A}+\overline{B}+\frac{1}{2} D)$ is log smooth and klt. By construction, $\overline{\Delta}$ supports the ample divisor $\overline{A}$ and $\overline{\Delta}|_X \sim_\bQ A+B = \Delta$. Hence the lemma holds in this case (with $X'=X$).

In general, write $\Delta=A+B$ as before and take a log resolution $\pi\colon X'\to X$ of $(X,\Delta)$ whose exceptional locus supports a $\pi$-ample divisor $-F$. By the negativity lemma we have $F\ge 0$. We may write
\[
K_{X'}+B' = \pi^*(K_X+B)+E
\]
for some $\bQ$-divisors $B',E\ge 0$ such that $B'$ and $E$ have no common components, and $E$ is $\varphi$-exceptional. Since $(X,B)$ is klt, the coefficients of $B'$ are less than $1$, thus the log smooth pair $(X',B')$ is also klt. Choose some $0<\varepsilon\ll 1$ such that $\varphi^*A-\varepsilon F$ is ample and $(X',B'+\varepsilon F)$ remains klt. By Bertini's theorem, we may also choose some $0\le A'\sim_\bQ \pi^*A-\varepsilon F$ such that $(X',\Delta':=A'+B'+\varepsilon F)$ is still log smooth and klt. By construction, we see that $K_{X'}+\Delta'\sim_\bQ \pi^*(K_X+\Delta)+E$ and $\Delta'$ supports the ample $\bQ$-divisor $A'$. Thus we may replace $(X,\Delta)$ by $(X',\Delta')$ and reduce to the log smooth case treated before.
\end{proof}

The following gives the second main ingredient in the proof of Theorem \ref{thm:bir model bdd}.

\begin{lem} \label{lem:pseff cone constant}
Let $\cZ\to \cB$ be a dominant morphism of varieties and let $\cX\to \cZ$ be a Fano type fibration. Then there exists a dense open subset $U\subseteq \cB$ such that for any line bundle $\cL$ on $\cX$ and any closed point $b\in U$, we have $\cL_b$ is pseudo-effective (resp. big) over $\cZ_b$ if and only if $\cL_\eta$ is pseudo-effective (resp. big) over $\cZ_\eta$, where $\eta$ is the generic point of $\cB$.
% $\cL_U$ is the restriction of $\cL$ to $\cX_U=\cX\times_B U$.
\end{lem}

\begin{proof}
The statement is local on $\cZ$, thus we may assume that $\cZ$ is affine, in which case a divisor on $\cX$ is big (resp. ample) over $\cZ$ if and only if it is big (resp. ample) by itself. Since $\cX\to \cZ$ is a Fano type fibration, by shrinking $\cZ$ again we may assume that there exists some $\bQ$-divisor $\cD\geq 0$ on $\cX$ such that $(\cX,\cD)$ is klt, $K_{\cX}+\cD\sim_\bQ 0$, and $\cD\geq A\geq 0$ for some ample $\mathbb Q$-divisor $A$. Since a divisor $\cL$ on $\cX$ is big if and only if $\cL-\varepsilon A$ is pseudo-effective for some $\varepsilon>0$, the pseudo-effective version of the lemma implies the big version. Thus it suffices to consider the pseudo-effective case of the statement.

Since $\cX\to \cZ$ is a Fano type fibration, by applying \cite{BCHM}*{Corollary 1.3.2} to a small $\mathbb Q$-factorialization of $\mathcal{X}$, we know that there are finitely many divisors $\cD_i\ge 0$ on $\cX$ that generate the (relative) pseudo-effective cone; in other words, for any $\bR$-divisor $\cL$ on $\cX$ that is pseudo-effective over $\cZ$ (equivalently, $\cL_\eta$ is pseudo-effective over $\cZ_\eta$) we have $\cL\sim _{\bR,\cZ} \sum a_i \cD_i$ for some $a_i\ge 0$. By shrinking $\cB$ we may assume that $\cB$ is smooth and affine, $\cX$ is flat over $\cB$, $\cX_b$ is normal, and $\cD_i$ does not contain $\cX_b$ for all $i$ and all $b\in\cB$. In particular, the restriction $\cD_{i,b}$ is well-defined and we have $\cL_b\sim _{\bR,\cZ_b} \sum a_i \cD_{i,b}$, hence $\cL_b$ is pseudo-effective over $\cZ_b$ whenever $\cL_\eta$ is pseudo-effective over $\cZ_\eta$. % Note that this remains true if we further shrink $\cB$, since $\cL$ is pseudo-effective over $\cZ$ as long as its restriction to the generic fiber of $\cX\to \cZ$ is pseudo-effective (by definition).

It remains to show that the converse direction holds after possibly further shrinking $\cB$. For this, we first note that since $\cZ$ is affine, we can factorize $\cZ\to \cB$ into an open immersion $\cZ\to \ocZ$ followed by a projective morphism $\ocZ\to \cB$ % compactification $\cZ\subseteq \ocZ$ over $\cB$ 
such that there is an ample Cartier divisor $\cM_0$ on $\ocZ$ with $\Supp(\cM_0)=\ocZ\setminus \cZ$. By Lemma \ref{lem:compactify pair}, there exist a log resolution $\pi\colon \cX'\to \cX$, a projective compactification $\cX'\subseteq \ocX$ over $\ocZ$ and an SNC $\bQ$-divisor $\ocD$ on $\ocX$ such that $\ocX\setminus \cX'\subseteq \Supp(\ocD)$, $\ocD=\overline{A}+\overline{B}$ where $\overline{A}\geq 0$ is an ample $\mathbb Q$-divisor and $\overline{B}\geq 0$, $(\ocX,\ocD)$ is klt, and
\[
K_{\cX'}+\cD'\sim_{\bQ} \pi^* (K_{\cX}+\cD)+\cE
\]
for some $\pi$-exceptional $\bQ$-divisor $\cE\ge 0$, where $\cD'=\ocD|_{\cX'}$. Let $U\subseteq \cB$ be an affine open subset such that $(\cX_b,\cD_b)$ is klt (in particular, this implies that $\cX_b\to \cZ_b$ is a Fano type fibration) for all $b\in U$, and $(\ocX,\ocD)$ is log smooth over $U$. Let $\cM$ be the pullback of $\cM_0$ to $\ocX$. By construction, $\ocX$ is projective over $\cB$, the divisor $\cM$ is semiample over $\cB$, and $\Supp(\cM)=\ocX\setminus \cX'$.

Now let $\cL$ be a line bundle on $\cX$ such that $\cL_b$ is pseudo-effective over $\cZ_b$. Then $H^0(\cX_b,\ell\cL_b)\neq 0$ for some positive integer $\ell$ since $\cX_b$ is of Fano type over $\cZ_b$ and $\cZ_b$ is affine. As $\ocX$ is smooth, we may extend $\cL':=\pi^*\cL$ to a line bundle $\ocL$ on $\ocX$. Note that for all sufficiently divisible positive integer $m$ we have
\[
m(K_{\cX'}+\cD')+\ell \cL'\sim m(\pi^* (K_{\cX}+\cD)+\cE)+\ell \pi^*\cL\sim m\cE+\ell\pi^*\cL.
\]
As $H^0(\cX_b,\ell\cL_b)\neq 0$ and $\cE\ge 0$, this implies that 
\[
H^0(\cX'_b,m(K_{\cX'_b}+\cD'_b)+\ell \cL'_b)\neq 0.
\]
Since $\cE$ is $\pi$-exceptional, it also implies that $\cL$ is pseudo-effective as long as $m(K_{\cX'}+\cD')+\ell \cL'$ is pseudo-effective for some positive integers $\ell, m$. Choose some sufficiently large and divisible integer $m$ (which depends on $\ell$) such that $m\overline{A}$ and $m\overline{B}$ are Cartier, and $m\overline{A}+\ell\ocL$ is ample. Since $H^0(\cX'_b,m(K_{\cX'_b}+\cD'_b)+\ell\cL'_b)\neq 0$ and $\Supp(\cM)=\ocX\setminus \cX'$, we may take a sufficiently large integer $r$ (depending on $\ell$ and $m$) such that 
\begin{equation} \label{eq:H^0 not 0 on Xbar}
H^0(\ocX_b,m(K_{\ocX_b}+\ocD_b)+\ell\ocL_b+r\cM_b)\neq 0.
\end{equation}
By construction, we have
\[
m(K_{\ocX}+\ocD)+\ell\ocL+r\cM = m(K_{\ocX}+\overline{B})+(m\overline{A}+\ell\ocL+r\cM).
\]
Note that $(\ocX,\overline{B})$ is log smooth over $\cB$, and $m\overline{A}+\ocL+r\cM$ is ample, thus by \eqref{eq:H^0 not 0 on Xbar} and \cite{HMX-BirAut}*{Theorem 1.8}, we deduce that 
\[
H^0(\ocX,m(K_{\ocX}+\ocD)+\ell \ocL+r\cM)\neq 0.
\]
By restriction to $\cX'$ we also get a section of $m(K_{\cX'}+\cD')+\ell \cL'$. From the previous discussion, we deduce that $\cL$ is pseudo-effective and this concludes the proof.
\end{proof}

We now return to the proof of Theorem \ref{thm:bir model bdd}. We first consider a special case.

\begin{lem} \label{lem:bir contraction bdd}
Let $\cS$ be a bounded set of Fano type fibrations. Then there exists a bounded set $\cT$ of Fano type fibrations such that for any $(X,\Delta)\to Z$ in $\cS$ and any birational contraction $f\colon X\dashrightarrow X'$ over $Z$, we have $(X',\Delta'=f_* \Delta)\to Z$ belongs to $\cT$.
\end{lem}

\begin{proof}
We may assume that $\cS$ is given by the fibers of a Fano type fibration $(\cX,\cD)\to \cZ$ over a finite type base $\cB$. By induction on $\dim\cB$, it suffices to prove the statement after some generically finite dominant base change $\cB'\to \cB$. Thus we may shrink $\cB$ freely as needed. In particular we assume that $\cB$ is smooth and connected throughout the proof. 

By Lemma \ref{lem:Q-factorial in family}, possibly after a generically finite dominant base change and replacing $\cX$ by a small birational modification, we may assume that $\cX_b$ is $\bQ$-factorial for all $b\in \cB$. Note that a small birational modification does not change the set of birational contractions. By Corollary \ref{cor:rel Pic Fano fibration}, possibly after another generically finite dominant base change, we may also assume that $N^1(\cX_b/\cZ_b)$ form a trivial local system. % By Lemma \ref{lem:pseff and movable cone constant}, by further shrinking $\cB$, we can assume that the pseudo-effective and movable cones in $N^1(\cX_b/\cZ_b)$ are constant in $b\in \cB$. 
By \cite{BCHM}*{Corollary 1.3.1}, there are only finitely many birational contractions $f_i\colon \cX\dashrightarrow \cX_i$ ($i=1,\dots,m$) over $\cZ$. Let $\cE_i$ be the sum of all the exceptional divisors of $f_i$. By shrinking the base $\cB$ again, we can assume that the following holds simultaneously for all $i$ and all $b\in \cB$:
\begin{itemize}
    \item $\cX_{i,b}$ is normal,
    \item $\cE_{i,b}\neq 0$ whenever $\cE_i\neq 0$,
    % \item the divisor $\cE_i$ has nonempty intersection with $\cX_b$,
    \item the map $f_i$ is defined at the generic point of $\cX_b$,
    \item the induced map $f_{i,b}\colon \cX_b\dashrightarrow \cX_{i,b}$ is a birational contraction, 
    \item $\cE_{i,b}$ is contained in the exceptional locus of $f_{i,b}$, and
    \item $f_{i,b}$ is small whenever $f_i$ is small.
\end{itemize} 

Let $\cX_b\dashrightarrow Y$ be a birational contraction over $\cZ_b$. Let $H$ be a very ample Cartier divisor on $Y$ and let $\cL_b$ be its birational pullback to $\cX_b$. Then $\cL_b$ is big and movable (i.e. the base locus of $|\cL_b|$ has codimension at least $2$ in $\cX_b$). It is also $\bQ$-Cartier since $\cX_b$ is $\bQ$-factorial. By Lemma \ref{lem:pseff cone constant}, possibly after replacing $\cL_b$ by a multiple we may extend it to a big line bundle $\cL$ on $\cX$. The ample model of $\cL$ over $\cZ$ necessarily coincides with one of the $f_i\colon \cX\dashrightarrow \cX_i$. We claim that $Y\cong \cX_{i,b}$. 

To see this, first note that the $\cL$-MMP over $\cZ$ coincides with some $f_j\colon \cX\dashrightarrow \cX_j$. We first show that $f_j$ is small. Let $\cL_j=f_{j*}\cL$ which is nef over $\cZ$. By the negativity lemma (see e.g. \cite{KM98}*{Lemma 3.38}), we have $\cL=f_j^*\cL_j + \cE$ for some divisor $\cE\geq 0$ whose support contains every exceptional divisor of $f_j$. If $f_j$ is not small, then $\cE\neq 0$. Restricting to the fiber we see that $\cL_b=f_{j,b}^*\cL_{j,b} + \cE_b$ and hence the base locus of $\cL_b$ contains the divisor $\cE_b\neq 0$, a contraction as $\cL_b$ is movable. Thus $f_j$ is small and so is $f_{j,b}$ (by the last bullet point above). The ample model $\cX_i$ (resp. $Y$) of $\cL$ (resp. $\cL_b$) is therefore the same as the ample model of $\cL_j$ (resp. $\cL_{j,b}$) over $\cZ$ (resp. $\cZ_b$). Since $\cL_j$ is nef over $\cZ$, its ample model is a morphism $\cX_j\to \cX_i$; combined with the fact that $\cX_{i,b}$ is normal, we see that $\cX_{i,b}$ is also the ample model of $\cL_{j,b}$. This proves that $Y\cong \cX_{i,b}$. It follows that every birational contraction of $(\cX_b,\cD_b)$ over $\cZ_b$ is a fiber of the finitely many families $(\cX_i,\cD_i=f_{i*}\cD)\to \cZ$ over $\cB$. 
\end{proof}

\begin{proof}[Proof of Theorem \ref{thm:bir model bdd}]
By Lemma \ref{lem:bdd defn equiv}, we may assume that $\cS$ is given by the fibers of an $\bR$-Gorenstein family $(\cX,\cD)\to \cZ$ of log Fano fibrations over a finite type base $\cB$. By the proof of \cite{HX-CY-bdd}*{Proposition 2.5} (or by Lemma \ref{lem:Q-factorial in family} and {\it loc. cit.}), possibly after a base change via some surjective morphism $\cB'\to \cB$, there exists a crepant birational morphism $(\cY,\cE)\to (\cX,\cD)$ whose restriction to the fiber over $b\in \cB$ extracts every divisor of non-positive discrepancy over $(\cX_b,\cD_b)$. Note that $(\cY,\cE)$ is also of Fano type over $\cZ$ by Lemma \ref{lem:Fano type}, and every birational model $f\colon X'\dashrightarrow \cX_b$ such that $f^*(K_{\cX_b}+\cD_b)=K_{X'}+\Delta'$ for some $\Delta'\ge 0$ can be obtained by a birational contraction $\cY_b\dashrightarrow X'$. The theorem thus follows from Lemma \ref{lem:bir contraction bdd}.
\end{proof}

\subsection{Special boundedness} \label{ss:special bdd}

To prove Theorems \ref{thm:fibration special bdd, no boundary} and \ref{thm:fibration special bdd}, we still need to upgrade Theorem \ref{thm:bir model bdd} to allow specially bounded input. The precise statement is the following.

\begin{thm} \label{thm:bir model log special bdd}
Let $\cS$ be a specially bounded set of log Fano fibration germs. Let $\cT$ be the set of Fano type fibration germs $(X',\Delta')\to Z\ni z$ for which there exist some $(X,\Delta)\to Z\ni z$ in $\cS$ and a birational map $f\colon X'\dashrightarrow X$ over $Z$ such that $f^*(K_X+\Delta)\ge K_{X'}+\Delta'$. Then $\cT$ is log specially bounded.
\end{thm}

To prove this, we need a few auxiliary lemmas.

\begin{lem} \label{lem:adjunction ineq}
Let $(X,G+D)$ be a log canonical pair where $G$ is a normal Weil divisor and $D\ge 0$, and let $X\to Z$ a projective morphism such that $-(K_X+G+D)$ is semiample over $Z$. Let $f\colon X'\dashrightarrow X$ be a birational map over $Z$ where $X'$ is projective over $Z$. Let $G'$ be the strict transform of $G$ on $X'$ and let $D'$ be an $\bR$-divisor on $X'$ such that
\[
f^*(K_X+G+D)\ge K_{X'}+G'+D'.
\]
Assume that $G'$ is normal and Cartier on $X'$. Then for the induced birational map $f_G\colon G'\dashrightarrow G$, we have
\[
f_G^*(K_G+\Diff_G(D))\ge K_{G'}+D'|_{G'}.
\]
\end{lem}

\begin{proof}
By enlarging $D'$ we may assume that $f^*(K_X+G+D)= K_{X'}+G'+D'$. The statement is local on $Z$, hence we may also assume that $-(K_X+G+D)$ is semiample. By Bertini's theorem, we can choose some $\bR$-divisor $0\leq H\sim_\bR -(K_X+G+D)$ such that $H$ (resp. $H|_G$) does not contain any exceptional divisor of $f$ (resp. $f_G$) in its support. In particular, we may assume that $f$ is an isomorphism at every generic point of $\Supp(H)\cap G$, hence $f^*H=H'$ and $f_G^*(H|_G)\le H'|_{G'}$, where $H'$ is the strict transform of $H$ on $X'$. As $K_X+G+D+H\sim_\bR 0$, we also have $f^*(K_X+G+D+H)=K_{X'}+G'+D'+H'\sim_\bR 0$ and this is a crepant pullback. By adjunction, we deduce that 
\[
f_G^*(K_G+\Diff_G(D+H))=K_{G'}+\Diff_{G'}(D'+H').
\]
Note that $\Diff_G(D+H)=\Diff_G(D)+H|_G$ since $H$ is $\bR$-Cartier and $\Diff_{G'}(D'+H')=(D'+H')|_{G'}$ since $G'$ is Cartier. Thus
\[
f_G^*(K_G+\Diff_G(D)) = K_{G'}+(D'+H')|_{G'} - f_G^*(H|_G) \ge K_{G'}+D'|_{S'}
\]
as desired.
\end{proof}

\begin{lem} \label{lem:extend bir map to tc}
Let $(\cX,\cD)\to \cZ$ be a special test configuration of a log Fano fibration $(X,\Delta)\to Z$. Let $(X',\Delta')$ be a projective pair over $Z$ and $f\colon X'\dashrightarrow X$ a birational map over $Z$ such that 
\begin{equation} \label{eq:crepant ineq generic fiber}
    f^*(K_X+\Delta) \ge K_{X'}+\Delta'.
\end{equation}
Then there exist a Fano type test configuration $(\cX',\cD')\to \cZ$ of $(X',\Delta')\to Z$ and a birational map $\cX'\dashrightarrow \cX$ over $\cZ$ extending $f$, such that $\cX'_0$ is the strict transform of $\cX_0$ and the induced birational map $f_0\colon \cX'_0\dashrightarrow \cX_0$ over $\cZ_0$ satisfies
\begin{equation} \label{eq:crepant ineq special fiber}
    f_0^*(K_{\cX_0}+\cD_0) \ge K_{\cX'_0}+\cD'_0.
\end{equation}
\end{lem}

\begin{proof}
Let $E_1,\dots,E_r$ be the prime divisors on $X'$ that are contracted by $f$. Through the isomorphism $\cX\setminus \cX_0\cong X\times (\bA^1\setminus \{0\})$, each $E_i\times (\bA^1\setminus \{0\})$ is identified with a divisor $\cE_i$ over $\cX$. By assumption, $a(E_i,X,\Delta)\le 0$, thus $a(\cE_i,\cX,\cD)\le 0$. By \cite{BCHM}*{Corollary 1.4.3}, we know that there exists a projective birational morphism $h\colon \cY\to \cX$ such that $\cY$ is $\bQ$-factorial and the exceptional divisors are exactly given by $\cE_1,\dots,\cE_r$. Since $\cE_i$ is invariant under the $\bG_m$-action, we also know that $h$ is $\bG_m$-equivariant. By construction, $\cY_0$ is irreducible (as all the $h$-exceptional divisors dominate $\bA^1$), and the induced birational map $\cY\setminus \cY_0 \dashrightarrow X'\times (\bA^1\setminus \{0\})$ is a birational contraction.

Let $L$ be an ample line bundle on $X'$ and choose some Weil divisor $\cL$ on $\cY$ whose restriction to $\cY\setminus \cY_0$ is the birational pullback of $L\times (\bA^1\setminus \{0\})$. Since $\cY$ is $\bQ$-factorial, we know that $\cL$ is $\bQ$-Cartier. Note that $\cY$ is of Fano type over $\cZ$ by Lemma \ref{lem:Fano type}, hence the ample model $g\colon \cY\dashrightarrow \cX'$ of $\cL$ over $\cZ$ exists and is automatically $\bG_m$-equivariant. By construction we have $\cX'\setminus \cX'_0\cong X'\times (\bA^1\setminus \{0\})$ and the induced map $\widetilde{f}\colon \cX'\dashrightarrow \cX$ extends $f$. Note that as $\cY_0$ is irreducible and semiample, it is not in the base locus of $\cL$, thus $g$ is an isomorphism at the generic point of $\cY_0$. It follows that $\cX'_0$ is the strict transform of $\cY_0$ (in particular, it is irreducible), and we have an induced birational map $f_0\colon \cX'_0\dashrightarrow \cX_0$.

% As $(\cX,\cD)\to \cZ$ is a special test configuration, we have $(\cX,\cX_0+\cD)$ is plt and $-(K_{\cX}+\cX_0+\cD)$ is ample over $\cZ$. By Bertini's theorem, there exists some $\bR$-divisor $\cH\geq 0$ such that $K_{\cX}+\cX_0+\cD+\cH\sim_{\bR,\cZ} 0$ and the pair $(\cX,\cX_0+\cD+\cH)$ remains plt. Let $(\cX',\cX'_0+\cG)$ be its crepant pullback to $\cX'$, i.e. 
% \[
% K_{\cX'}+\cX'_0+\cG = \widetilde{f}^*(K_{\cX}+\cX_0+\cD+\cH)\sim_{\bR,\cZ} 0.
% \]
% Then the pair $(\cX',\cX'_0+\cG)$ is plt (note that $\cG\ge 0$ since every exceptional divisor $\cE$ of $\widetilde{f}$ satisfies $A_{\cX,\cD}(\cE)\le 1$ by construction). Let $\cD'$ be the closure of $\Delta'\times (\bA^1\setminus \{0\})$ in $\cX'$. By \eqref{eq:crepant ineq generic fiber}, we have
% \[
% \widetilde{f}^*(K_{\cX}+\cX_0+\cD)\ge K_{\cX'}+\cX'_0+\cD',
% \]
% hence $\cG\ge \cD'$. Thus $(\cX',\cD')\to \cZ$ is a Fano type test configuration of $(X',\Delta')\to Z$, and we obtain \eqref{eq:crepant ineq special fiber} by Lemma \ref{lem:adjunction ineq}.

Let $\cD'$ be the closure of $\Delta'\times (\bA^1\setminus \{0\})$ in $\cX'$. By \eqref{eq:crepant ineq generic fiber}, we have
\[
\widetilde{f}^*(K_{\cX}+\cD) \ge K_{\cX'}+\cD'.
\]
As $(\cX,\cD)\to \cZ$ is a special test configuration, we see that $(\cX',\cD')\to \cZ$ is a Fano type test configuration of $(X',\Delta')\to Z$ by Lemma \ref{lem:Fano type tc}. Note that we also have $\widetilde{f}^* \cX_0 = \cX'_0$, hence $\widetilde{f}^*(K_{\cX}+\cX_0+\cD) \ge K_{\cX'}+\cX'_0+\cD'$. Thus we obtain \eqref{eq:crepant ineq special fiber} by Lemma \ref{lem:adjunction ineq}.
\end{proof}

We are ready to prove Theorem \ref{thm:bir model log special bdd}.

\begin{proof}[Proof of Theorem \ref{thm:bir model log special bdd}]
By assumption, there exists a bounded set $\cS_0$ of log Fano fibrations such that any fibration germ $(X,\Delta)\to Z\ni z$ in $\cS$ specially degenerates to some $(X_0,\Delta_0)\to Z_0\ni z_0$ in $\cS_0$.  If $f\colon X'\dashrightarrow X$ is a birational map over $Z$ such that $f^*(K_X+\Delta)\ge K_{X'}+\Delta'$ for some $\Delta'\ge 0$, then by Lemma \ref{lem:extend bir map to tc} we get an induced Fano type test configuration of $(X',\Delta')\to Z\ni z$ whose central fiber is of the form $(X'_0,\Delta'_0)\to Z_0\ni z_0$ for some birational model $f_0\colon X'_0\dashrightarrow X_0$ such that $f_0^*(K_{X_0}+\Delta_0)\ge K_{X'_0}+\Delta'_0$. Write $f_0^*(K_{X_0}+\Delta_0)=K_{X'_0}+D'_0$. Then $D'_0\ge \Delta'_0\ge 0$ and $\Supp(\Delta'_0)\subseteq \Supp(D'_0)$. By Theorem \ref{thm:bir model bdd}, the set of such fibrations $(X'_0,D'_0)\to Z_0$ is bounded, hence the set of $(X',\Delta')\to Z\ni z$ is log specially bounded.
% Let $f\colon X'\dashrightarrow X$ be a birational map over $Z$ such that $f^*(K_X+\Delta)\ge K_{X'}+\Delta'$ for some $\Delta'\ge 0$. Write $f^*(K_{X}+\Delta)=K_{X'}+D'$. Then $D'\ge \Delta'\ge 0$. By Lemma \ref{lem:extend bir map to tc} we get an induced Fano type test configuration of $(X',D')\to Z\ni z$ whose central fiber is of the form $(X'_0,D'_0)\to Z_0\ni z_0$ for some birational model $f_0\colon X'_0\dashrightarrow X_0$ such that $f_0^*(K_{X_0}+\Delta_0)\ge K_{X'_0}+D'_0$. By Theorem \ref{thm:bir model bdd}, the set of such fibrations $(X'_0,D'_0)\to Z_0$ is bounded, hence the set of $(X',\Delta')\to Z\ni z$ is log specially bounded as $\Supp(\Delta')\subseteq \Supp(D')$.
\end{proof}

We now have all the ingredients for the proof of Theorems \ref{thm:fibration special bdd, no boundary} and \ref{thm:fibration special bdd}.

\begin{proof}[Proof of Theorem \ref{thm:fibration special bdd, no boundary}]
Let $X\to Z\ni z$ be a Fano type fibration germ such that $\hVol_X(\Delta)\ge \varepsilon$ for some $\bR$-divisor $0\le \Delta\sim_{\bR,Z} -K_X$. % $X$ is of Fano type over $Z$, the $\bR$-divisor $\Delta$ is big, and $\hVol_X (\Delta)\ge \varepsilon$. 
By \cite{BCHM}*{Corollary 1.4.3}, a small $\bQ$-factorial modification $X''\to X$ exists. By Lemma \ref{lem:Fano type} and \cite{BCHM}*{Corollary 1.3.2}, the ample model $X''\dashrightarrow X'$ of $-K_{X''}$ over $Z$ also exists. Note that $X'\to Z$ is a log Fano fibration by construction. Let $h\colon X\dashrightarrow X'$ be the induced birational contraction and let $\Delta'$ be the strict transform of $\Delta$ on $X'$. Since $K_X+\Delta\sim_{\bR,Z} 0$, we have $K_{X'}+\Delta'\sim_{\bR,Z} 0$ as well. In particular, the birational map $(X,\Delta)\dashrightarrow (X',\Delta')$ is crepant. By Lemma \ref{lem:Vol MMP}, we have $\hVol_{X'}(\Delta')\ge \hVol_X(\Delta)\ge \varepsilon$, hence by Theorem \ref{thm:fibration special bdd, finite rational coef}, we know that $X'\to Z\ni z$ belongs to a specially bounded set of log Fano fibration germs. 

By the negativity lemma, we have $-K_X\ge h^*(-K_{X'})$, or $h^*K_{X'}\ge K_X$. By Theorem \ref{thm:bir model log special bdd}, we see that the Fano type fibration germ $X\to Z\ni z$ belongs to a log specially bounded set. By definition, this implies that the set $\cS$ of such fibrations is bounded up to Fano type degenerations.
\end{proof}

As explained in Section \ref{ss:straregy for fiber bdd}, this also completes the proof of Theorem \ref{thm:fiber bdd}.

% The proof of Theorem \ref{thm:fibration special bdd} follows a similar strategy.

\begin{proof}[Proof of Theorem \ref{thm:fibration special bdd}]
By perturbing $\delta$, we may assume that $\delta\in \bQ$. Let $(X,\Delta)\to Z\ni z$ be a log Calabi-Yau fibration germ in $\cS$. Let $\Delta_0 = \frac{\delta}{2}\Supp(\Delta)$. Since $X$ is of Fano type and $(X,\Delta)$ is log Calabi-Yau over $Z$, it is not hard to see that their convex combination $(X,\frac{1}{2}\Delta)$ is of Fano type over $Z$. Therefore, since $\Delta_0\le \frac{1}{2}\Delta$, we see that $(X,\Delta_0)$ is of Fano type over $Z$ as well. Similar to the above proof of Theorem \ref{thm:fibration special bdd, no boundary}, let $h\colon X\dashrightarrow X'$ be the ample model of $-(K_{X}+\Delta_0)$ which exists by \cite{BCHM}*{Corollary 1.3.2}. Let $\Delta'=h_*\Delta$ and $\Delta'_0=h_* \Delta_0$. We have $\hVol_{X'}(\Delta')\ge \varepsilon$ by Lemma \ref{lem:Vol MMP}, and $h^*(K_{X'}+\Delta_0')\ge K_X+\Delta_0$ by the negativity lemma. By Lemma \ref{lem:Vol change boundary}, we also have $\hVol_{X',\Delta'_0}(\Delta'-\Delta'_0)\ge \frac{\varepsilon}{2^n}$ since $\Delta'_0\le \frac{1}{2}\Delta'$ by construction. Note that $K_{X'}+\Delta'\sim_{\bR,Z}0$ and thus $\Delta'-\Delta'_0\sim_{\bR,Z} -(K_{X'}+\Delta'_0)$. By Theorem \ref{thm:fibration special bdd, finite rational coef}, it follows that the log Fano fibration germ $(X',\Delta'_0)\to Z\ni z$ belongs to a specially bounded set. By Theorem \ref{thm:bir model log special bdd}, the Fano type fibrations $(X,\Delta_0)\to Z\ni z$ belong to a log specially bounded set. As $\Supp(\Delta_0)=\Supp(\Delta)$, the set $\cS$ is also log specially bounded.
\end{proof}

\appendix

\section{Boundedness of models implies termination}

In this appendix, we prove that if the pairs that appear in a log canonical MMP belong to a bounded family, then the MMP must terminate after finitely many steps. As a corollary, we show that in any sequence of a log canonical MMP the pairs are mutually non-isomorphic, answering \cite{Kol-KM98exercise}*{Problem 82}.

\begin{thm} \label{thm:bdd imply termination}
Let $(X,\Delta)$ be a log canonical pair and let 
\[
\begin{tikzcd}
(X,\Delta)=:(X_0,\Delta_0) \arrow[r,dashed] & (X_1,\Delta_1) \arrow[r,dashed] & \cdots \arrow[r,dashed] & (X_i,\Delta_i) \arrow[r,dashed] & \cdots,
\end{tikzcd}
\]
be a sequence of steps of a $(K_X+\Delta)$-MMP. Assume that the set of pairs $\{(X_i,\Delta_i)\,|\,i\in\bN\}$ is bounded. Then this MMP terminates after finitely many steps. 
\end{thm}

Here the boundedness assumption means that there exists an $\bR$-Gorenstein family of pairs (see Definition \ref{defn:R-Gor family}) $(\cX,\cD)\to \cB$ that contains every $(X_i,\Delta_i)$ as a fiber.

\begin{cor} \label{cor:no loop in MMP}
Let $(X,\Delta)$ be a log canonical pair and let $(X,\Delta)\dashrightarrow (X',\Delta')$ be a sequence of steps of some $(K_X+\Delta)$-MMP. Then $(X',\Delta')$ is not isomorphic to $(X,\Delta)$.
\end{cor}

Before we give the proof, let us recall Batyrev's stringy $E$-function (see e.g. \cite{Batyrev-stringy-E-function}*{Definition 3.7} or \cite{Veys-stringy-survey}*{Definition 7.7}). First let $(X,\Delta)$ be a klt log smooth sub-pair. Let $D=\cup_{i\in I} D_i$ be a simple normal crossing divisor on $X$ that contains every irreducible component of $\Supp(\Delta)$, and write $\Delta=\sum_{i\in I} a_i D_i$. Note that $a_i>-1$ for all $i$ by the klt assumption. Then we define the stringy $E$-function of the sub-pair $(X,\Delta)$ as
\[
E_{\mathrm{st}}(X,\Delta;u,v):=\sum_{J\subseteq I} H(D_J^\circ;u,v) \prod_{j\in J} \frac{uv-1}{(uv)^{1+a_j}-1},
\]
where $D_J^\circ:=\cap_{j\in J} D_j \setminus (\cup_{j\not \in J} D_j)$ (with the understanding that $D_\emptyset^\circ=X\setminus \Supp(D)$ and $\prod_{j\in J}=1$ if $J=\emptyset$), and  $H(D_J^\circ;u,v)$ is the Hodge-Deligne polynomial of $D_J^\circ$, which is defined by
\[
H(U;u,v) := \sum_{i,p,q} (-1)^i h^{p,q}(H^i_c (U,\bC)) u^p v^q
\]
for any algebraic variety $U$. It is not hard to see that the definition of the stringy $E$-function is independent of the choice of $D$.

In general, let $(X,\Delta)$ be a log sub-pair, let $\fa\subseteq \cO_X$ be an ideal sheaf, and let $c\in \bR$ be such that $(X,\Delta+\fa^c)$ is klt, in the sense that
\[
A_{X,\Delta+\fa^c}(F):= A_{X,\Delta}(F) - c\cdot \ord_E(\fa)>0
\]
for all prime divisor $F$ over $X$. Note that we allow $c$ and the coefficient of $\Delta$ to be negative. Let $\pi\colon (Y,D)\to (X,\Delta+\fa)$ be a crepant log resolution. In other words,
\[
\Ex(\pi)+\Supp(\pi^{-1}_*\Delta)+\pi^{-1}\Supp(\cO_X/\fa)
\]
is a simple normal crossing divisor on $Y$, and
\[
-\mult_{F} D = a(F,X,\Delta+\fa^c):=a(F,X,\Delta) - c\cdot \ord_{F}(\fa)
\]
for every prime divisor $F$ on $Y$. Then we define the stringy $E$-function of the klt sub-pair $(X,\Delta+\fa^c)$ as
\[
E_{\mathrm{st}}(X,\Delta+\fa^c;u,v):=E_{\mathrm{st}}(Y,D;u,v).
\]
By \cite{Batyrev-stringy-E-function}*{Theorems 3.6 and 3.8}, the definition is independent of the crepant log resolution $(Y,D)$.

\begin{lem} \label{lem:string Hodge decrease}
Let $(X,\Delta+\fa^c)$ and $(X',\Delta'+\fb^{d})$ be klt sub-pairs and let $\varphi\colon X\dashrightarrow X'$ be a birational map. Assume that: 
\begin{enumerate}
    \item There exist proper birational morphisms $f\colon Y\to X$ and $f'\colon Y\to X'$ such that $f' = \varphi\circ f$, and
    \item 
    $A_{X,\Delta+\fa^c}(F)\le A_{X',\Delta'+\fb^d}(F)$ for all divisors $F$ over $X$ and the inequality is strict for some $F$. % $a(F,X,\Delta+\fa^c)\le a(F,X',\Delta'+\fb^d)$
\end{enumerate}
Then for $u,v\gg 0$ we have
\[
E_{\mathrm{st}}(X,\Delta+\fa^c;u,v) > E_{\mathrm{st}}(X',\Delta'+\fb^d;u,v).
\] 
\end{lem}

\begin{proof}
Let $Y$ be a common log resolution of $(X,\Delta+\fa^c)$ and $(X',\Delta'+\fb^d)$ and let $D=\cup_{i\in I} D_i$ be an SNC divisor on $Y$ that contains the strict transforms of $\Delta$ and $\Delta'$, the preimage of $\Supp(\cO_X/\fa)$ and $\Supp(\cO_{X'}/\fb)$, and all the exceptional divisors over $X$ or $X'$. Let $a_i=A_{X,\Delta+\fa^c}(D_i)$ and $a'_i=A_{X',\Delta'+\fb^d}(D_i)$. 
By assumption we have $a'_i\ge a_i>0$ and the inequality is strict for some $i$. Without loss of generality we may assume that $a'_1>a_1$. Note that the Hodge-Deligne polynomials have positive leading coefficients by \cite[Theorem 5.39]{PS-mix-Hodge}. Then for any $u,v\gg 0$ and any $J\subseteq I$, we have
\[
H(D_J^\circ;u,v) \prod_{j\in J} \frac{uv-1}{(uv)^{a_j}-1} \ge H(D_J^\circ;u,v) \prod_{j\in J} \frac{uv-1}{(uv)^{a'_j}-1}.
\]
Moreover, the inequality is strict when $J=\{1\}$. By the definition of the stringy $E$-function, this implies that
\[
E_{\mathrm{st}}(X,\Delta+\fa^c;u,v) > E_{\mathrm{st}}(X',\Delta'+\fb^d;u,v)
\]
when $u,v\gg 0$.
\end{proof}

\begin{lem} \label{lem:stringy constant}
Let $\ocX\to \cB$ be a projective morphism. Let $\ocD$ be an $\bR$-divisor on $\ocX$ such that $(\ocX,\ocD)$ is log smooth and klt, and every stratum of $(\ocX,\ocD)$ is smooth over $\cB$. Let $\cX\subseteq \ocX$ be an open subset that is a union of some open strata of $(\ocX,\ocD)$, and let $\cD=\ocD|_{\cX}$. Then the stringy $E$-function $E_{\mathrm{st}}(\cX_b,\cD_b;u,v)$ is constant in $b\in \cB$.
\end{lem}

\begin{proof}
% Let $\ocD_i$ ($i\in I$) be the irreducible components of $\ocD$. 
It suffices to show that for any open stratum $\cW^\circ$ of $(\ocX,\ocD)$, the Hodge-Deligne polynomial $H(\cW^\circ_b;u,v)$ of $\cW^\circ_{b} := \cW^\circ \cap \ocX_b$ is independent of $b\in \cB$. We prove this by Noetherian induction. To this end, let $\cW:=\overline{\cW^\circ}$ which is a stratum of $(\ocX,\ocD)$. There is nothing to prove when $\cW = \emptyset$. For the inductive step, note that by Deligne's mixed Hodge theory, we have 
\begin{equation} \label{eq:Hdg poly additive}
    H(X;u,v)=H(X\setminus Z;u,v)+H(Z;u,v)
\end{equation}
whenever $Z$ is a closed subvariety of $X$ (see e.g. \cite[Theorem 1.1]{Sri-Hodge}). Since $\cW\setminus \cW^\circ$ is a union of open strata of smaller dimensions, we deduce from the induction hypothesis and \eqref{eq:Hdg poly additive} that $H(\cW_{b}\setminus \cW^{\circ}_{b};u,v)$ is constant in $b$. By assumption, $\cW$ is smooth over $\cB$, hence $H(\cW_b;u,v)$ is also constant in $b$. Thus another application of \eqref{eq:Hdg poly additive} implies that $H(\cW^\circ_b;u,v)$ is constant in $b$. This completes the proof.
\end{proof}

\begin{proof}[Proof of Theorem \ref{thm:bdd imply termination}]
First assume that $(X,\Delta)$ is klt. By the negativity lemma, we know that for any $i\in\bN$ we have $a(E,X_i,\Delta_i)\le a(E,X_{i+1},\Delta_{i+1})$ for all divisors $E$ over $X$, and the inequality is strict for some $E$. By Lemma \ref{lem:string Hodge decrease}, we see that 
\[
E_{\mathrm{st}}(X_i,\Delta_i;u,v) > E_{\mathrm{st}}(X_{i+1},\Delta_{i+1};u,v)
\]
when $u,v\gg 0$. It follows that 
\begin{equation} \label{eq:E(X_i) diff from E(X_j)}
    E_{\mathrm{st}}(X_i,\Delta_i;u,v) \neq E_{\mathrm{st}}(X_j,\Delta_j;u,v)
\end{equation}
whenever $i\neq j$. If  there is an $\bR$-Gorenstein family of pairs $(\cX,\cD)\to \cB$ that contains every $(X_i,\Delta_i)$ as a fiber, then after possibly stratifying $\cB$, we may assume that there is a crepant log resolution $\pi\colon (\cY,\cG)\to (\cX,\cD)$, a projective morphism $\overline{\cY}\to \cB$ with an open immersion $\cY\subseteq \overline{\cY}$ over $\cB$, and an $\bR$-divisor $\overline{\cG}$ on $\overline{\cY}$ such that   $\cG=\overline{\cG}|_{\cY}$, the sub-pair $(\overline{\cY},\overline{\cG})$ is log smooth, every stratum of $(\overline{\cY},\overline{\cG})$ is smooth over $\cB$, and $\cY$ is a union of some open strata of $(\overline{\cY},\overline{\cG})$. By Lemma \ref{lem:stringy constant}, we see that the stringy $E$-function $E_{\mathrm{st}}(\cY_b,\cG_b;u,v)$ ($b\in \cB$) is constant on each connected component of $\cB$, hence there are only finitely many possible stringy $E$-functions $E_{\mathrm{st}}(X_i,\Delta_i;u,v)$. It follows that the MMP must terminate in finitely many steps, otherwise we get a contradiction to \eqref{eq:E(X_i) diff from E(X_j)}.

In general, let $W_i$ be the reduced non-klt locus of $(X_i,\Delta_i)$. We claim that for any $i\in\bN$ and any sufficiently small $0<\varepsilon\ll 1$ (which may depend on $i$), we have
\begin{equation} \label{eq:perturb log discrep}
a(E,X_i,\Delta_i-\cI_{W_i}^\varepsilon) \le  a(E,X_{i+1},\Delta_{i+1}-\cI_{W_{i+1}}^\varepsilon)
\end{equation}
for all divisors $E$ over $X$ and the inequality is strict for some $E$. It is enough to prove this for prime divisors on some common log resolution $Y$ of $(X_i,\Delta_i+\cI_{W_i})$ and $(X_{i+1},\Delta_{i+1}+\cI_{W_{i+1}})$. Let $E$ be one such divisor. If the center of $E$ on $X_i$ is contained in the exceptional locus of the birational map $X_i\dashrightarrow X_{i+1}$ (there is at least one such divisor), then 
\[
a(E,X_i,\Delta_i)<a(E,X_{i+1},\Delta_{i+1})
\]
by the negativity lemma, hence the strict inequality in \eqref{eq:perturb log discrep} holds for sufficiently small $\varepsilon$. As there are only finitely many such divisors $E$ on $Y$, we can make a uniform choice of $\varepsilon$. If the center $\xi=c_{X_i}(E)$ of $E$ on $X_i$ is not contained in the exceptional locus, then the map $X_i\dashrightarrow X_{i+1}$ is an isomorphism at $\xi$ and identifies the non-klt locus $W_i$ with $W_{i+1}$ around $\xi$. It follows that $a(E,X_i,\Delta_i)=a(E,X_{i+1},\Delta_{i+1})$ and $\ord_E(\cI_{W_i}) = \ord_E(\cI_{W_{i+1}})$, hence \eqref{eq:perturb log discrep} also holds. This proves the claim.

Note that the non-klt locus $W_i$ is determined by the pair $(X_i,\Delta_i)$. If there exists an $\bR$-Gorenstein family $(\cX,\cD)\to \cB$ that contains every pair $(X_i,\Delta_i)$ as a fiber, then possibly after stratifying $\cB$, we may assume that the non-klt locus $\cW$ of $(\cX,\cD)$ restricts to $W_i$ on $X_i$. By passing to log resolutions and compactifications over $\cB$ as in the klt case and using Lemma \ref{lem:stringy constant}, we deduce that there are only finitely many possibilities for the functions $E_{\mathrm{st}}(X_i,\Delta_i-\cI_{W_i}^\varepsilon;u,v)$ (in the variables $\varepsilon,u,v$). On the other hand, by \eqref{eq:perturb log discrep} and Lemma \ref{lem:string Hodge decrease} we also have
\[
E_{\mathrm{st}}(X_i,\Delta_i-\cI_{W_i}^\varepsilon;u,v) > E_{\mathrm{st}}(X_{i+1},\Delta_{i+1}-\cI_{W_{i+1}}^\varepsilon;u,v)
\]
when $0<\varepsilon\ll 1$ and $u,v\gg 0$, hence
\[
E_{\mathrm{st}}(X_i,\Delta_i-\cI_{W_i}^\varepsilon;u,v) \neq E_{\mathrm{st}}(X_j,\Delta_j-\cI_j^\varepsilon;u,v)
\]
whenever $i\neq j$. Thus the MMP terminates in finitely many steps by the same argument as in the klt case.
\end{proof}

\begin{proof}[Proof of Corollary \ref{cor:no loop in MMP}]
If $(X',\Delta')\cong (X,\Delta)$, then by repeating the same sequence of $(K_X+\Delta)$-MMP steps, we get an infinite sequence of $(K_X+\Delta)$-MMP that only involves finitely many isomorphism classes of pairs. This contradicts Theorem \ref{thm:bdd imply termination}.
\end{proof}

\bibliography{ref}

\end{document}